\documentclass[final,onefignum,onetabnum]{siamart171218}
\usepackage{amssymb}
\usepackage{comment} 
\usepackage[utf8]{inputenc}

\title{Variational Asymptotic Preserving Scheme for the Vlasov-Poisson-Fokker-Planck System}
\author{Jose A. Carrillo\thanks{Mathematical Institute, University of Oxford, Oxford OX2 6GG, UK. carrillo@maths.ox.ac.uk}, Li Wang\thanks{School of Mathematics, University of Minnesota, Twin cities, MN 55455. wang8818@umn.edu}, Wuzhe Xu \thanks{School of Mathematics, University of Minnesota, Twin cities, MN 55455. xu000355@umn.edu} and Ming Yan \thanks{Department of Computational Mathematics, Science and Engineering and Department of Mathematics, Michigan State University, East Lansing, MI 48824. myan@msu.edu}}
\date{}

\usepackage{lipsum}
\usepackage{amsfonts}
\usepackage{graphicx}
\usepackage{epstopdf}
\usepackage{algorithmic}
\ifpdf
  \DeclareGraphicsExtensions{.eps,.pdf,.png,.jpg}
\else
  \DeclareGraphicsExtensions{.eps}
\fi

\usepackage{amsmath,amsfonts,amssymb}
\newcommand{\eps}{\varepsilon}
\newcommand{\half}{\frac{1}{2}}
\newcommand{\RR}{\mathbb{R}}

\newcommand{\rd}{\mathrm{d}}

\newcommand{\W}{\mathcal{W}}

\newcommand{\A}{\mathsf{A}}
\newcommand{\B}{\mathsf{B}}
\newcommand{\CC}{\mathsf{C}}
\newcommand{\M}{\mathsf{M}}
\newcommand{\D}{\mathsf{D}}
\newcommand{\PP}{\mathsf{P}}
\newcommand{\UU}{\mathsf{U}}
\newcommand{\I}{\mathsf{I}}
\newcommand{\HH}{\mathsf{H}}

\newcommand{\hHH}{\hat{\mathsf{H}}}

\newcommand{\prox}{\operatorname{prox}}
\newcommand{\bj}{\boldsymbol{j}}
\newcommand{\bv}{\boldsymbol{v}}
\newcommand{\bvj}{\boldsymbol{v_j}}

\newcommand{\argmin}{\operatornamewithlimits{argmin}}

\newsiamremark{remark}{Remark}
\newsiamremark{hypothesis}{Hypothesis}
\crefname{hypothesis}{Hypothesis}{Hypotheses}
\newsiamthm{claim}{Claim}

\headers{Variational AP Scheme for VPFP system}{JOSE A.CARRILLO, LI WANG, WUZHE XU AND MING YAN}

\title{Variational Asymptotic Preserving Scheme for Vlasov-Poisson-Fokker-Planck system}

\author{Jose A. Carrillo\thanks{Mathematical Institute, University of Oxford, Oxford OX2 6GG, UK. carrillo@maths.ox.ac.uk} \and Li Wang\thanks{School of Mathematics, University of Minnesota, Twin cities, MN 55455. wang8818@umn.edu} \and Wuzhe Xu \thanks{School of Mathematics, University of Minnesota, Twin cities, MN 55455. xu000355@umn.edu} \and Ming Yan \thanks{Department of Computational Mathematics, Science and Engineering and Department of Mathematics, Michigan State University, East Lansing, MI 48824. myan@msu.edu}}

\usepackage{amsopn}

\makeatletter
\newcommand*{\addFileDependency}[1]{
  \typeout{(#1)}
  \@addtofilelist{#1}
  \IfFileExists{#1}{}{\typeout{No file #1.}}
}
\makeatother

\begin{document}

\maketitle

\begin{abstract}
We design a variational asymptotic preserving scheme for the Vlasov-Poisson-Fokker-Planck system with the high field scaling, which describes the Brownian motion of a large system of particles in a surrounding bath. Our scheme builds on an implicit-explicit framework, wherein the stiff terms coming from the collision and field effects are solved implicitly while the convection terms are solved explicitly. To treat the implicit part, we propose a variational approach by viewing it as a Wasserstein gradient flow of the relative entropy and solve it via a proximal quasi-Newton method. In so doing, we get positivity and asymptotic preservation for free. The method is also massively parallelizable and thus suitable for high dimensional problems. We further show that the convergence of our implicit solver is uniform across different scales. A suite of numerical examples are presented at the end to validate the performance of the proposed scheme. 
\end{abstract}

\section{Introduction}
The kinetic description of a gas of charged particles interacting through a mean electrostatic field created by their spatial distribution can be described by the Vlasov-Poisson-Fokker-Planck (VPFP) system: 
\begin{subequations}\label{vpfp00}
\begin{align}
    &{\partial_{t} f+v \cdot \nabla_{x} f-\frac{q}{m_e} \nabla_{x} \phi \cdot \nabla_{v} f=\frac{1}{\tau_e}}  \nabla_{v} \cdot\left(v f+ \mu_e \nabla_{v} f \right),  \\
    &{-\triangle_{x} \phi = \frac{q}{\epsilon_0}\left( \rho-h\right) }. \label{vpfp002}
\end{align}
\end{subequations}
Here $f(t,x,v)$ is the distribution function of particles at $t \in \mathbb{R}_{+}$, position $x \in \mathbb{R}^d$, and with velocity $v \in \mathbb{R}^d$. $\rho(t,x)$ is the density of electrons
\begin{equation}
   \rho(t,x)=\int_{\RR^d} f(t,x,v)  \rd v ,
\end{equation}
and $\phi(t,x)$ is the potential of electrostatic field obtained self consistently through the Poisson equation \eqref{vpfp002}. $h(x)$ is the density of positive background charges that satisfies global neutrality relation:
\[
\int_{\mathbb{R}^{d}} \int_{\mathbb{R}^{d}} f(0,x, v) d x d v=\int_{\mathbb{R}^{d}} h(x) \rd x .
\]
The constants $q$, $m_e$, $\epsilon_0$ and $\tau_e$ represent the elementary charge, electron mass, vacuum permittivity, and relaxation time, respectively. $\sqrt{\mu_e} = \sqrt{\frac{k_B T_{th}}{m_e}}$ is the thermal velocity, with $k_B$ being the Boltzmann constant and $T_{th}$ the temperature of the bath. The Fokker-Planck term on the right hand side of \eqref{vpfp00} represents the interaction of particles with background as a thermal bath. 

There has been a vast literature on the analytical aspect of the VPFP system. Existence and uniqueness results have been obtained in several frameworks: the existence of classical solutions was obtained by Victory and O'Dwyer in \cite{VHO90} locally in time and Weckler and Rein \cite{rein92} globally in time. Bouchut \cite{Bouchut91, Bouchut95} also gave an existence and uniqueness result in three dimensions for strong and global in time solution. In the more general setting of weak solutions, Carrillo and Soler allowed initial data in $L^p$ space \cite{CS95} and Morrey space \cite{CS97} and proved the existence of locally in time weak solution. Zheng and Majda obtained the existence of global measure solutions in one dimension \cite{ZM94}. 
The investigation of the quantitative properties of this system, especially its long time behavior, has also been adequate. Among them, we refer to the paper by Bouchut and Dolbeault \cite{bouchut1995long} and reference therein for the strong convergence to the unique stationary solution of the Cauchy problem via compactness argument, the one by Carrillo, Soler and Vazquez \cite{CSV96} on the asymptotic behavior of the frictionless case by similarity argument, and the one by Bonilla, Carrillo, and Soler \cite{BCS97} for the initial boundary value problem.

To study the physical behavior of the VPFP system, two important quantities are considered. One is the mean free path $l_e = \sqrt{\mu_e} \tau_e$, which is the average distance traveled by a particle between two successive collisions, and the other is the Debye length $\Lambda = \sqrt{\frac{\epsilon_0 k_B T_{th}}{q^2 \mathcal N}}$, where $\mathcal N$ denotes the concentration of the particles. When the mean free path of the electrons is much smaller than the Debye length, \eqref{vpfp00} can be rewritten in the following dimensionless form:
\begin{subequations} \label{vpfp}
\begin{align}
    &{\partial_{t} f+v \cdot \nabla_{x} f-\frac{1}{\eps} \nabla_{x} \phi \cdot \nabla_{v} f=\frac{1}{\eps}}  \nabla_{v} \cdot\left(v f+ \nabla_{v} f \right), \\
    &{-\triangle_{x} \phi =\rho -h}\,,
\end{align}
\end{subequations}
where $\eps = (\frac{l_e}{\Lambda})^2$. See \cite{ACGS99} for more details about the asymptotic limits. Sending $\eps \rightarrow 0$, we arrive at the so-called high field limit  
\begin{align}
&\partial_{t} \rho-\nabla_{x} \cdot\left(\rho \nabla_{x} \phi \right)=0, \label{rho_limit} \\
&{-\triangle_{x} \phi =\rho -h}\,, \nonumber
\end{align}
which is a nonlinear convection equation for mass density $\rho$. Indeed, one can first integrate equation (1.1a) w.r.t $v$ to get 
\begin{equation} \label{conser}
    \partial_t \rho + \nabla_x \cdot J = 0,
\end{equation}
where $\displaystyle J = \int_{\RR^d} v f(t,x,v) \rd v$. Then multiplying (1.1a) by $v$ and integrating w.r.t $v$, one obtains
\begin{equation} \label{contin}
    \eps (\partial_t J + \nabla_x \cdot Q) + \rho \nabla_x \phi +J = 0,
\end{equation}
where $\displaystyle Q = \int_{\RR^d} v \otimes v f(t,x,v) \rd v $. In the limit of $\eps \rightarrow 0 $, \eqref{contin} leads to 
$J = - \rho \nabla_x \phi $. Then \eqref{rho_limit} comes from plugging the above relation into \eqref{conser}. See $\cite{CGJS00, nieto2001high, goudon2005multidimensional, Poupaud92}$ for a physical and rigorous derivation of this limit, as well as the well-posedness of the limiting system.  

Numerically solving the VPFP system \eqref{vpfp00} shares the same difficulty as most of the kinetic equation: high dimensionality. Several methods have been developed, such as \cite{HV96, HV98, WO05, DA17}, to name a few. These methods are either deterministic or stochastic, with an effort in capturing some physical phenomena associated with the Vlasov-Poisson system such as Landau damping when the diffusion effect is rather weak. However, in the high field scaling we consider here, additional challenge comes from the stiffness of the field and collision terms, which generally calls for a resolved spatial and temporal discretization that can be very expensive.
Asymptotic preserving method, which aims at treating the stiff system and preserving its corresponding asymptotic limit at the discrete level, provides a unified solver to mutiscale problems. See \cite{Jin10, HJL17} for a review. 
In the specific context of VPFP system with high field scaling, we mention two particular methods. One is developed by Jin and Wang \cite{jin2011asymptotic} based on an implicit-explicit (IMEX) time discretization with a finite difference method in space and velocity, and the other is a quadrature-based moment closure method by Cheng and Rossmanith \cite{CR14}.

In this paper, we intend to design a new asymptotic preserving method for VPFP system in the same vein as \cite{jin2011asymptotic} but with marked difference. In particular, similar to \cite{jin2011asymptotic}, we group the stiff field and collision terms into one spatially dependent Fokker-Planck type operator and solve it implicitly, while treating the rest non-stiff terms explicitly. However, unlike the direct iterative solver (e.g., conjugate gradient or GMRES) employed in \cite{jin2011asymptotic} for the implicit part, here we propose a variational approach. This is hinted by the fact that the stiff term can be viewed as a Wasserstein gradient flow of the relative entropy with respect to the local Maxwellian, and therefore can be solved with the Jordan-Kinderlehrer-Otto (JKO) scheme \cite{jordan1998variational}. It then remains to solve the resulting optimization problem, for which we propose a proximal quasi-Newton method. The reason is that, when $\eps$ is small or the magnitude of $f$ varies significantly, the gradient type optimization methods experience a deteriorative convergence. Therefore, we design a pre-conditioner that uses partial second order information. As a result, our method is not only {\it asymptotic preserving} in the sense that we allow for unresolved spatial, temporal, and velocity discretization to capture the correct high field limit, but also the resulting implicit system solver enjoys a {\it uniform convergence}. This is an important issue that has not been emphasized in the literature. We also point out that, the variational formulation together with the JKO scheme, offers a natural implicit treatment for the collision term that also mimics the real physical process (i.e., entropy decrease), and therefore provides a desirable addition to the current family of asymptotic preserving (AP) schemes for kinetic equation. Moreover, its parallelizability makes it very appealing for high dimensional problems.

The rest of the paper is organized as follows. In the next section, we recall the implicit-explicit treatment for $\eqref{vpfp}$, which can be split into three steps: an explicit convection step, a Poisson solver, and an implicit collision step. We then put emphasize on the implicit collision solver by first introducing the variational formulation and then proposing the corresponding Newton-type optimization solver.  In section 3, we examine the properties of the proposed method, including positivity, asymptotic preservation and uniform convergence. Section 4 is devoted to numerous numerical examples, which validates the efficiency of our method as well as the aforementioned properties. The paper is concluded in Section 5.

\section{Numerical method}
In this section, we provide a detailed derivation of our numerical scheme, including temporal and spatial discretization, along with the optimization algorithm for inverting the implicit algebraic system. Throughout the paper, we consider one dimension in space and $d$-dimension ($d = 1, 2, 3$) in velocity. We also restrict ourselves to periodic boundary contion in $x$, and Fourier spectral method is adopted in solving the Poisson equation. As will be explained below, the spatial and velocity treatments are decoupled in the Vlasov-Fokker-Planck equation, therefore extending to higher dimension in space is straightforward and will not introduce substantial additional computational cost if the algorithm is parallelized. 

To be more specific, let $\Omega_x = [-L_x,L_x]$ be the spatial domain, and we partition it into $N_x$ uniform cells with $\Delta x = \frac{2 L_x}{N_x}$, and denote each grid point by $x_i = -L_x+i \Delta x$, $1 \leq i \leq N_x$. Likewise, we denote $\Omega_v = [-L_v, L_v]^d$ as the velocity domain, and evenly partition it into $N_v$ pieces in each dimension with $\Delta v = \frac{2L_v}{N_v}$. Then the velocity grid point is denoted as $v_{j_k} = -L_v +(j_k-\half) \Delta v$ with $1\leq j_k \leq N_v$ , $1\leq k \leq d$. Let $\tau$ be the time step, then $t^n = n \tau$, $n \geq 0$. Hence $f^n_{i, \bj}$ represents the approximation of $f (t_n, x_i, \bvj)$, where $\bj = \{j_1, \cdots, j_d\}$. We always use zero flux boundary condition in $v$, i.e.,  $( (v+\nabla_x \phi)f+\nabla_v f) \cdot \nu =0$, where $\nu$ is the outer normal direction for $\Omega_v$.

\subsection{Implicit-Explicit scheme}
As is done in~\cite{jin2011asymptotic}, we group the stiff terms in~\eqref{vpfp} into one spatially dependent Fokker-Planck type operator, treat it implicitly, and solve the rest non-stiff parts explicitly. Therefore, we have the following semi-discrete scheme: 
\begin{subequations} 
\begin{align*}
    &{\frac{f^{n+1}- f^n}{\tau}+v \cdot \nabla_{x} f^n = \frac{1}{\eps}}  \nabla_{v} \cdot\left((v+ \nabla_x \phi ) f+ \nabla_{v} f \right)^{n+1},  \\
    &{-\triangle_{x} \phi^{n+1}=\rho^{n+1} -h}.
\end{align*}
\end{subequations}
To implement, we note that the above semi-discretization scheme can be split into three steps without introducing the splitting error. 

\begin{description}
  \item[Step 1: Explicit transport step] 
\end{description}
 We first get an intermediate stage $f^*$ from the transport step $f^* = f^n - \tau v \cdot \nabla_x f^n $, 
where the spatial discretization is conducted via the MUSCL scheme:
\begin{equation*}
f^*_{i,\bj}=f^n_{i,\bj}  + \frac{\tau}{\Delta x } v_{\bj}( f_{i+\frac{1}{2}, \bj}-f_{i-\frac{1}{2}, \bj} )\,.
\end{equation*}
Here the flux is taken as 
\begin{align*}
f_{i+\frac{1}{2}, \bj} = & \max(v_{\bj},0) \left( f_{i, \bj}+\frac{1}{2} \phi\left(\theta_{i+\frac{1}{2},\bj}\right)\left(f_{i+1, \bj}-f_{i, \bj}\right) \right)
\\ &  + \min(v_{\bj},0)\left( f_{i+1, \bj}+\frac{1}{2} \phi\left(\theta_{i+\frac{1}{2},\bj}\right)\left(f_{i+1, \bj}-f_{i, \bj}\right) \right) \,,
\end{align*}
where $\theta_{i+\frac{1}{2},\bj} = \frac{f_{i, \bj}-f_{i-1, \bj}}{f_{i+1, \bj}-f_{i, \bj}}$ is the smoothness indicator function, and we choose the minmod slope limiter $\phi(\theta)=\max \{0, \min \{1, \theta\}\}$.

\begin{description}
  \item[Step 2: Poisson Step]
 \end{description}
After obtaining $f^*$, we get $\rho^*$ by integrating $f^*$ over $\bv$, for which we can use a simple midpoint rule: $\rho_i^* = \sum_{\bj}f^*_{i,\bj} (\Delta v)^d$. We then solve for $\phi^*$ via the Fourier based spectral method, and then get $\nabla_x \phi^*$.

\begin{description}
  \item[Step 3: Implicit collision step]
\end{description}
First note that the mass is conserved in the collision step, thus $\rho^{n+1} = \rho^*$ and $\phi^{n+1} = \phi^*$. Then for each $x_i$, we have
\begin{equation} \label{eqn315}
\frac{f^{n+1}_i-f^*_i}{\tau} = \frac{1}{\eps} \nabla_{v} \cdot ((v+\nabla_x \phi^*_i) f_i^{n+1}+\nabla_v f_i^{n+1}) \,,
\end{equation}
which will be solved by the variational method described blow.

\subsection{Variational formulation}
This section is devoted to the development of a variational numerical scheme for the implicit collision step. First, we would like to mention that there exist quite a few methods for discretizing the Fokker-Planck operator, such as Chang-Cooper scheme \cite{buet2010chang}, Scharfetter-Gummel discretization \cite{schlichting2020scharfetter}, and squareroot approximation \cite{heida2018convergences}. Among them, some are known to preserve positivity and dissipate entropy, two properties for the continuum equation that are desirable to be preserved at the discrete level. Here we intend to provide a different approach to address the stiffness issue. In particular, when $\varepsilon$ is small, a generic time implicit scheme would lead to a linear system that is ill-conditioned. Our variational scheme induces a natural way of building preconditioners  arising from an optimization method and efficiently resolves the ill-condition issue.  As a result, our method not only enjoys positivity preserving and entropy dissipating, but also is asymptotic preserving and uniformly efficient. Other smart preconditioners can also be devised for classical methods to avoid stiffness.

Let
\begin{equation} \label{eqn:Mstar}
M_i^* = \frac{\rho_i^*}{(\sqrt{2\pi})^d} e^{-\frac{|v + (\nabla_x \phi)^*_i|^2}{2}}
\end{equation}
be the local Maxwellian, then \eqref{eqn315} can be rewritten as 
\begin{equation} \label{eqn316}
\frac{f_i^{n+1} - f_i^*}{\tau} = \frac{1}{\eps} \nabla_v \cdot \left(f_i^{n+1} \nabla_v \ln \left(\frac{f_i^{n+1}}{M_i^*}\right)\right) = \frac{1}{\eps} \nabla_v \cdot \left(f_i^{n+1} \nabla_v \frac{\delta E(f_i^{n+1}|M^*)}{ \delta f_i^{n+1}}\right)\,,
\end{equation}
where $E(f|M) = \int_{\RR^d} f \ln (\frac{f}{M}) \rd v$ is the relative entropy of $f$ with respect to $M$, and $\frac{\delta E}{\delta f}$ denotes the first variation of $E$ in $f$. In view of \eqref{eqn316}, it can be considered as the gradient flow of the relative entropy in the Wasserstein metric, i.e., 
\[
\frac{f_i^{n+1} - f_i^*}{\tau} = - \frac{1}{\eps}\nabla_{d_{\W}} E(f_i^{n+1}|M_i^*)\,,
\] 
which can be solved via the celebrated JKO scheme \cite{jordan1998variational}. That is, $f_i^{n+1}$ is obtained to minimize the following functional 
\begin{equation} \label{eqn317}
f_i^{n+1} \in \argmin_{f_i \in \mathcal P_{ac} (\Omega_v)} ~ \left\{ \half  d_{\W}( f_i, f^*_i)^{2} + \frac{\tau}{\eps} E(f_i|M^*_i)\right\},
\end{equation}
where $d_{\W}( f_i, f^*_i)$ is the Wasserstein distance between $f_i$ and $f^*_i$, and $\mathcal P_{ac} (\Omega_v)$ is the set of probability measures on $\Omega_v$ that are absolutely continuous with respect to Lebesgue measure.
The formulation~\eqref{eqn317} has attracted a lot of attention on the analytical level as it provides a natural choice for $f^{n+1}$ that decreases the relative entropy, i.e., $E(f^{n+1}_i |M^*_i) \leq E(f^{n}_i |M^*_i) $. However, when it comes to numerical implementation, the computation of the Wasserstein distance constitutes a major obstacle. Only recent advances in this regard have helped to make this formulation numerically accessible, see \cite{Peyrebook19} and reference therein. In this paper, we will adopt the dynamic formulation by Benamou and Brenier \cite{benamou2000computational} and its fully discrete version \cite{carrillo2019primal}. In particular, we can reframe the Wasserstein distance into a convex optimization subject to linear constraints:
\begin{equation}\label{wassdic}
d_{\W}( f_0, f_1)^{2}=\min _{(f, m) \in \mathcal{C}_{1}} \int_{0}^{1} \int_{\Omega_v} \Phi(f( t,v), \|m(t,v)\|) \mathrm{d} v \mathrm{d} t ,
\end{equation}
where
\begin{equation}
\Phi(f, \|m\|)=\left\{\begin{array}{cl}{\frac{\|m\|^{2}}{f}} & {\text { if } f>0,} \\ {0} & {\text { if }(f, m)=(0,0),} \\ {+\infty} & {\text { otherwise }.}\end{array}\right. \nonumber
\end{equation}
and the constraint set $\mathcal C_1$ consists of 
\begin{align*}
&\partial_{t} f+\nabla_v \cdot m =0 \text { on } \Omega_v \times[0,1], \quad  m \cdot \nu =0 \text { on } \partial \Omega_v \times[0,1],  \\
& f(\cdot, 0)=f_{0}, \quad f(\cdot, 1) =f_1 \text { on } \Omega_v\,,
\end{align*}
where $\nu$ is the outer normal direction of $\Omega_v$. 

Plugging \eqref{wassdic} into \eqref{eqn317}, we arrive at the following constrained optimization problem: given $M^*(v)$ and $f^*(v)$, one obtains $f^{n+1}(v) = f(1, v)$ with $f(t,v)$ solving
\begin{equation} \label{dynJKO}
 \min_{(f, m) \in \mathcal C} \left\{\eps \int_0^1  \int_{\Omega_v} \Phi(f, \| m \|) \rd v \rd t +2 \tau  E(f(1,v)|M(v)) \right\}, 
 \end{equation}
where the constraint set $\mathcal C$ is
\begin{equation} \label{constraints}
 \partial_{t} f+\nabla_v \cdot m =0 \text { on } \Omega_v \times[0,1],~ m \cdot \nu =0 \text { on } \partial \Omega_v \times[0,1], ~f(0,v)  =f^*(v) \text { on } \Omega_v .
\end{equation}
Here the subscript $i$ is omitted as this step is independent of $x$. Note the difference between constraints $\mathcal C$ and $\mathcal C_1$ is that in $\mathcal C_1$, we do not know $f(1,v)$ a priori, and it is in fact coming from solving the optimization \eqref{dynJKO}, which is similar to an optimal control problem. 

We further write down the fully discrete form for \eqref{dynJKO} and \eqref{constraints}. Denote $f=[f_{\bj}]^{\intercal} \in \mathbb{R}^{d N_v}$ and $m = [m_1; \cdots; m_d] \in \mathbb{R}^{d N_v \times d } $, where $m_l = [m_{l,\bj}]^{\intercal} \in \mathbb{R}^{d N_v}$. Then $\|m\|_{\bj}^2 = \sum_{l=1}^d m_{l,\bj}^2$. The fully discrete JKO scheme now writes: 
\begin{align}
f^{n+1}_{\bj} \in & \arg \min_{f, m} \left\{ \sum_{\bj}  \left( \eps \Phi(f_{\bj},\|m\|_{\bj}) + 2 \tau f_{\bj} \ln\left(\frac{f_{\bj}}{M^*_{\bj}}\right) \right)  \Delta v^d \right\},  \label{disc1}
\\  & \text { s.t. } \quad  f_{\bj}-f^*_{\bj} + \sum_{l=1}^d \D_{\bv, l} m_{l, \bj} = 0, \qquad m_{l, \bj} \cdot \nu |_{\partial \Omega}  =0 \,,  \label{disc2}
\end{align}
where $\D_{\bv, l}$ is a discrete representation of the divergence that will be detailed later. 
Note that the PDE constraint in \eqref{constraints} is discretized only in time step, which has been pointed out in \cite{li2020fisher} (Theorem 3) that it will significantly reduce the dimension of the problem while maintaining the first order accuracy in $\tau$. Indeed, if we discretize the auxiliary inner time derivative in \eqref{constraints} with $N_t$ nodes, then the unknown $f$ would be of size $dN_v \times N_t$, and $m$ is of size $dN_v \times N_t \times d$. Here we choose $N_t=1$ and thus keeps the size of $f$ and $m$ to a minimum. 

To facilitate the explanation later, we let 
$
u = [f;m]
$
and rewrite \eqref{disc1}-\eqref{disc2} into 
\begin{align}
\min_{u}~ F(u):=
\sum_{\bj} \left( \eps \Phi(f_{\bj}, \|m\|_{\bj}) + 2\tau f_{\bj} \ln{\left(\frac{f_{\bj}}{M^*_{\bj}}\right)} \right) \Delta v^d  \quad \text{s.t.} \quad \A u=b,  \label{primal}
\end{align}
where $\A:=\begin{pmatrix}
\I_{d{N_v} \times d{N_v}} & \A_m 
\end{pmatrix}$ and $b:=f^*$. Note that all the operations here are element-wise. Here $\A_m$ gives a discretized divergence $\D_{\bv,l} m_{l,\bj}$, which satisfies the zero flux boundary condition. For instance, we use the center difference here. Then in one dimension, the boundary grid points are $v_{\half} = -L_v$ and $v_{N_v+\half}= L_v$, and the boundary condition becomes $0=m_{\half} = \frac{m_0+m_1}{2}$, which implies $m_0 = -m_1$. Then $D_v m_1 = \frac{m_2-m_0}{2 \Delta v} = \frac{m_2+m_1}{2 \Delta v}$. Extension to higher dimension is straightforward. As an example, we give $\A$ for $d=3$. We denote 
\[
\A = 
\begin{pmatrix}
\I &\A_1 &\A_2 &\A_3
\end{pmatrix}.
\]
where $\A_1,~\A_2,~\A_3$ represent the discritizations of $D_{v_1},~D_{v_2}~,D_{v_3}$ respectively. Define
\[
\D_v = 
\frac{1}{2 \Delta v}
\begin{pmatrix}
1 & 1  \\
-1 &0 & 1  \\
 & \ddots & \ddots & \ddots\\
 & & -1 & 0 &1 \\
 & & &-1 & -1
\end{pmatrix}
\in \mathbb{R}^{N_v \times N_v} .
\]
Denote Kronecker tensor product as $\otimes$, then we have
\[
\A_1 = \I_{N_v^2} \otimes \D_v,\quad \A_2 = \I_{N_v} \otimes \left( \D_v \otimes \I_{N_v} \right), \quad \A_3 =  \D_v \otimes \I_{N_v^2} .
\]

\subsection{Proximal quasi-Newton method}
In this subsection, we introduce the proximal quasi-Newton type method. First of all, we rewrite $\eqref{primal}$ as an unconstrained problem by using the following indicator function: 
\[
\chi(u) = \left\{ 
\begin{array}{cl}
    0, & \mbox{ if } \A u =b, \\
    +\infty, & \text{otherwise.}
\end{array} \right.
\]
Then $\eqref{primal}$ becomes
\begin{equation}\label{twof}
\min_{u}  F(u) + \chi (u)\,,
\end{equation}
where $F(u)$ is defined in \eqref{primal}. As written, $F(u)$ is a convex but nonsmooth function of $u$, and therefore a proximal type of algorithm is needed, as stated in \cite{carrillo2019primal}. However, in our specific case considered here, a simplification can be made. In fact, as shown in \cite[ Theorem 5.1]{jordan1998variational}, the minimizer of \eqref{primal} converges to the unique {\it positive} solution of the equation
\begin{equation} \label{eqn320}
\partial_t f= \frac{1}{\eps} \nabla_v \cdot ((v+ \nabla_x \phi) f + \nabla_v f)
\end{equation}
when $\tau \rightarrow 0$. Moreover, 
thanks to Lemma 8.6 in $\cite{santambrogio2015optimal}$ and mass conservation, the strict positivity of $f_{\bj}^n$ can be established as long as initially $f^0$ is non-negative $f^0_{\bj}\geq 0$, and has strictly positive initial mass, i.e. $\sum_{\bj} f_{\bj}^0 >0$.
Therefore, we can simplify $F(u)$ as:
\[
F(u) = \sum_{\bj} \left( \eps \frac{\| m_{\bj} \|^2}{f_{\bj}} +2 \tau f_{\bj} \ln\left(\frac{f_{\bj}}{M_{\bj}}\right) \right) \Delta v^d \,,
\]
which is now a smooth function in $u$, and hence gives access to the second order information that could significantly accelerate the convergence. 

Below we first state our algorithm, and then we explain the reasons for this choice. Here the stepsize $\gamma>0$ is chosen such that $f^k \succ 0$ for every iteration.

\begin{algorithm}[H]
\caption{Proximal quasi-Newton algorithm for \eqref{eqn317}}\label{alg:PN}\begin{algorithmic}
\REQUIRE {$u^{(0)} = [f^*;m^*]$ with $m^*\equiv 0$, the maximum iteration number ($N_{\textrm{max}}$), step-size $\gamma > 0$}
\WHILE{$k \leq N_{\textrm{max}}$}
\REPEAT
\STATE {
 $ \text{1. Compute } (\HH^k)_{i,j} = \begin{cases}
      (\nabla^2 F(u^{(k)})_{i,i}, & \text{if}\ i=j ,\\
      0, & \text{otherwise}.
    \end{cases}$ }
\STATE {$\text{2. Update } u^{(k+1)} = \prox_{\chi}^{\HH^k} (u^{(k)}- \gamma (\HH^k)^{-1} \nabla F(u^{(k)}))$.} 
\UNTIL{$\text{stopping criteria achieved}$}
\STATE {$u^{(\infty)} = u^{(k+1)}$.}
\ENDWHILE
\RETURN {$u^{n+1} = u^{(\infty)}= [f^{(\infty)}; m^{(\infty)}]$.}
\end{algorithmic}\end{algorithm}
There are two reasons for choosing this algorithm. One reason is due to the appearance of $\eps$ and small values of $f^n_{i, \bj}$. When $\eps$ is small or the magnitude of $f^n_{i, \bj}$ varies largely, the convergence of any optimization algorithm that only uses first order information will converge very slowly. This is because the Hessian of $F(u)$ becomes ill-conditioned in these scenarios. Therefore, using the second order information as in our algorithm would significantly improve the convergence rate (See Theorem \ref{convThm} and Remarks \ref{remark1} $\&$ \ref{remark2}). More importantly, although the stiffness introduced by small $\eps$ has been handled by the implicit JKO scheme and therefore enjoys the AP property---it allows for under-resolved mesh sizes and captures the correct asymptotic limit, it still comes with another difficulty which renders a direct implicit solver converging {\it non-uniformly}. The proximal quasi-Newton method we proposed here overcomes this difficulty. Another reason is that, it is well-known that computing the Hessian is expensive and often results in a dense matrix, which poses additional computational cost especially when the dimension is high. Instead, we only use the diagonal information of the Hessian as a surrogate, which is shown to still serve the purpose of accelerating the convergence while maintaining the sparsity of the matrix. 

Next we show how to compute the scaled proximal operator $z=\operatorname{prox}^{\HH}_{ \chi}(u)$, which can be obtained from a closed-form formula in our specific case. First, the definition of the scaled proximal operator is:
\[
z =  \operatorname{prox}^{\HH}_{ \chi}(u) \in \argmin_z~\chi (z)+\frac{1}{2}\|z-u\|_{\HH}^{2}=\argmin_{z:\A z=b} \frac{1}{2}\|z-u\|_{\HH}^{2}.
\]
Its corresponding Lagrangian is:
$L(z,\lambda) = \frac{1}{2} \|z-u \|_{\HH}^2+  \lambda^{\intercal} (b-\A z).$
Then optimality condition gives 
$\frac{\partial L}{\partial z} =  \HH(z-u) - \A^{\intercal} \lambda = 0.$
Hence
\begin{equation}\label{opt}
   z=u + \HH^{-1} \A^{\intercal} \lambda .
\end{equation}
By the primal feasibility, i.e. $\A z = b$, and $\eqref{opt}$, we get

\begin{equation*}
\lambda  =  (\A \HH ^{-1} \A^{\intercal})^{-1}(b - \A u),
\end{equation*}
which gives the closed-form formula for $\operatorname{prox}_{ \chi}^{\HH}$:
\begin{equation}\label{proxchi}
z=\operatorname{prox}_{ \chi}^{\HH}(u)  = u +   \HH ^{-1} \A^{\intercal}(\A \HH ^{-1} \A^{\intercal})^{-1}(b - \A u)  .
\end{equation}
In practice, computing the inverse of a matrix can be expensive. In our formula \eqref{proxchi}, there are two inverse of matrix, $\HH^{-1}$ and $(\A \HH^{-1} \A^{\intercal})^{-1}$. Computing $\HH^{-1}$ is trivial as it is diagonal, whereas computing $(\A \HH^{-1} \A^{\intercal})^{-1}$ might be time-consuming. However, due to the special structure of $\A$ and $\HH$, there exists fast methods. Indeed, to clearly illustrate the idea, consider 1D case and the diagonal matrix $\HH$ of the form
\[
\HH = 
\begin{pmatrix}
\HH_1 & 0\\
0 & \HH_2
\end{pmatrix}.
\]
By definition of $\A = [\I~\D]$, we get $\A \HH^{-1} \A^{\intercal} = \HH_1^{-1} + \D \HH_2^{-1} \D^{\intercal}$. Note that $\HH_2$ is diagonal and hence $\D \HH_2 \D^{\intercal} $ is just a weighted Laplacian, which can be efficiently inverted by fast algorithms such as multigrid method, see for instance $\cite{capizzano2010note}$.

Alternately, instead of fixing the stepsize $\gamma$ in Algorithm 1, we can also use a line search technique, and the algorithm is summarized in Algorithm 2. 
\begin{algorithm}[H]
\caption{Proximal quasi-Newton algorithm with line search for \eqref{eqn317}}\label{alg:PNS}
\begin{algorithmic}
\REQUIRE {$u^{(0)} = [f^*;m^*]$ with $m^* \equiv 0$, $0<\theta < \half$, \text {the maximum iteration Number} ($N_{\textrm{max}}$)}
\STATE {Let $k=0$}
\WHILE{$k \leq N_{\textrm{max}}$}
\REPEAT
\STATE {
 $ \text{1. Compute } (\HH^k)_{i,j} = \begin{cases}
      (\nabla^2 F(u^{(k)}))_{i,i}, & \text{ if}\ i=j, \\
      0, & \text{otherwise}.
    \end{cases}$ }
\STATE {$\text{2. Line search: let } t^l =1,  v^k=  \prox_{\chi}^{\HH^k} (u^{(k)}- (\HH^k)^{-1} \nabla F(u^{(k)})) -u^{(k)}$.} 

\WHILE{$F(u^{(k)}+t^l  v^k) > F(u^{(k)}) +  t^l \theta (\nabla F(u^{(k)}))^{\intercal} v^k  \text{ and } \min_{j} f^{(k)}_j <0$}
\STATE {$t^l = \half t^l$.}
\ENDWHILE
\STATE {$u^{(k+1)} = u^{(k)}+t^l v^k $.}
\UNTIL{$\text{stopping criteria achieved}$}
\STATE {$u^{(\infty)} = u^{(k+1)}$.}
\ENDWHILE
\RETURN {$u^{n+1} = u^{(\infty)}= [f^{(\infty)}; m^{(\infty)}]$.}
\end{algorithmic}\end{algorithm}
The advantages of the line search are obvious. First, the search step automatically preserves the positivity of $f_j^{(k)}$. Second, it often needs less steps to converge, see the numerical examples in Section 4. In addition, Algorithm 2 falls into the category of proximal Newton-type methods in \cite{lee2014proximal}, for which it is proven that if $\{ \HH^k \}$ are uniformly positive definite, i.e. $s \I \preceq \HH^k $ uniformly for $ s>0$, then for a closed, convex objective function whose infimum can be attained, $\{ u^{(k)} \}$ generated by Algorithm 2 is guaranteed to converge to the optimal point. 
In our numerical examples, we observe that $t^l = 1$ after sufficiently many iterations, and therefore we see superlinear convergence at the neighborhood of the optimal point (see Fig.\ref{fig:PN_1D}).

\section{Properties}
In this section, we study some properties of the numerical scheme. We firstly focus on the convergence behavior of the Newton type method, and then examine the properties of the entire solver, including positivity and asymptotic preserving property.

\subsection{Convergence of the proximal Newton type method}
We mainly focus on the convergence behavior of Algorithm~\ref{alg:PN} in this subsection. We first examine the convexity of $F$ in the following lemma. 

\begin{lemma}
$F(u)$ is strictly convex, i.e. $\nabla^2 F(u) \succ 0$ if $f \succ 0$ for any $d = 1,2,3$.
\end{lemma}
\begin{proof}
We only prove the $d = 3$ case as the other two cases can be easily reduced to. Recall $u=[f;m]$, then the Hessian $\nabla^2 F \in \RR^{4N_v \times 4N_v}$ reads:
\[
\nabla^2 F:= 
\begin{pmatrix}
\M & \CC_1 & \CC_2 & \CC_3\\
\CC_1 & \B & 0 &0 \\
\CC_2 & 0 & \B & 0\\
\CC_3 & 0 & 0 & \B
\end{pmatrix},
\]
where $\M,\B,\CC_l \in \RR^{N_v \times N_v}$ are all diagonal matrices defined as:
\begin{align*}
(\M)_{p,q} &= 
\left\{
    \begin{array}{cl}
   (2\eps \frac{\sum_{l=1}^3 m_l^2 }{f^3} + \frac{2 \tau}{f} )_q \Delta v^3 , &  \text{ if $p=q$,} \\
    0, & \text{otherwise,}
\end{array} \right.\\
(\B)_{p,q} &= 
\left\{
    \begin{array}{cl}
   (\frac{2 \eps}{f})_q \Delta v^3 , &  \text{ if $p=q$,} \\
    0, & \text{otherwise,}
\end{array} \right .\\
(\CC_l)_{p,q} &= 
\left\{
    \begin{array}{cl}
   (-\frac{2 \eps m_l}{f^2})_q \Delta v^3, &  \text{ if $p=q$,} \\
    0, & \text{otherwise.}
\end{array} \right .
\end{align*}
To obtain the eigenvalues $\zeta$ of $\nabla^2 F$, note that each entry in $\nabla^2 F$ is a diagonal matrix, thus they are pairwise multiplication commutative, we then have:
\begin{align*}
    | \nabla^2 F - \zeta \I| &=
 - | \CC_3 |  \begin{vmatrix}
 \CC_1 &\CC_2 & \CC_3 \\
 \B-\zeta \I & 0  &0 \\
 0 & \B - \zeta \I  & 0\\
 \end{vmatrix}
 +
|\B - \zeta \I| \begin{vmatrix}
\M - \zeta \I & \CC_1  &\CC_2 \\
\CC_1 & \B-\zeta \I  &0 \\
\CC_2  & 0 & \B-\zeta \I  
\end{vmatrix}\\
 &= - |\CC_3| |(\B-\zeta \I)^2 \CC_3| + |\B-\zeta \I| |(\M - \zeta \I)(\B - \zeta \I)^2-(\B - \zeta \I)\CC_2^2 - (\B - \zeta \I)\CC_1^2  | \\
 &=|\B - \zeta \I|^2 
 |(\M - \zeta \I)(\B - \zeta \I) -\CC_1^2-\CC_2^2-\CC_3^2 |.
\end{align*}
After calculation, eigenvalues of $\nabla^2 F$ are:
\begin{align}
\zeta_{1,q} &=(\frac{2\eps}{f})_q \Delta v^3 \label{eig1} \,;\\
\zeta_{2,q}&= \left(
 \frac{ \eps \sum_{l=1}^3 m_l^2}{f^3}+\frac{\tau + \eps}{f}+
\sqrt{(\frac{ \eps \sum_{l=1}^3 m_l^2}{f^3})^2 + \frac{2\eps \sum_{l=1}^3 m_l^2 (\tau+\eps)}{f^4}+(\frac{\tau-\eps}{f})^2}
\right)_q \Delta v^3 \label{eig2} \,;  \\
\zeta_{3,q}&=  
\left( \frac{ \eps \sum_{l=1}^3 m_l^2}{f^3}+\frac{\tau + \eps}{f}-
\sqrt{(\frac{ \eps \sum_{l=1}^3 m_l^2}{f^3})^2 + \frac{2\eps \sum_{l=1}^3 m_l^2 (\tau+\eps)}{f^4}+(\frac{\tau-\eps}{f})^2}
\right)_q \Delta v^3 \,. \label{eig3} 
\end{align}
which can be easily shown to be positive given $f>0$. 
\end{proof}
This lemma ensures that $\nabla^2 F(u^{(k)})$ are positive definite provided $f^{(k)}>0$. Moreover, we can easily see that $\HH^k$, which only keeps the diagonal elements of $\nabla^2 F(u^{(k)})$ are positive definite as well, and therefore guarantees the executability of our algorithm. 

Similar to the result in \cite{li2020fisher}, we have the following local convergence estimate, which indicates the role of $\HH^k$. 
\begin{theorem} \label{convThm}
Denote $u^*$ the unique minimizer of \eqref{twof}. Let $G^k = \int_0^1 \nabla^2 F(u^* + s(u^{(k)}-u^*)) \rd s$ and
suppose that there exists $0<\alpha< \beta$, such that 
$\alpha \I \preceq (\HH^k)^{-1} G^k \preceq \beta \I$, then for Algorithm \ref{alg:PN} we have $\| u^{(k+1)} - u^*  \|_{\HH^k} \leq C \| u^{(k)} - u^*  \|_{\HH^k}$, where $C = \max (|1-\gamma \alpha|,|\gamma \beta -1|) $. In particular, if we choose $\gamma = \frac{2}{\alpha + \beta}$, we have the optimal convergence rate with $C = \frac{\beta - \alpha}{\beta + \alpha}$. 
\end{theorem}
\begin{proof}
First notice that $u^* = \prox_{\chi}^{\HH}(u^*-\gamma {\HH}^{-1} \nabla F(u^*))$ for $\HH=\HH^k$, then we have
\begin{align*}
    &\| u^{(k+1)} - u^*  \|_{\HH^k} \\
    & =  \| \prox_{\gamma \chi}^{\HH^k} (u^{(k)}-\gamma (\HH^k)^{-1} \nabla F(u^{(k)}))- \prox_{\gamma \chi}^{\HH^k} (u^*-\gamma (\HH^k)^{-1} \nabla F(u^*)) \|_{\HH^k} \\
    & \leq \| (u^{(k)}-u^*) - \gamma (\HH^k)^{-1} \nabla (F(u^{(k)})-F(u^*))\|_{\HH^k} \\
    & = \| (\I - \gamma (\HH^k)^{-1} G^k) (u^{(k)}-u^*)  \|_{\HH^k}  \\
    &= \| (\HH^k)^\half (\I - \gamma (\HH^k)^{-1} G^k) (u^{(k)} - u^*)  \| \\
    & = \| (\I - \gamma (\HH^k)^{-\half} G^k(\HH^k)^{-\half}) (\HH^k)^{\half}(u^{(k)}-u^*)  \|  \\
    & \leq \|\I - \gamma (\HH^k)^{-\half} G^k(\HH^k)^{-\half} \| \| u^{(k)}-u^* \|_{\HH^k}.
\end{align*}
Here the first inequality uses the fact that $\prox_{\gamma \chi}^{\HH^k}$ is a nonexpansive operator under $\HH^k$ norm, i.e., $\|\prox_{\gamma \chi}^{\HH^k}(u) - \prox_{\gamma \chi}^{\HH^k}(v)\|_{\HH^k} \leq \|u-v\|_{\HH^k}$ (see for instance \cite{lee2014proximal} for a proof) and the second equation uses the fact that $F(u^{(k)})-F(u^*) = G^k(u^{(k)} - u^*)$. 
Now since $ (\HH^k)^{-\half} G^k(\HH^k)^{-\half}$ is similar to $(\HH^k)^{-1} G^k$, we have $ C = \|\I - \gamma (\HH^k)^{-\half} G^k(\HH^k)^{-\half} \| = \max (|1-\gamma \alpha|,|\gamma \beta -1|)$. 
\end{proof}

\begin{remark} \label{remark1}
Adapting Theorem~\ref{convThm} to our case, we know that the convergence rate highly depends on the structure of $(\HH^k)^{-1} G^k$. Note that when $\HH^k = \I$, our method reduces to the projected gradient method. According to \eqref{eig1}, \eqref{eig2},  \eqref{eig3}, the eigenvalue $\zeta_{1,q} \rightarrow 0$ when $\eps \rightarrow 0$, which implies that $\alpha$ in the above theorem goes to zero, hence $C \rightarrow 1$. This explains why the gradient type methods converge slowly when $\eps$ is close to 0. On the contrary, with the preconditioner $(\HH^k)^{-1}$, the diagonal entries of $(\HH^k)^{-1} G^k$ all approximately equal to 1 in the neighborhood of optimal point $u^*$, thus we can approximate the eigenvalues $\sigma$ of $(  \HH^k)^{-1} G^k$ by solving (here we adopt the notations in Lemma 1):
\[
\begin{vmatrix}
\I - \sigma \I & \M^{-1} \CC_1 & \M^{-1} \CC_2 &\M^{-1} \CC_3 \\
\B^{-1} \CC_1 & \I-\sigma \I & 0 &0 \\
\B^{-1} \CC_2 & 0 & \I-\sigma \I & 0\\
\B^{-1} \CC_3 & 0 & 0 & \I -\sigma \I
\end{vmatrix} = 0.
\] 
This implies
\[
(1-\sigma)^2 \I \cdot |(1-\sigma)^2 \I -\M^{-1} \B^{-1}(\CC^2_1+\CC^2_2+\CC^2_3)|=0.
\]
Thus
\begin{equation}\label{evalue000}
\sigma = 1 \text{ or } 1 \pm \left( \sqrt{  \frac{2 \eps  \sum_{l=1}^3 m_l^2}{2 \eps \sum_{l=1}^3 m_l^2+2 \tau f^2}} \right)_i \Delta v^3 \,,
\end{equation}
which indicates that $\sigma \rightarrow 1$ when $\eps \rightarrow 0$. Therefore the condition number is close to 1 and hence gives much faster linear convergence ($C \ll 1$). 
\end{remark}
\begin{remark} \label{remark2}
Another case that may render the condition number of $G^k$ big is when the magnitude of $f$ varies largely. We will see in the following that the preconditioner $\HH^k$ also helps in this case. Let $j_1, j_2$ be two indexes such that $f_{j_1} =O(\eta)$ where $\eta \ll 1$ and $f_{j_2} = O(1)$. 
In the gradient type method when ${\HH^k}= \mathsf I$, according to $\eqref{eig1}$, we have $ \max \zeta \geq \frac{2 \eps}{f_{min}}$ and $\min \zeta \leq \frac{2 \eps}{f_{max}}$, which implies that the conditional number $\kappa$ of $G^k$ has $\kappa \geq \frac{f_{max}}{f_{min}} = O(\frac{1}{\eta})$. Immediately according to Theorem \ref{convThm}, the convergence rate $C \rightarrow 1$ when $\kappa \rightarrow \infty$.

On the other hand, from the expression of the eigenvalues \eqref{evalue000} for $(\HH^k)^{-1} G^k$, one sees 
\[
\sqrt{  \frac{2 \eps  \sum_{l=1}^3 m_l^2}{2 \eps \sum_{l=1}^3 m_l^2+2 \tau f^2}}  = \sqrt{\frac{2 \eps \frac{\sum_{l=1}^3 m_l^2}{f^2}}{2 \eps \frac{\sum_{l=1}^3 m_l^2}{f^2}+2 \tau}} = \sqrt{\frac{ \eps \left( \frac{\|m\|}{f}\right)^2 }{ \eps \left( \frac{\|m\|}{f}\right)^2 + \tau}} \,.
\]
Therefore, if the speed $\frac{\|m\|}{f}$ to the continuity equation \eqref{constraints} is bounded above by $C_1$ and below by $C_2$, then the above quantity is bounded between $\frac{\eps C_2}{\eps C_2 + \tau}$ and $\frac{\eps C_1}{\eps C_1 + \tau}$, which readily gives a uniform bound on the condition number. Here we do not have a rigorous proof to show the existence of $C_1$ and $C_2$, but from the numerical examples, we do observe a uniform bound on $\frac{\|m\|}{f}$.  
\end{remark}
Similarly, we have the local convergence estimate for Algorithm~\ref{alg:PNS} with line search as follows. 
\begin{theorem}\label{Thm:PNls2}
Suppose that there exists $r,~r',~R'>0$, such that $r' I \prec H^k \prec R' I$ and $r I \prec \nabla^2F(u^{(k)})$ for all $k \geq 0$. Assume also that $\nabla^2F$ is Lipschitz continuous with constant $L_2$. Let $\{ u^{(k)} \}$ be the sequence generated by Algorithm~\ref{alg:PNS} after sufficient large number of iteration, then
\[
\| u^{(k+1)} - u^* \|_{\HH^k} \leq C \| u^{(k)} - u^* \|^2_{\HH^k} + (1-(1-q)t^l) \| u^{(k)} - u^* \|_{\HH^k}\,,
\]
where 
$C=\frac{L_2\sqrt{R'}}{2r \sqrt{r'}}(2-t^l+qt^l)$ and $q = \| I-(\HH^k)^{-1/2} \nabla^2F(u^{(k)}) (\HH^k)^{-1/2} ||$.
\end{theorem}
See the proof in Appendix~\ref{apdx:1}.

\begin{remark}
As with Algorithm~\ref{alg:PN}, the convergence behavior of Algorithm~\ref{alg:PNS} depends on the structure of $(\HH^k)^{-1} \nabla^2 F(u^{(k)})$ (Note that here we have $\nabla^2 F(u^{(k)})$ instead of $G^k$ in Theorem~\ref{convThm}). From Remark~\ref{remark1}, we see that eigenvalues $\sigma$ of $(\HH^k)^{-1} \nabla^2 F(u^{(k)})$ satisfy $\sigma \in (0,2)$ when $\Delta v \leq 1$, thus $q \in (0,1)$, which ensures the local linear convergence. In the case of $\eps \rightarrow 0$, from $\eqref{evalue000}$ we have $q \rightarrow 0$, and therefore
\[
\| u^{(k+1)} - u^* \|_{\HH^k} \leq C \| u^{(k)} - u^* \|^2_{H^k} + (1-t^l) \| u^{(k)} - u^* \|_{H^k}\,.
\]
Moreover, when $\eps \ll 1$, $\| \HH^k- \nabla^2 F(u^k) \| \rightarrow 0$ as $k \rightarrow \infty$, thus $ \{\HH^k\}_k$ satisfies Dennis-Mor\'e criterion, i.e.,
$\|(\HH_k - \nabla^2 F(u^*))(u_{k+1} - u_k)\|/ \|u_{k+1} - u_k\| \rightarrow 0$,
and Algorithm~\ref{alg:PNS} accepts unit step length after sufficient large number of iterations, i.e. $t^l=1$ (See Lemma 3.5 in \cite{lee2014proximal}). Therefore, a superlinear convergence is obtained after sufficient large number of iterations, which agrees with our numerical experiments (see Fig.~\ref{fig: mesh search}).
\end{remark}

In close, we state the following global sublinear convergence of Algorithm~\ref{alg:PNS}.
\begin{theorem}\label{Thm:ls}
Let $\{ u^{(k)} \}_{k=1}^{+\infty}$ be a sequence generated by Algorithm $\ref{alg:PNS}$, then 
$\displaystyle \min_{k = 0, \cdots , K-1} \{ \| u^{(k+1)} - u^{(k)}   \|^2 _{\HH^k} \} \leq \frac{\theta}{K} F(u^{(0)})$.
\end{theorem}
\begin{proof}
According to the search direction property (Proposition 2.4 in \cite{lee2014proximal}), we have 
\[
(\nabla F(u^{(k)}))^\intercal  v^k + \| v^k  \|^2 _{\HH^k} \leq 0.
\]
Thus by the sufficient descent requirement in the line search and the above inequality, we have 
\begin{align*}
    F(u^{(k+1)}) = F(u^{(k)}+t^l  v^k) & \leq F(u^{(k)}) +  t^l \theta (\nabla F(u^{(k)}))^{\intercal} v^k  \\
    & \leq F(u^{(k)}) - t^l \theta \| v^k  \|^2 _{\HH^k} \\
    & = F(u^{(k)}) - \frac{\theta }{t^l} \| u^{(k+1)} - u^{(k)}   \|^2 _{\HH^k} \\
    & \leq F(u^{(k)}) - \theta \| u^{(k+1)} - u^{(k)}   \|^2 _{\HH^k}.
\end{align*}
Summing up all the inequalities for $k = 0, \cdots, K-1$, we get
\[
F(u^{(K)}) \leq F(u^{(0)})  - \theta \sum_{k=0}^{K-1} \| u^{(k+1)} - u^{(k)}   \|^2 _{\HH^k}\,,
\]
which readily implies the result.  
\end{proof}

\subsection{Positivity}
Note first that the MUSCL scheme we used in the transport step preserves the positivity. Also, the proximal quasi-Newton method for the collision step is positivity preserving as long as the iteration step size $\gamma$ is properly chosen (In practice, this is done either by line search as explained in Algorithm 2.2, or by trial and error as used in Algorithm 2.1). Therefore, the full scheme is positivity preserving. It then remains to show that the step $\gamma$ can indeed be chosen properly, that is, its magnitude does not go to zero when $\eps$ vanishes. For simplicity, we consider $d=1$ in the rest of this subsection.

To start, we write down update rule explicitly, using Algorithm 1 or 2 with \eqref{proxchi}:
\begin{align*}
    u^{(k+1)} &= u^{(k)} - \gamma \HH^{-1} \nabla F(u^{(k)}) + \HH^{-1} \A^{\intercal}(\A \HH^{-1} \A^{\intercal})^{-1}[b-\A(u^{(k)}-\gamma \HH^{-1} \nabla F(u^{(k)}))] ,\\
    &=u^{(k)}-\gamma[\I -\HH^{-1} \A^{\intercal}(\A \HH^{-1} \A^{\intercal})^{-1}\A ] \HH^{-1}\nabla F(u^{(k)}).
\end{align*}
Here we omit the superscript ${(k)}$ in $\HH$ for notation simplicity. Recall that $\A=[\I \quad \D]$, $\HH^{-1} \nabla F = [e(f,m) \quad m]^{\intercal}$, where $e(f,m) = \frac{-\eps \frac{m^2}{f^2}+2 \tau (\log \frac{f}{M})+1}{2\frac{m^2}{f^3}+\frac{2\tau}{f}}$ and $\HH := 
\begin{pmatrix}
\HH_1 &0\\
0 & \HH_2
\end{pmatrix}
$ 
with $\HH_1, \HH_2 \in \RR^{N_v \times N_v}$ and
\begin{align*}
(\HH_1)_{i,j} &= 
\left\{
    \begin{array}{cl}
   (2\eps \frac{ m_i^2 }{f_i^3} + \frac{2 \tau}{f_i})  \Delta v , &  \text{ if $i=j$,} \\
    0, & \text{otherwise,}
\end{array} \right.\\
(\HH_2)_{i,j} &= 
\left\{
    \begin{array}{cl}
   (\frac{2 \eps}{f_i}) \Delta v , &  \text{ if $i=j$,} \\
    0, & \text{otherwise.}
\end{array} \right .
\end{align*}
Then we have 
\[
\HH^{-1} \A^{\intercal}(\A \HH^{-1} \A^{\intercal})^{-1}\A = 
\begin{pmatrix}
\PP & \PP \D \\
\cdot & \cdot
\end{pmatrix},
\]
where $\PP=\HH_1^{-1}(\HH_1^{-1}+\D \HH_2^{-1} \D^{\intercal})^{-1} = (\I +\D \HH_2^{-1} \D^{\intercal} \HH_1)^{-1}$. Thus
\[
f^{{(k+1)}}=f^{(k)}-\gamma [e(f^{(k)},m^{(k)})-(\PP e(f^{(k)},m^{(k)}) + \PP \D m^{(k)})].
\]
Since $\D m^{(k)} = f^{(0)} - f^{(k)} = f^* - f^{(k)}$ from the constraint, one sees
\begin{align}
     f^{(k+1)}&=f^{(k)}-\gamma [e(f^{(k)},m^{(k)})-(\PP e(f^{(k)},m^{(k)}) + \PP (f^* - f^{(k)}))] \nonumber \\ 
    & = f^{(k)} - \gamma[(\I-\PP)e(f^{(k)},m^{(k)})+ \PP f^{(k)}]+\gamma \PP f^* \nonumber
    \\ 
    & = (1-\gamma \PP) f^{(k)} + \gamma \PP f^* - \gamma (\I - \PP) e(f^{(k)}, m^{(k)})\,. \label{update_p}
\end{align}
To proceed, we study the $\eps$ dependence of matrix $\PP$ when $\eps$ vanishes in the following proposition. Note that when $\eps \ll  0$, $(\HH_1)_{i,i} \approx  \frac{2 \tau}{f_i} \Delta v$ and $(\HH_2^{-1})_{i,i} \approx \frac{f_i}{2 \eps} \Delta v$.

\begin{proposition}
For $\PP =  (\I +\D \HH_2^{-1} \D^{\intercal} \HH_1)^{-1}$, there exists an invertible matrix $\UU$ and a diagonal matrix $\Lambda$, both of which are independent of $\eps$, such that $\PP = U(\frac{1}{\eps} \Lambda+ \I)^{-1}U^{-1}$.
\end{proposition}
\begin{proof}
First, we note that $\D$ has exactly one zero eigenvalue, 
where $\D$ is:
\[
\D = 
\begin{pmatrix}
1 & 1  \\
-1 &0 & 1  \\
 & \ddots & \ddots & \ddots\\
 & & -1 & 0 &1 \\
 & & &-1 & -1
\end{pmatrix}
\in \mathbb{R}^{N_v \times N_v}.
\]
Next we define the diagonal matrix $(\hat{\HH}_2)_{i,i} := \frac{2}{ f_i \Delta v}$, then ${\HH}_2^{-1} = \frac{1}{\eps} {\hHH}^{-1}_2 $ and $\D \HH_2^{-1} \D^{\intercal} \HH_1 = \frac{1}{\eps} \D \hHH_2^{-1} \D^{\intercal} \HH_1$.
Consequently, along with the fact that both $\HH_1$ and $\hHH_2$ are invertible, $Rank(\D \hHH_2^{-1} \D^{\intercal} \HH_1)=  Rank((\D \hHH_2^{-1/2}) (\D \hHH_2^{-1/2})^{\intercal} \HH_1)  = N_v-1$. Also note that since $\D \hHH_2^{-1} \D^{\intercal} $ is symmetric and $\HH_1$ is positive definite, we have 
\[
\HH_1^{\half} \D \hHH_2^{-1} \D^{\intercal} \HH_1 \HH_1^{-\half} = \HH_1^{\half} \D \hHH_2^{-1} \D^{\intercal}  \HH_1^{\half}\,,
\]
and therefore $\D\hHH_2^{-1} \D^{\intercal} \HH_1$ is similar to a symmetric matrix, and thus diagonalizable. Moreover, since both $\D\hHH_2^{-1} \D^{\intercal}$ and $\HH_1$ are positive semi-definite, we conclude that $\D\hHH_2^{-1} \D^{\intercal} \HH_1$ has exactly one zero eigenvalue and all the rest are positive. So there exist invertible matrix $U$ independent of $\eps$, s.t.
\[
\D \hHH_2^{-1} \D^{\intercal} \HH_1 = U \Lambda U^{-1}.
\]
where $\Lambda$ is a diagonal matrix with non-negative entries and likewise
\[
\D \HH_2^{-1} \D^{\intercal} \HH_1 = \frac{1}{\eps} U \Lambda U^{-1}.
\]
As a result,   
\[
K= \I +\D \HH_2^{-1} \D^{\intercal} \HH_1 = U (\frac{1}{\eps} \Lambda+ \I) U^{-1} 
\]
is invertible. Thus 
\[
P= K^{-1} = U(\frac{1}{\eps} \Lambda+ \I)^{-1}U^{-1}.
\]
\end{proof}

From the above proposition, we see that when $\eps \rightarrow 0$, the matrix $\PP \rightarrow 0$, and~$\eqref{update_p}$ becomes
\[
f^{(k+1)} =  f^{(k)} - \gamma  e(f^{(k)}, m^{(k)}).
\]
Therefore the selection of step size $\gamma$ that guarantees positivity doesn't vanish with $\eps$, instead its magnitude only depends on the initial data $f^{(0)}$.

\subsection{Asymptotic property}
First we look at the implicit collision step. Consider \eqref{disc1}--\eqref{disc2} for any fixed $x_i$, and let $\eps \rightarrow 0$ we have the following constrained optimization:
\begin{align}\label{limvar}
     f^{n+1}_{\bj} \in \argmin_{f} \left\{ \sum_{\bj}  f_{\bj} \ln\left(\frac{f_{\bj}}{M^*_{\bj}}\right) \Delta v^d \right\},
     \qquad \text{s.t.} \quad \sum_{\bj} f_{\bj} = \sum_{\bj} M^*_{\bj} \, .
\end{align}
where the constraint is obtained by summing over $\bj$ in \eqref{disc2}.
To make this limit of the variational problems \eqref{disc1}--\eqref{disc2} towards \eqref{limvar} as $\epsilon\to 0$ fully rigorous, one can make a direct use of the theory of $\Gamma$-convergence, see \cite{Braides}. Let us denote by $F_\epsilon(u)$ the functional defining the variational problems \eqref{disc1}--\eqref{disc2} and $F_0(u)$ the functional for \eqref{limvar}. 
In fact, it is very easy to check that in this finite dimensional setting, the sequence of functionals $F_\epsilon(u)$ is monotone with respect to $\epsilon$ and thus, the $\Gamma$-convergence of the $\epsilon$-regularized problems \eqref{disc1}--\eqref{disc2} to \eqref{limvar} follows from \cite[Chapter 2]{Braides}. This shows that the infimum value of the functional $F_\epsilon(u)$ converges to the infimum value of the functional $F_0(u)$ as $\epsilon\to 0$. Moreover, any cluster point of approximating sequences in $\epsilon$ will  converge to a point where the infimum of $F_0(u)$ is achieved. Therefore, this shows our claim above on the right limiting optimization problem. Since further discussion of this point is not needed for the purposes of this work, we leave to the reader to check that we have $\epsilon$-equicoercivity of the minimizing sequences in the $f$ variables and in the scaled $\sqrt{\epsilon} \,m$ variables, this together with the previous statement of convergence of the infimum values and the constraint \eqref{disc2} lead to the convergence of the minimizers of $F_\epsilon$ to the minimizer of $F_0$.

Coming back to limiting collisional step \eqref{limvar}, the corresponding Lagrangian writes as:
\[
L(f,\lambda)= \sum_{\bj}  f_{\bj} \ln\left(\frac{f_{\bj}}{M^*_{\bj}}\right)\Delta v^d  + \lambda \left(\sum_{\bj} f_{\bj} - \sum_{\bj} M^*_{\bj}\right),
\]
which leads to the following optimality condition: 
\[
\frac{\delta L}{\delta f_{\bj}} = \ln\left(\frac{f_{\bj}}{M^*_{\bj}}\right)\Delta v^d + (\lambda+\Delta v^d) \textbf{1} = 0.
\]
Therefore, one sees that $f_{\bj}$ differs from $M_{\bj}^*$ by one constant multiplier $\exp(-\frac{\lambda+\Delta v^d}{\Delta v^d})$ for all $\bj$. Along with mass conservation, we then have $f^{n+1}=M^*$. Recall the definition of $M^*$ in \eqref{eqn:Mstar}, and since $\rho^* = \rho^{n+1}$, we have
\begin{equation*}
  f^{n+1}_{\bj} \rightarrow \frac{\rho^{n+1}}{(\sqrt{2\pi})^d} e^{-\frac{|v_{\bj} + \nabla_x \phi^{n+1}|^2}{2}}\,,  \qquad \text{for all}  \quad n > 0 \,.
\end{equation*}
This allows us now to connect to the transport step in order to obtain the limiting scheme and check for consistency with the limiting equation \eqref{rho_limit}, i.e., showing the asymptotic property of the scheme.
Plugging it into the transport step and summing over $\bj$, we have
\begin{equation}\label{sum_over_v}
   \frac{\rho^{n+1} - \rho^n}{\tau} + \sum_{\bj} v_{\bj} \cdot \nabla_x f^n_{\bj} = 0 \,,
\end{equation}
where 
\begin{align*} \sum_{\bj} v_{\bj}\cdot \nabla_{x} f_{\bj}^{n} &=\nabla_{x} \cdot \left[\sum_{\bj}\left(v_{\bj}+\nabla_{x} \phi^{n}-\nabla_{x} \phi^{n}\right) e^{-\frac{\left|v_{j}+\nabla_{x} \phi^{n}\right|^{2}}{2}} \frac{\rho^{n}}{(\sqrt{2 \pi})^d}\right] \nonumber \\ &=-\nabla_{x} \cdot\left[\sum_{\bj} \nabla_{x} \phi^{n} e^{-\frac{\left|v_{\bj}+\nabla_{x} \phi^{n}\right|^{2}}{2}} \frac{\rho^{n}}{(\sqrt{2 \pi})^d}\right] \nonumber \\ &=-\nabla_{x} \cdot \left(\rho^{n} \nabla_{x} \phi^{n}\right) \sum_{\bj} \frac{1}{(\sqrt{2 \pi})^d} e^{-\frac{\left|v_{\bj}+\nabla_{x} \phi^{n}\right|^{2}}{2}} \,.
\end{align*}
It is obvious that $\sum_{\bj} \frac{1}{(\sqrt{2 \pi})^d} e^{-\frac{\left|v_{\bj}+\nabla_{x} \phi^{n}\right|^{2}}{2}}$ approximates one with at least second order accuracy in $v$. Therefore $\eqref{sum_over_v}$ gives a consistent  semi-discretization for the limit equation $\eqref{rho_limit}$, which concludes the asymptotic property of our scheme.

\section{Numerical examples}
In this section, we provide several examples demonstrating the efficiency and accuracy of our algorithms. The examples are presented in the order of increasing dimensions in $v$. The stopping criteria is chosen as: 
\[ 
\frac{|F(u^{(k+1)})-F(u^{(k)})|}{ | F(u^{(k)})|} < \delta, \qquad \frac{\|u^{(k+1)}-u^{(k)}\|_1}{ \| u^{(k)}\|_1} < \delta \,.
\]
where $\delta = 10^{-7}$ for all examples. Throughout the examples, we use Algorithm 2 with $\theta = 0.01$ unless otherwise specified. 
\subsection{1D in velocity}
\subsubsection{Convergence}
We first show that the convergence of our optimization algorithm is uniform in $\eps$. As this step matters only in $v$ direction, we consider the spatially homogeneous case:
\begin{equation}\label{sph}
\left\{\begin{array}{l}{\partial_{t} f=\frac{1}{\varepsilon} \nabla_{v} \cdot\left(vf +  \nabla_{v} f \right)}, \\ f(0,v)= 2e^{-\frac{(v-1.5)^2}{1.2}}+ \frac{1}{2} e^{-\frac{(v+1.5)^2}{1.5}}\,.\end{array}\right.
\end{equation}
The computational domain is chosen as $v \in [-5,5]$, and time step $\tau =0.05$. For one step JKO scheme, we show convergence behavior with varying $\eps$ by computing the relative error 
\begin{equation}
\textrm{error}_k = \frac{\|u^{(k)} - u^* \|_1}{\|u^*\|_1} . \label{errork}
\end{equation}
in Fig.~\ref{fig:PN_1D} for both fixed step size and adaptive step size with line search. Here $u^*$ is obtained by using Algorithm 1 with 160 iterations. It is seen that with fixed step size, a linear convergence is observed; while with line search, an initial linear convergence is followed by a super-linear convergence, which happens when the step size approaches one. 
\begin{figure}[!ht]
    \centering
    \includegraphics[width=0.48\textwidth]{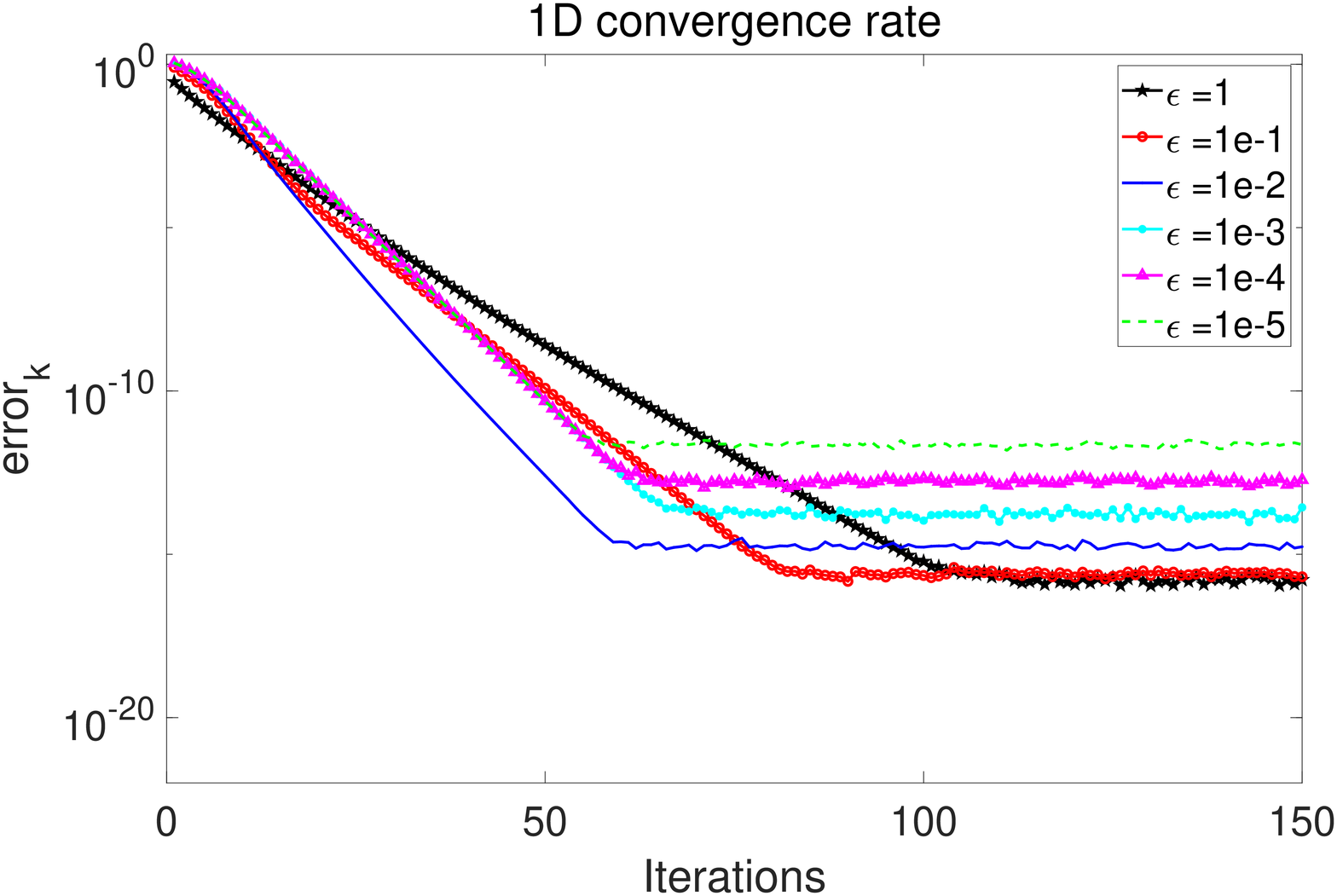}
    \includegraphics[width=0.48\textwidth]{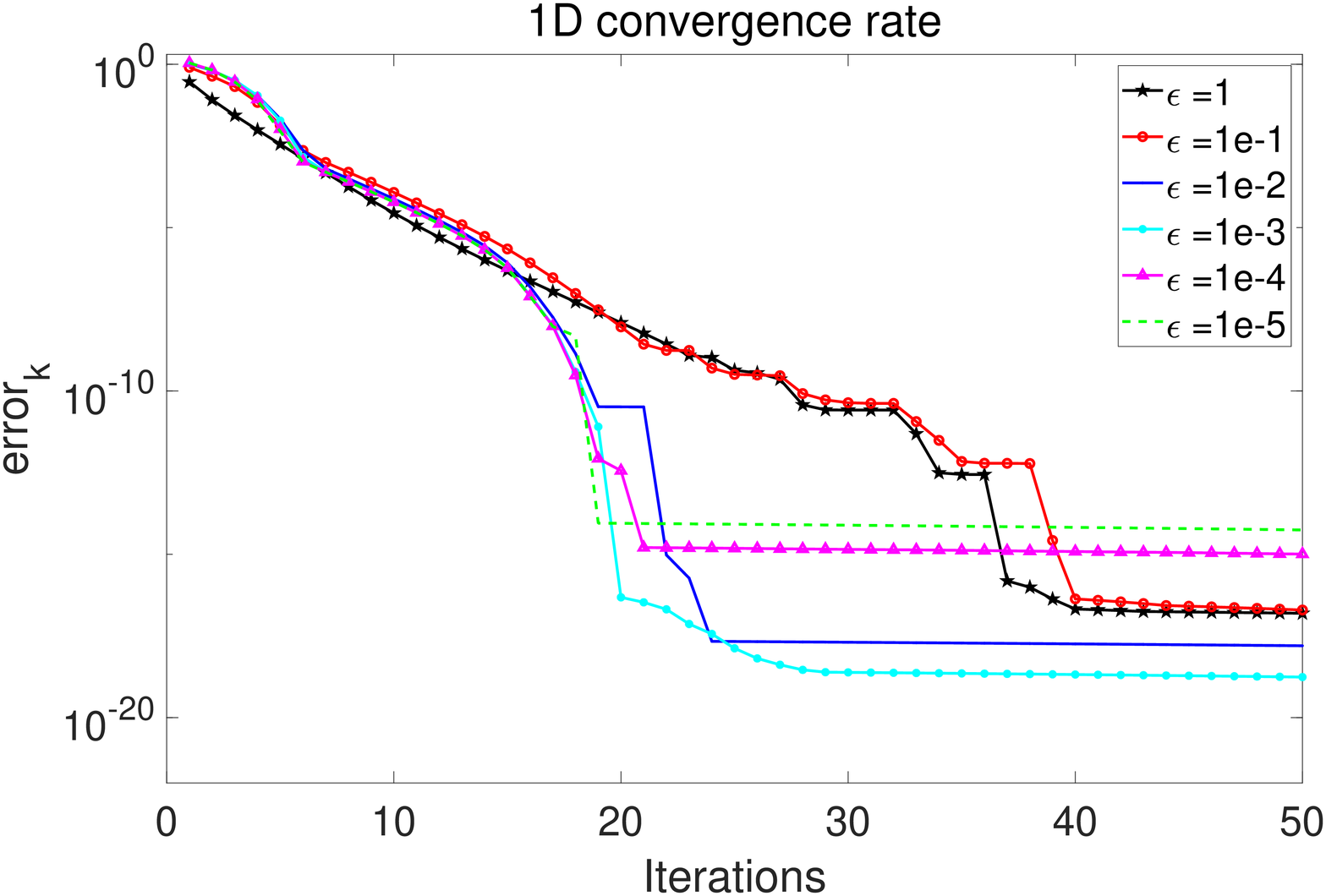}
    \caption{Convergence of one step JKO scheme with respect to $\epsilon$. Left: Proximal quasi-Newton with fixed step size: $\gamma = 0.5,0.5,0.5,0.4,0.4,0.4$ for $\eps = 1, 1e-1, 1e-2, 1e-3,1e-4,1e-5$ respectively. Right: Proximal quasi-Newton with line search. In both cases, $\Delta v = 10/64$, $\tau = 0.05$.}
    \label{fig:PN_1D}
\end{figure}
We also record the real simulation time of two methods in one outer time step when $\delta = 1e-7$ and $\tau = 0.05$. Result are shown in Table~\ref{tab:time cost}, where one sees that these two approaches are comparable in terms of efficiency.
\begin{table}[!ht]
    \centering
    \begin{tabular}{c c c}
    \hline\hline
          Method & Fix step size & Line search \\ [0.5ex] 
          \hline
         $\varepsilon = 1$  & 0.017s & 0.008s \\
         $\varepsilon = 1e-1$ & 0.011s & 0.008s \\
         $\varepsilon = 1e-2$ & 0.013s & 0.015s \\
         $\varepsilon = 1e-3$ & 0.02s & 0.007s \\
         $\varepsilon = 1e-4$ & 0.013s & 0.011s \\
         $\varepsilon = 1e-5$ & 0.013s & 0.030s \\
         \hline
    \end{tabular}
    \caption{ Run time of one outer time step. For fix step size method, we use $\gamma = 0.5,0.5,0.5,0.4,0.4,0.4$ for $\eps = 1,1e-1, 1e-2, 1e-3,1e-4,1e-5$ respectively. In both cases, $\Delta v = 10/64$, $\tau = 0.05$.}
    \label{tab:time cost}
\end{table}

Next, we check the dependence of convergence on the mesh size $\Delta v$ and time step $\tau$ with the same setting as above. Two cases with $\eps = 1$ and $\eps = 1e-5$ are considered. $u^*$ is again obtained by running Algorithm 1 with 160 iterations. The results are collected in Fig.~\ref{fig: mesh} and \ref{fig: mesh search}, where an almost uniform convergence behavior is observed with different $\Delta v$, which indicates the {\it independency} of our algorithm on the mesh size. 
\begin{figure}[!ht]
    \centering
    \includegraphics[width=0.45\textwidth]{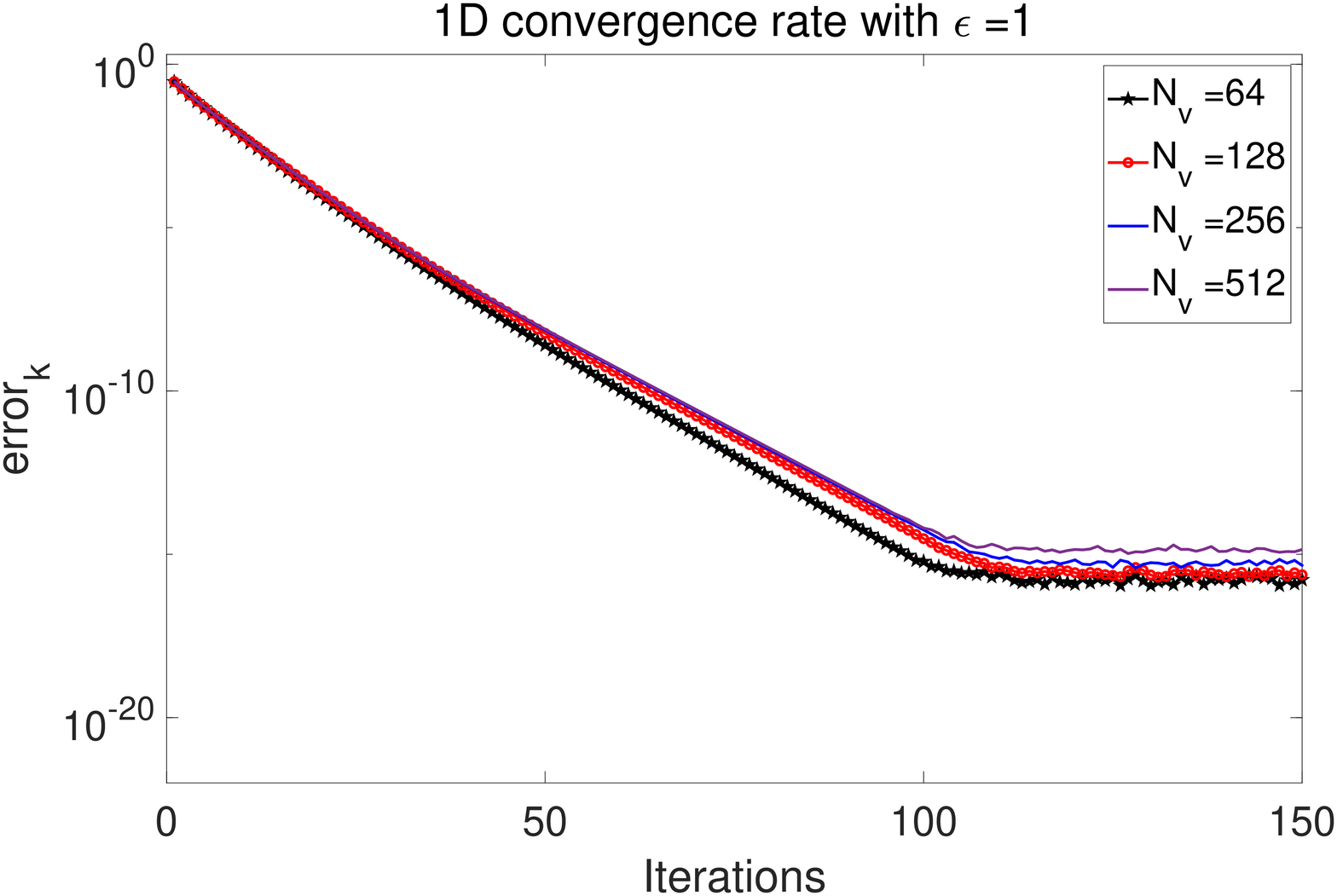}
    \includegraphics[width=0.45\textwidth]{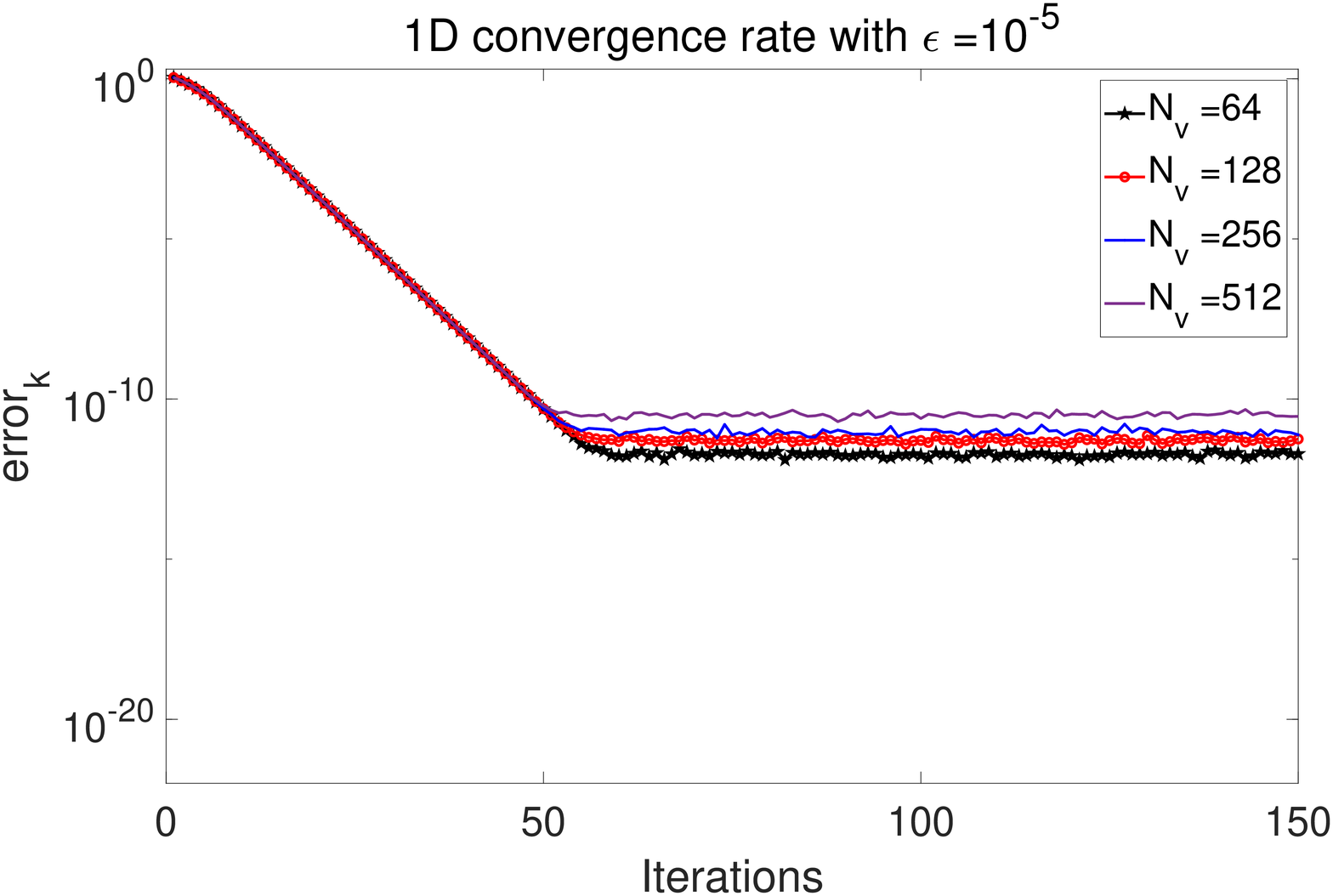}
    \vfill
    \includegraphics[width=0.45\textwidth]{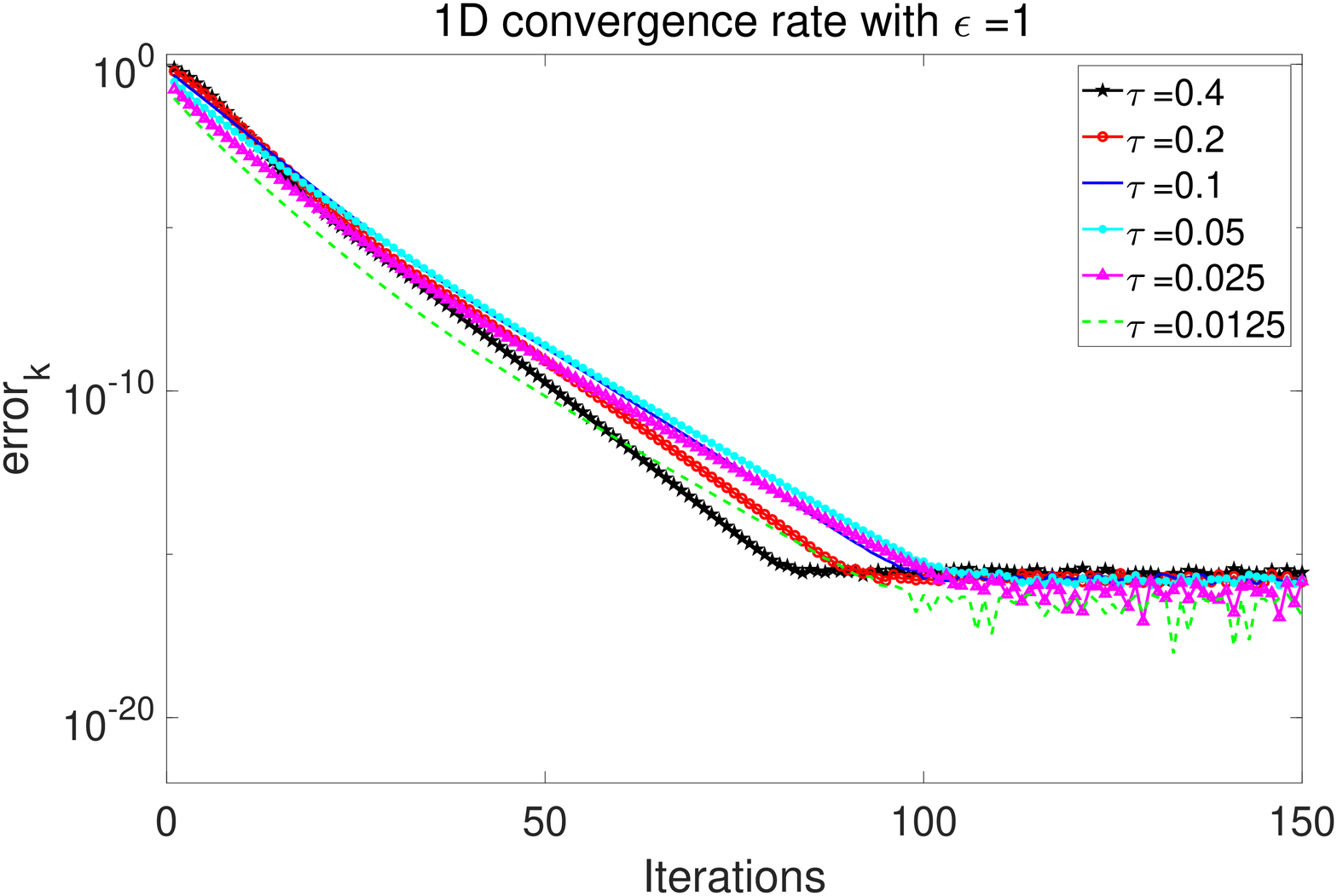}
    \includegraphics[width=0.45\textwidth]{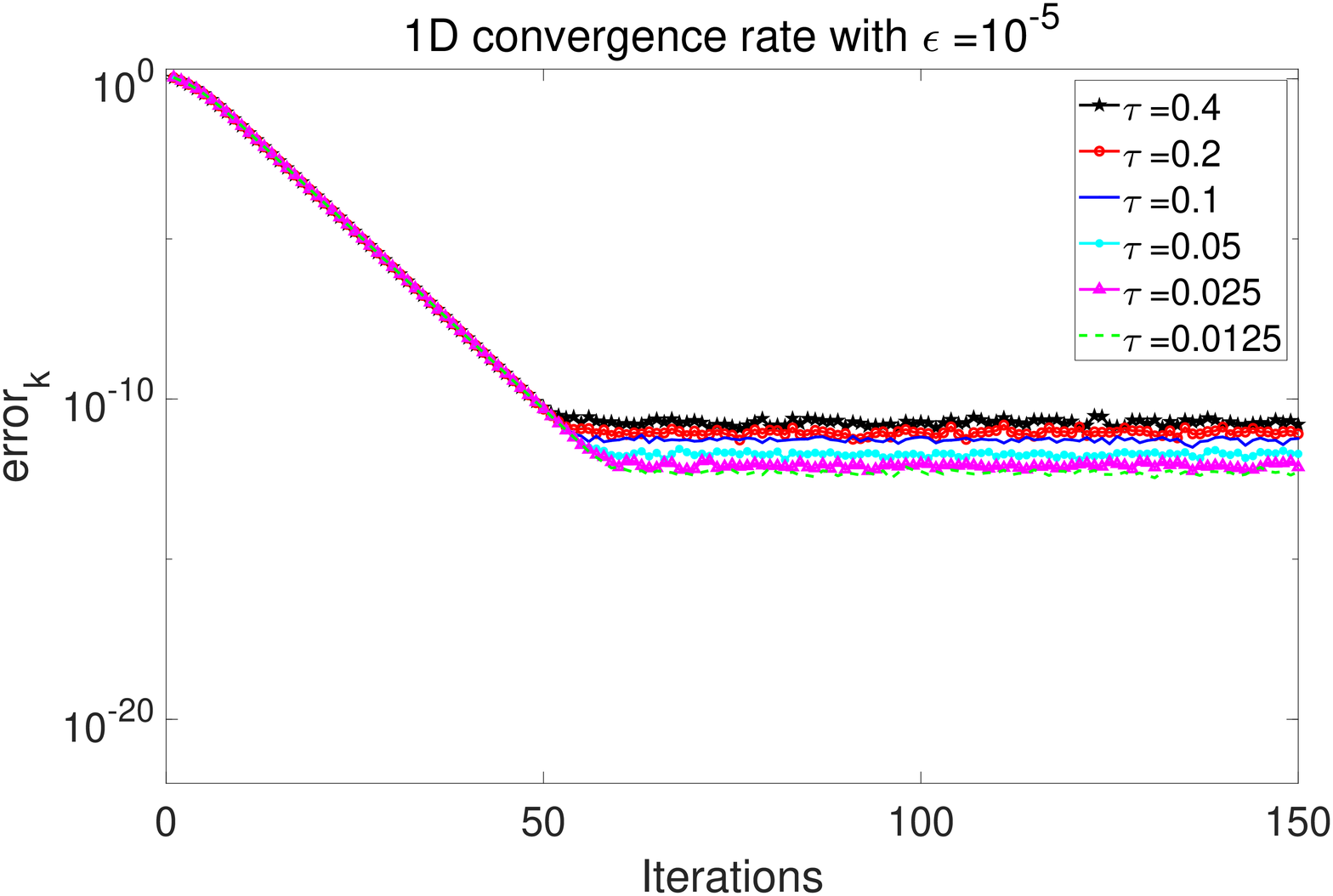}
    \caption{Convergence of Algorithm 1 with different $ \Delta v$ (top), or different $\tau$ (bottom). Top left: $\eps =1$, step size $\gamma =0.5$ and $\tau=0.05$. Top right: $\eps =1e-5$, step size $\gamma =0.4$ and $\tau=0.05$.  Bottom left: $\eps =1$, step size $\gamma =0.5$ and $\Delta v=10/64$. Bottom right: $\eps =1e-5$, step size $\gamma =0.4$ and $\Delta v=10/64$. }
    \label{fig: mesh}
\end{figure}

\begin{figure}[!ht]
    \centering
    \includegraphics[width=0.45\textwidth]{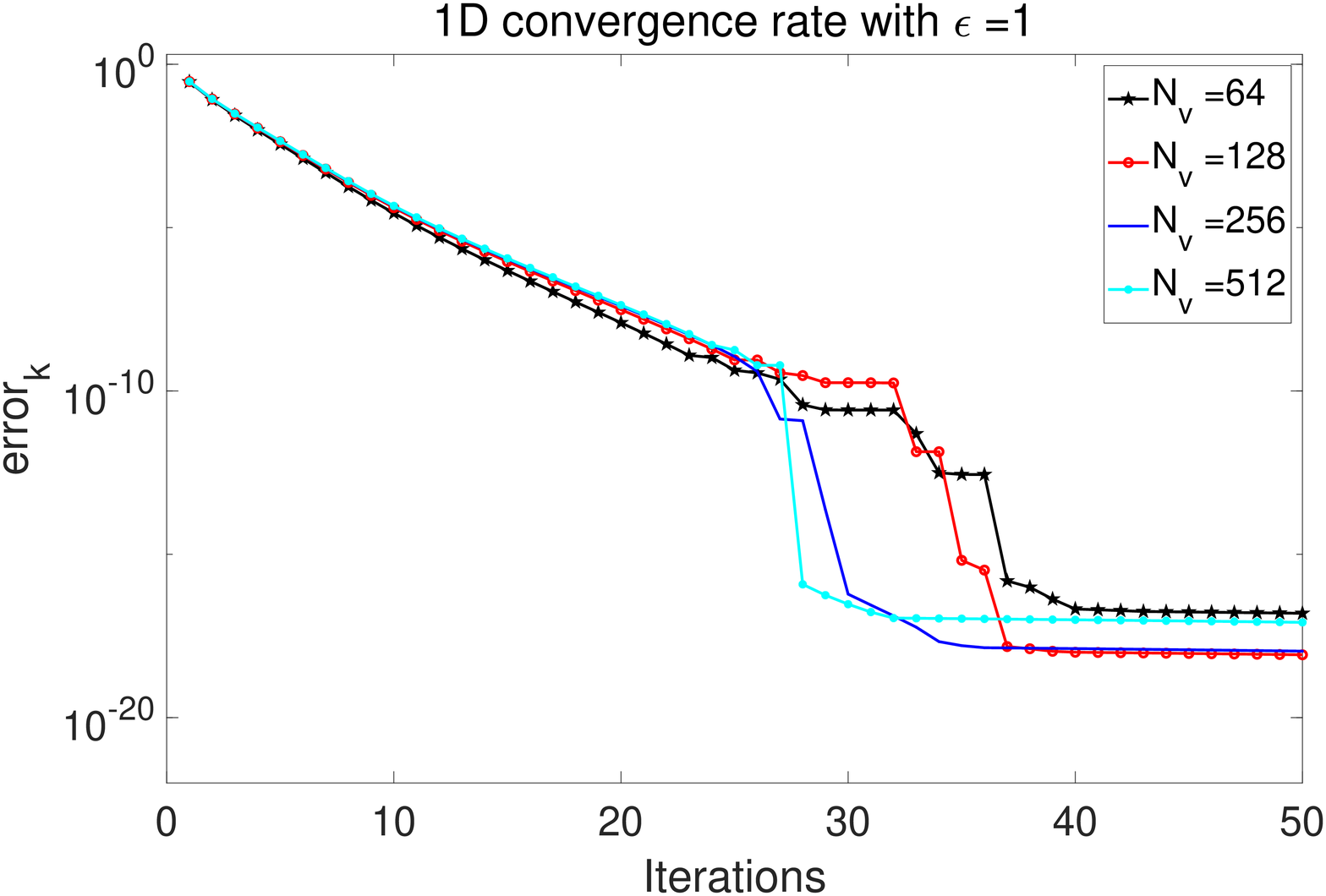}
    \includegraphics[width=0.45\textwidth]{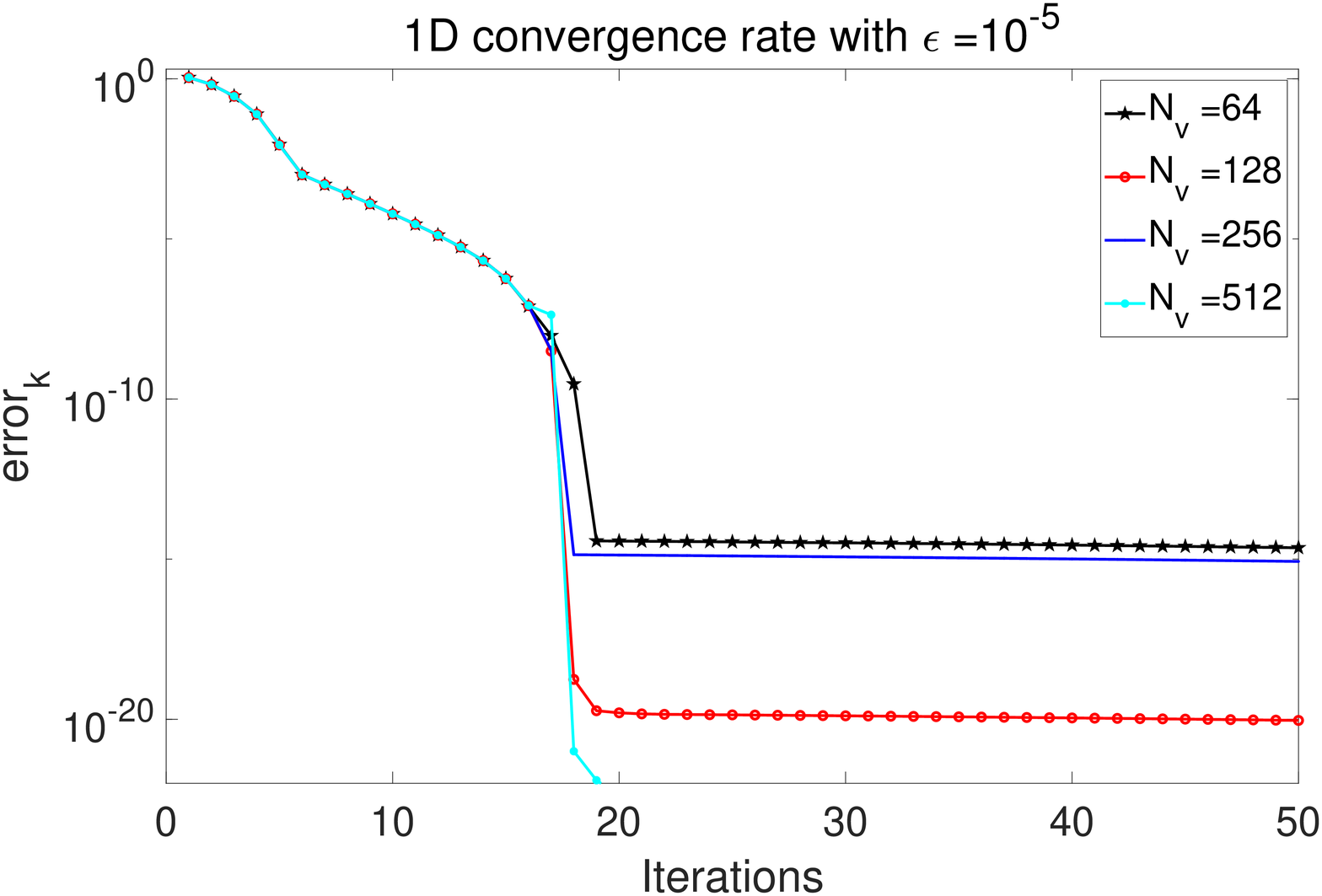}
    \vfill
    \includegraphics[width=0.45\textwidth]{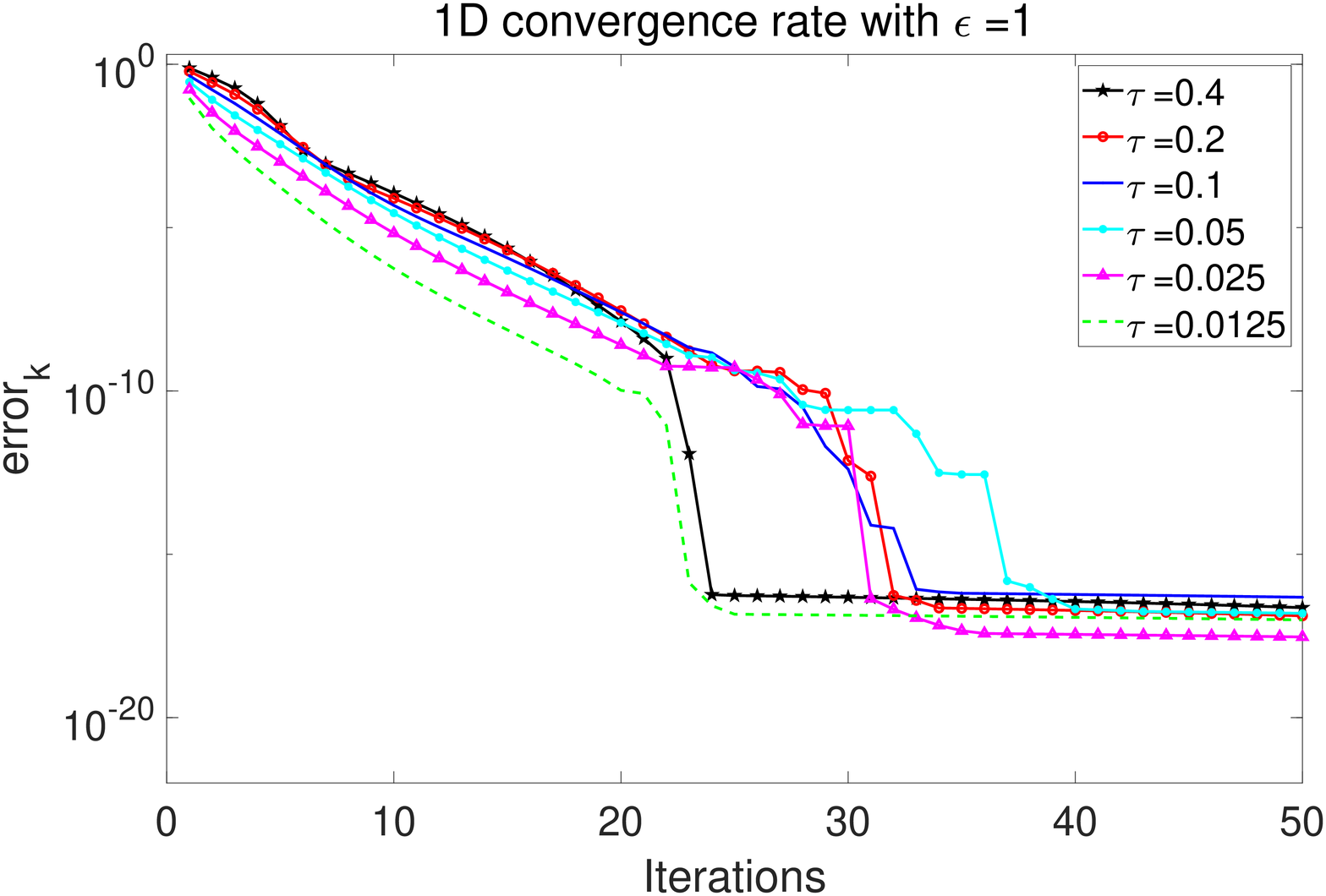}
    \includegraphics[width=0.45\textwidth]{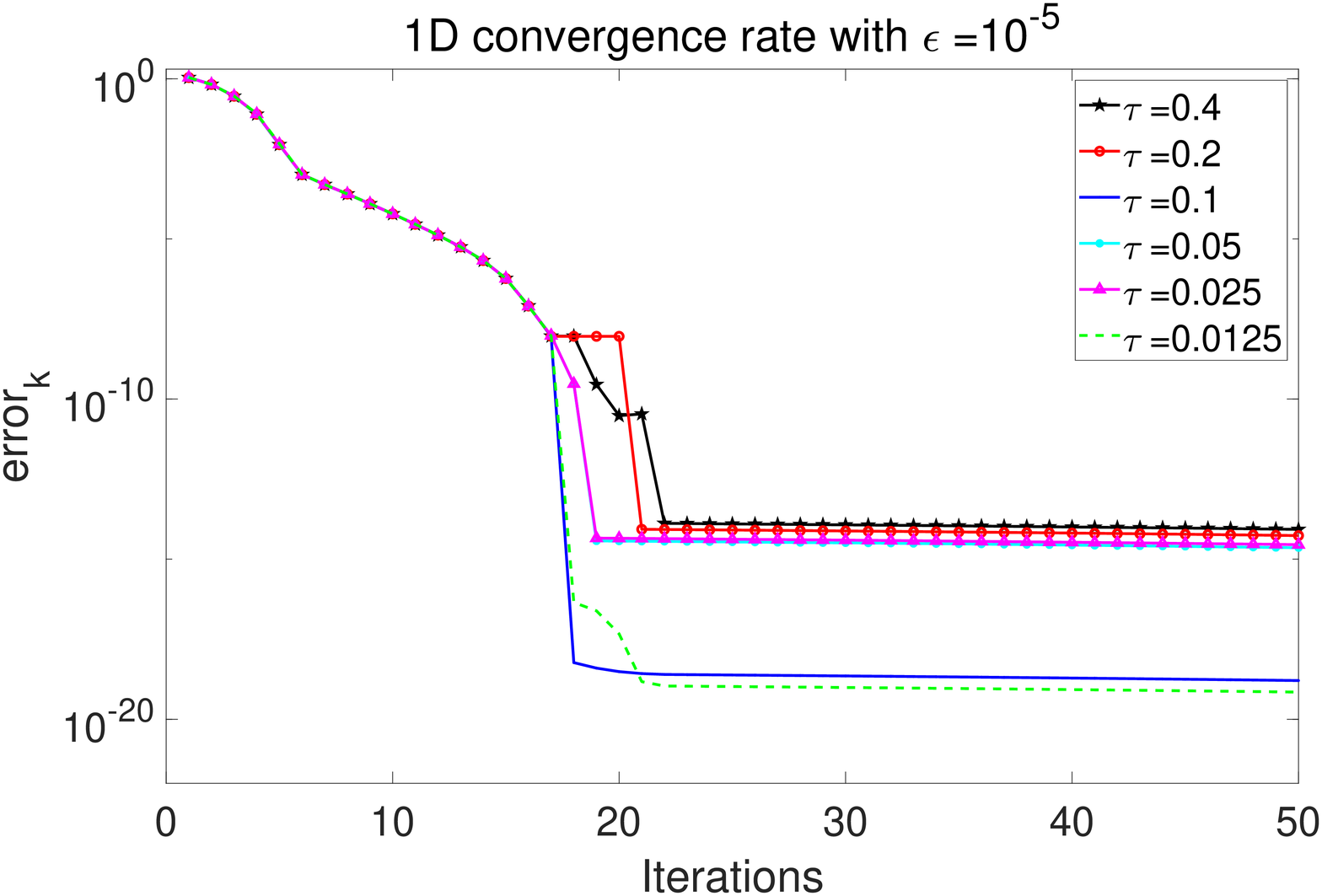}
    \caption{Convergence of Algorithm 2 with different $ \Delta v$ (top), or different $\tau$ (bottom). Top left: $\eps =1$, $\tau=0.05$. Top right: $\eps =1e-5$, $\tau=0.05$. Bottom left: $\eps =1$, $N_v=64$. Bottom right: $\eps =1e-5$, $N_v=64$.}
    \label{fig: mesh search}
\end{figure}
We also show convergence behavior at different time steps. In Fig~\ref{fig: diff_time_step}, in the case when $\eps $ large, it converges slower at beginning for both fix step size method and line search method. As $f$ approaches to the equilibrium, less iterations are required to converge. And in the case $\eps$ is small, it reaches equilibrium in merely one time step, thus we see flat curve after first several time steps, which implies it stay at equilibrium.
\begin{figure}[!ht]
    \centering
    \includegraphics[width=0.45\textwidth]{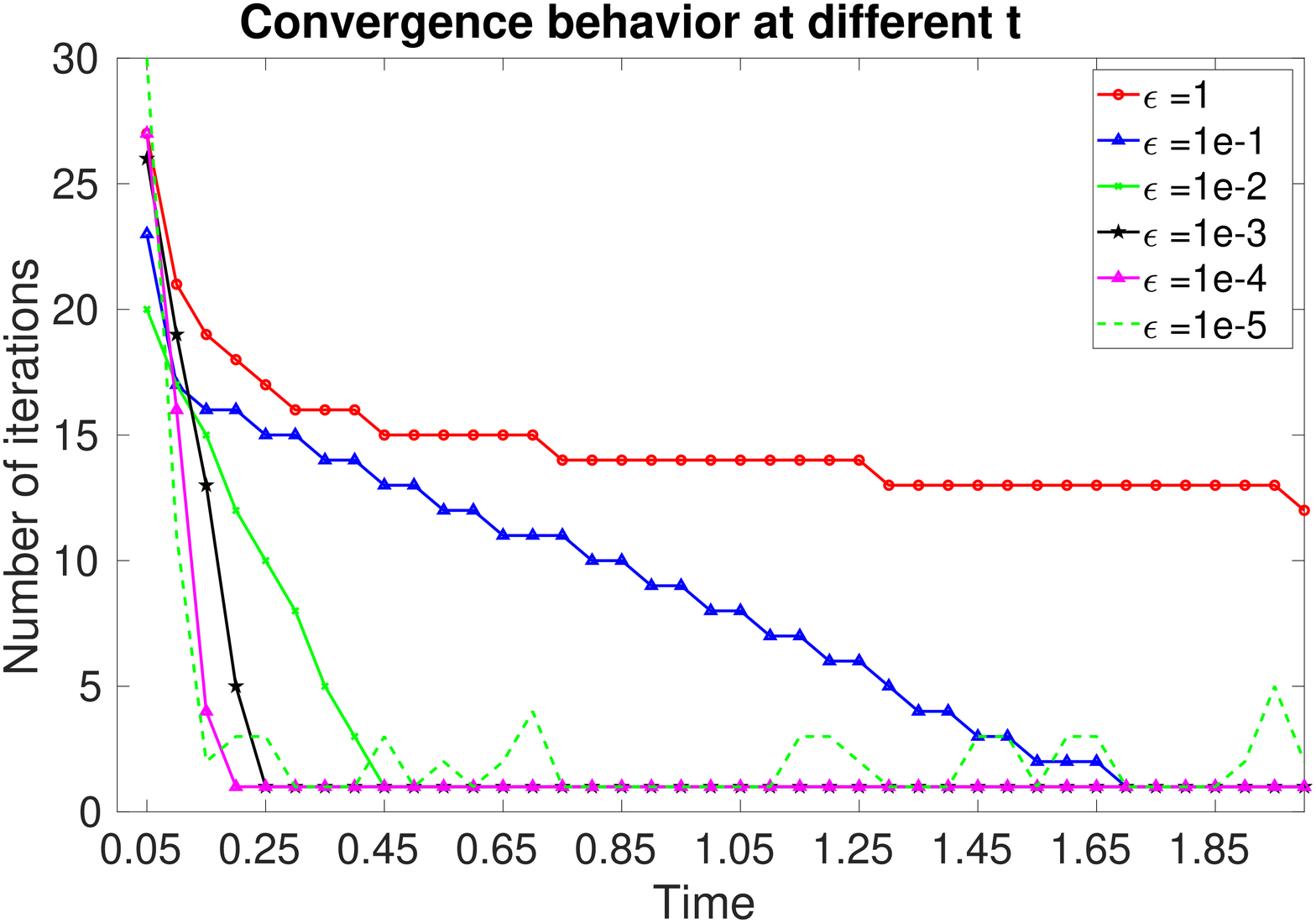}
    \includegraphics[width=0.45\textwidth]{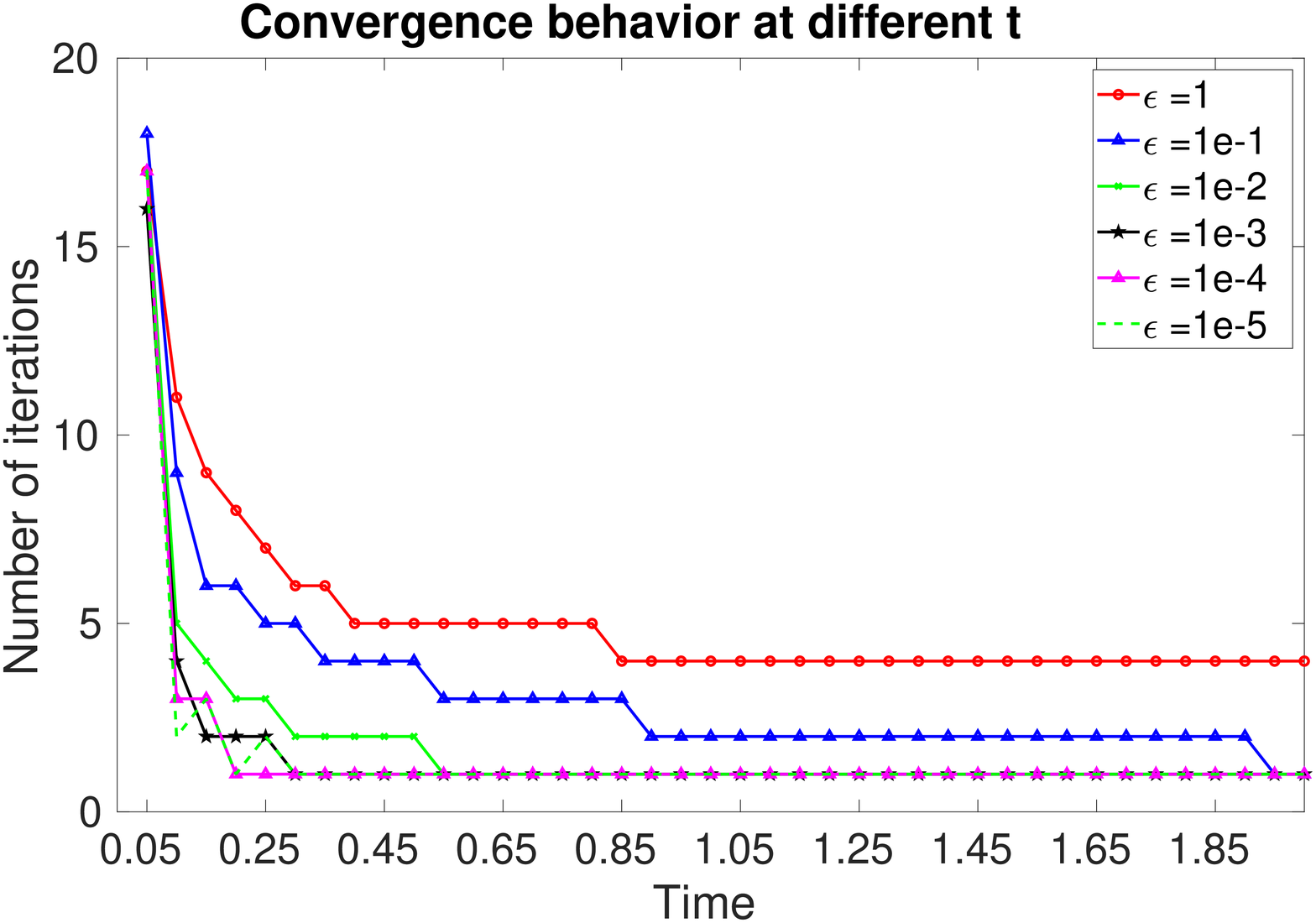}
    \caption{Convergence behavior at different time steps with $\tau=0.05$, $\Delta v=10/64$ and stopping criterion $\delta = 1e-7$. Left figure is generated by Algorithm 1 with $\gamma = 0.5,0.5,0.5,0.4,0.4,0.4$ for $\eps = 1, 1e-1, 1e-2, 1e-3,1e-4,1e-5$ respectively. Right figure is generated by Algorithm 2. }
    \label{fig: diff_time_step}
\end{figure}
\subsubsection{Accuracy}
In this subsection, we test the order of accuracy of our variational scheme with distinct $\eps$. For accuracy in $v$, we consider spatially homogeneous case \eqref{sph} with fix $\tau=0.0063$, and compute the following relative error with decreasing $\Delta v$: 
\[
e_{\Delta v} = \| f_{\Delta v}(v,T) -f_{\frac{\Delta v}{2}}(v,T) \|_1 := \sum_{j=1}^{N_v} |(f_{\Delta v})_j(T) - (f_{\frac{\Delta v}{2}})_j(T)| \Delta v.
\]
The results are gathered in Fig.~\ref{fig:second_order_in_v}, where a uniform second order accuracy is observed. 
\begin{figure}[!ht]
    \centering
    \includegraphics[width=0.6\textwidth]{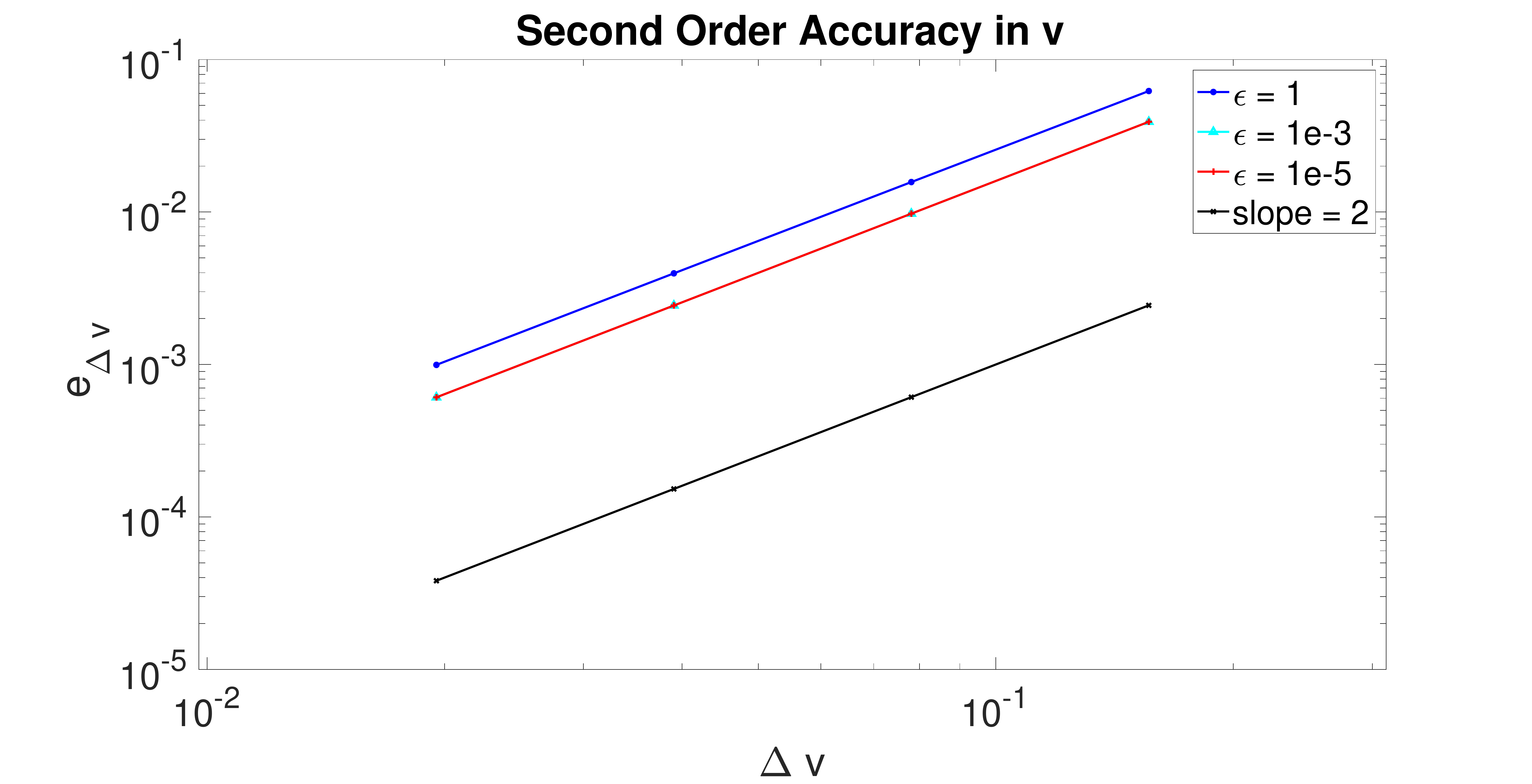}
    \caption{Relative error $e_{\Delta v}$ with  $\Delta v = 10/64, 10/128, 10/256, 10/512, 10/1024$ and fixed $\tau = 0.0063$,  $T=0.1$ and $N_{max} = 1000$. The black line indicates second order accuracy. }
    \label{fig:second_order_in_v}
\end{figure}
To check the accuracy in $x$ and $t$, we consider the spatially inhomogeneous VPFP system \eqref{vpfp} with the following initial condition:
\begin{equation} \label{IC}
 \rho^{0}(x)={\sqrt{2 \pi}}(2+\cos (2 \pi x)),~ f^{0}(x, v)=\frac{\rho^{0}(x)}{2\sqrt{2 \pi}}\left(e^{-\frac{|v+1.5|^{2}}{2}}+e^{-\frac{|v-1.5|^{2}}{2}}\right), 
 h(x)=\frac{5.0132}{1.2661} e^{\cos (2 \pi x)}.
\end{equation} 
and compute the relative error:
\[
e_{\tau} = \| f_{\tau}(T,x,v) -f_{\frac{\tau}{2}}(T,x,v) \|_1 := \sum_{i=1}^{N_x} \sum_{j=1}^{N_v} |(f_{\tau})_{i,j}(T) - (f_{\frac{\tau}{2}})_{i,j}(T)| \Delta v \Delta x.
\]
and
\[
e_{\Delta x} = \| f_{\Delta x}(T,x,v) -f_{\frac{\Delta x}{2}}(T,x,v) \|_1 := \sum_{i=1}^{N_x} \sum_{j=1}^{N_v} |(f_{\Delta x})_{i,j}(T) - (f_{\frac{\Delta x}{2}})_{i,j}(T)| \Delta v \Delta x.
\]
As expected, we observe first order accuracy in time and second order accuracy in space, both uniformly in $\eps$. The results are shown in Fig.~\ref{fig:order of tau and x}.
\begin{figure}[!ht]
    \centering
    \includegraphics[width=0.45\textwidth]{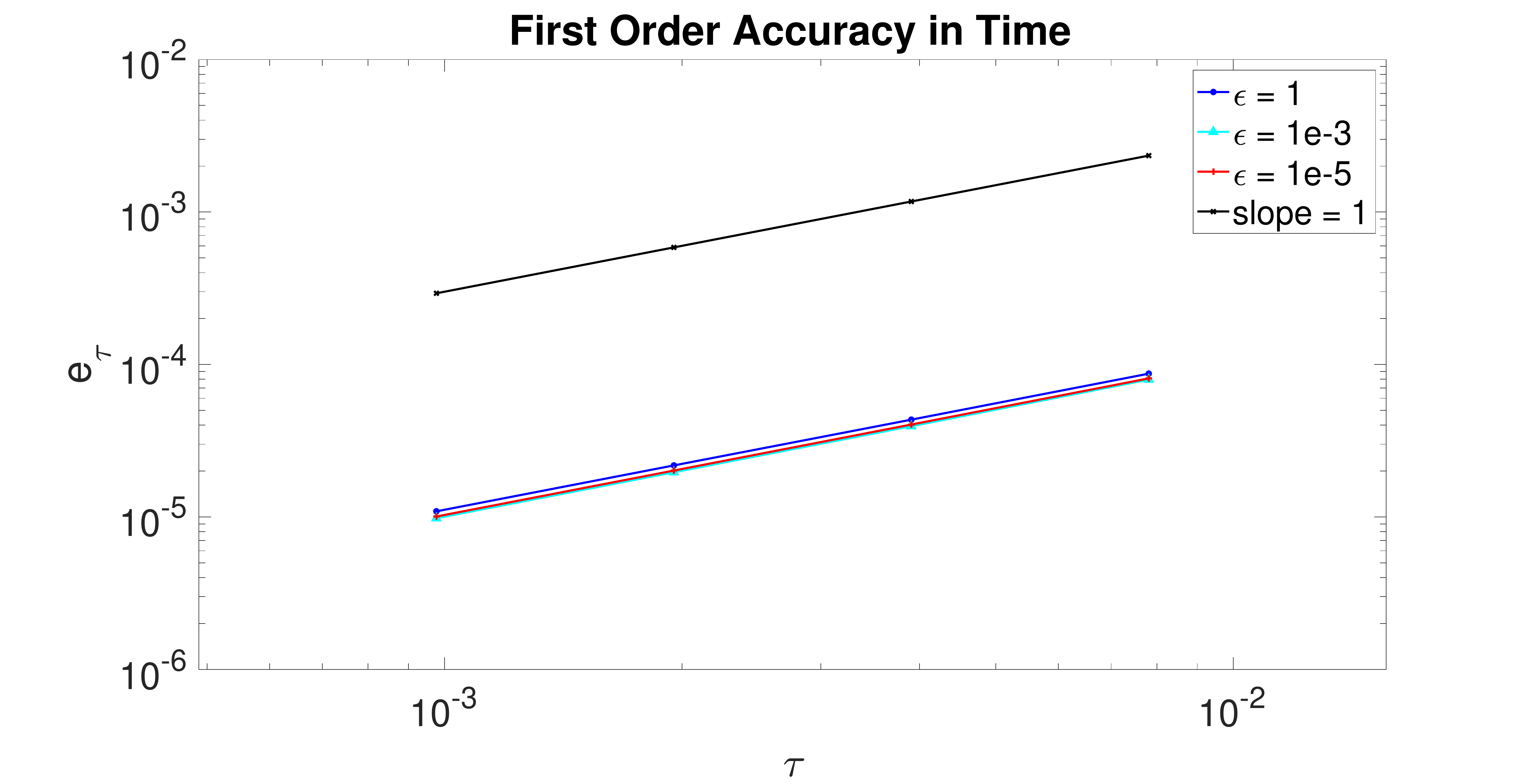}
    \includegraphics[width=0.45\textwidth]{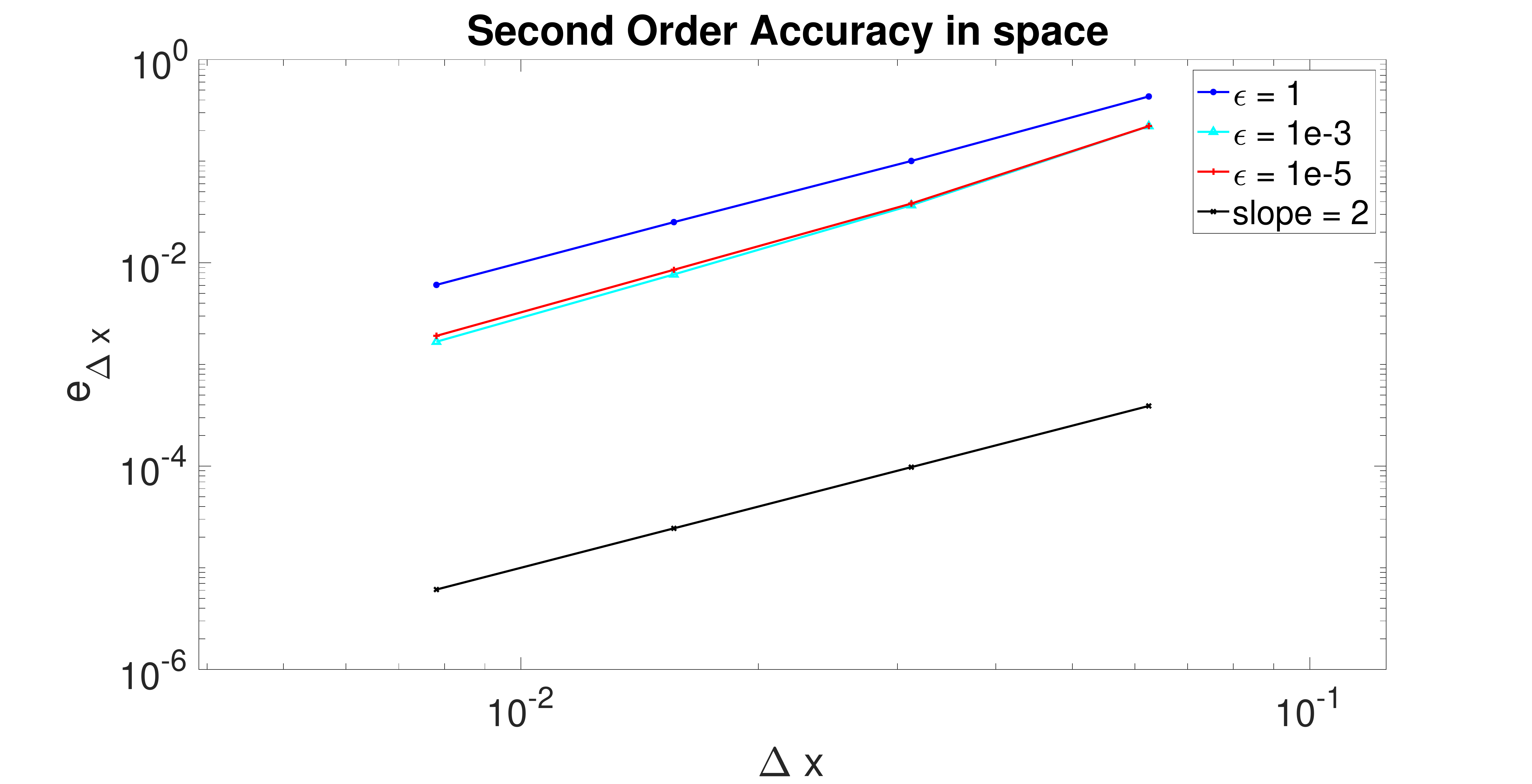}
    \caption{Left: relative error $e_{\tau}$ with fixed $\Delta v = 12/64$, $\Delta x =1/16 $, and varying $\tau =  \Delta x /8 , \Delta x /16 , \Delta x /32 , \Delta x /64 , \Delta x /128 $. The black line indicates first order accuracy. Right: relative error $e_{\Delta x}$ with fixed $\Delta v = 10/64$, and varying $\Delta x = 1/16, 1/32, 1/64, 1/128, 1/256$ and $\tau = \Delta x / 8$. The black line indicates second order accuracy.  In both cases, $T=0.1$, $N_{max} = 1000$. }
    \label{fig:order of tau and x}
\end{figure}

\subsubsection{The asymptotic preserving property}\label{vpfp_ap}
This section is devoted to check the asymptotic property of our scheme. For this purpose, 
consider the spatially inhomogeneous VPFP system \eqref{vpfp} with the following initial condition \eqref{IC}. The computational domain is chosen as $x \in [0,1] $ and $v\in [-6,6]$. At every time $t_n=n \cdot  \tau$, we consider the $l_1$ distance between our solution $f^n$ with the local equilibrium $M^n$ as
\[
\|f^n-M^n \|_1= \sum_{i,j} |f^n(x_i,v_j)-M_i^n(v_j)|  \Delta x  \Delta v.
\]
Fig.~\ref{fig:DR1D} shows that this distance decreases at the order of $\mathcal O (\eps)$ with decreasing $\eps$, which confirms asymptotic property.
\begin{figure}[!ht]
    \centering
    \includegraphics[width=0.55\textwidth]{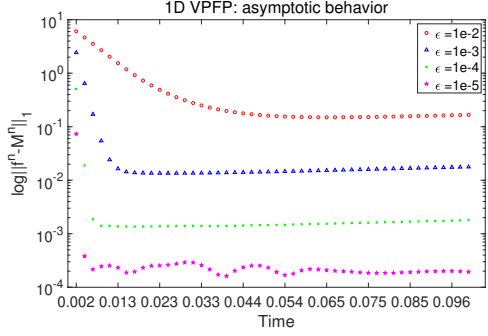}
    \caption{Evolution of distance between our solution $f$ and the local equilibrium $M$ and with decreasing $\eps$. Here $N_{max} = 1000$, $\Delta x =1/64 $, $\Delta v=12/64$, and $\tau = \Delta x/16 $.}
    \label{fig:DR1D}
\end{figure}

\subsubsection{Entropy decay}
In this section, we first consider the Vlasov-Fokker-Planck (VFP) system
\begin{equation}
\partial_t f +  v \nabla_x f -   \nabla_x \phi_0 \cdot \nabla_v f =  \nabla_v \cdot (vf + \nabla_v f) \nonumber
\end{equation}
with a fix external potential $\phi_0(x)$ and check the entropy decay property. The initial condition is taken to be: 
\begin{align*}
& \rho^{0}(x)={\sqrt{2 \pi}}(2+\cos (2 \pi x)), \quad f^{0}(x, v)=\frac{\rho^{0}(x)}{2\sqrt{2 \pi}}\left(e^{-\frac{|v+1.5|^{2}}{2}}+e^{-\frac{|v-1.5|^{2}}{2}}\right), \\
& \phi_0(x) = \frac{1}{5} \sin (2 \pi x).
\end{align*}
According to $\cite{bouchut1995long, DV01, GMM17}$, $f$ converges exponentially fast to the global equilibrium
\[
f_{\infty}= \frac{2 \sqrt{2 \pi} }{\int_0^1 e^{-\frac{1}{5 } \sin x} \rd x} e^{-\frac{v^2}{2}-\frac{1}{5 } \sin (2 \pi x)}.
\]
To see this, we compute the evolution of the relative entropy
\begin{equation} \label{entropyff}
E(f|f_{\infty}) = \int \int f \log \frac{f}{f_{\infty}} \rd v \rd x
\end{equation}
and display the results in Fig.~\ref{fig:entropy}. As shown, the relative entropy decays in time with an exponential rate at the beginning. This decay, however, is flattened at around $10^{-2}$, which indicates a discrepancy between $f$ and $f_\infty$.
On the right figure of Fig.~\ref{fig:entropy}, we see that this discrepancy decays with finer grids, which implies that our scheme does not preserve the global equilibrium exactly, but only up to some numerical error. 

\begin{figure}[!ht]
  \centering
  \includegraphics[width=0.45\textwidth]{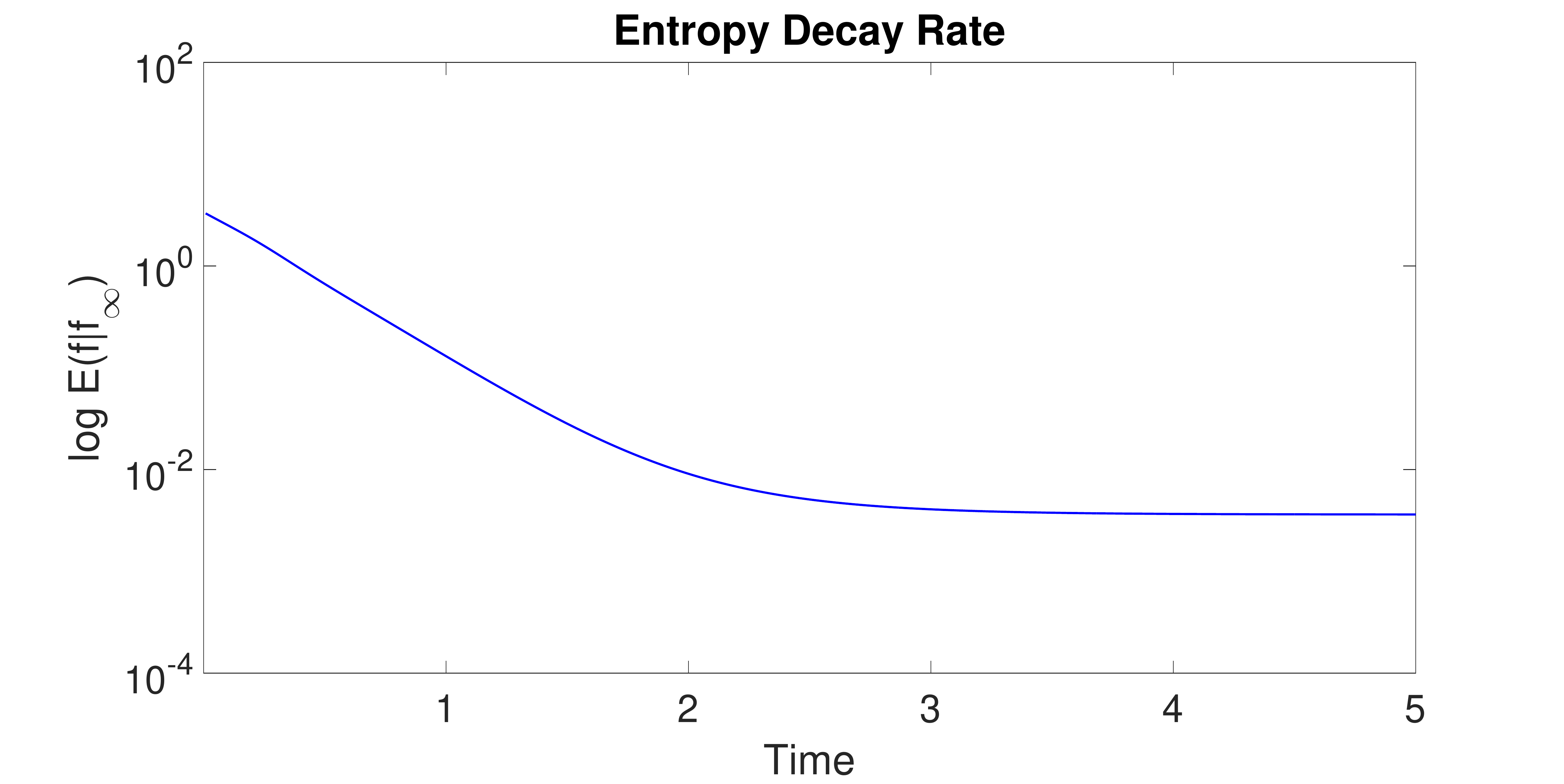}
  \centering
  \includegraphics[width=0.45\textwidth]{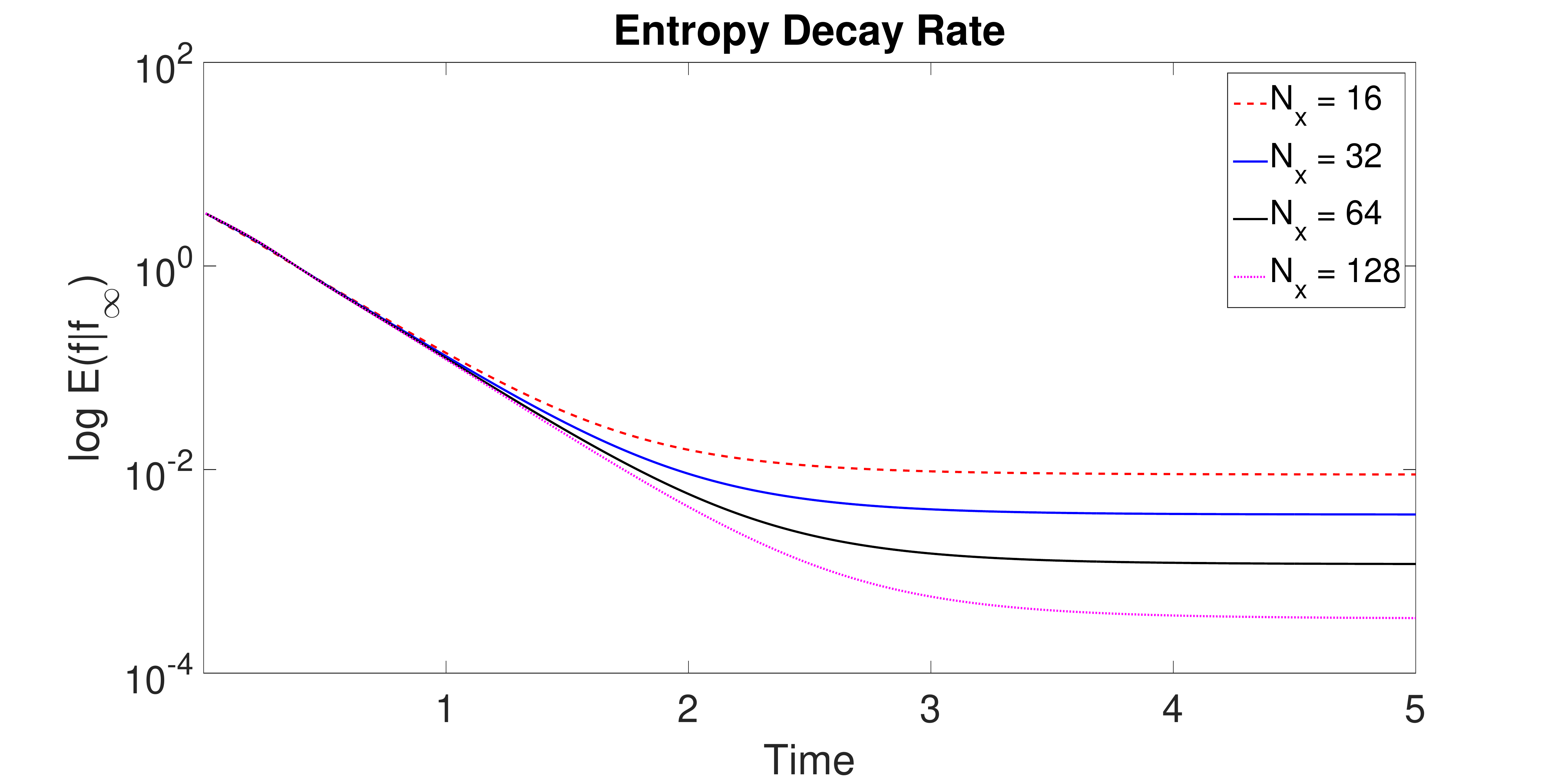}
  \caption{Left: exponential decay of entropy with $N_x = 32$. Right: entropy decay with different $N_x$. Here $x \in [0,1]$, $v \in [-6,6]$, $N_{max} = 1000$, $\Delta x = 1/N_x$, $\Delta v =12/64$, and $\tau = \Delta x /15$.  }
  \label{fig:entropy}
\end{figure}

Next we consider VPFP system $\eqref{vpfp}$ with $\eps =1$ and check the entropy decay. The initial data is taken the same as that in Section \ref{vpfp_ap},and the computational domain is chosen as $x \in [0,1] $ and $v\in [-6,6]$. In this case, we do not have an explicit formula for $f_\infty$, so we compute it numerically by running our scheme for long enough time until it converges to a steady state. Fig.~\ref{fig:entropy_vpfp} then displays the exponential decay of the relative entropy \eqref{entropyff}, as partially predicted in \cite{MN06}. 
\begin{figure}[!ht]
  \centering
  \includegraphics[width=0.6\textwidth]{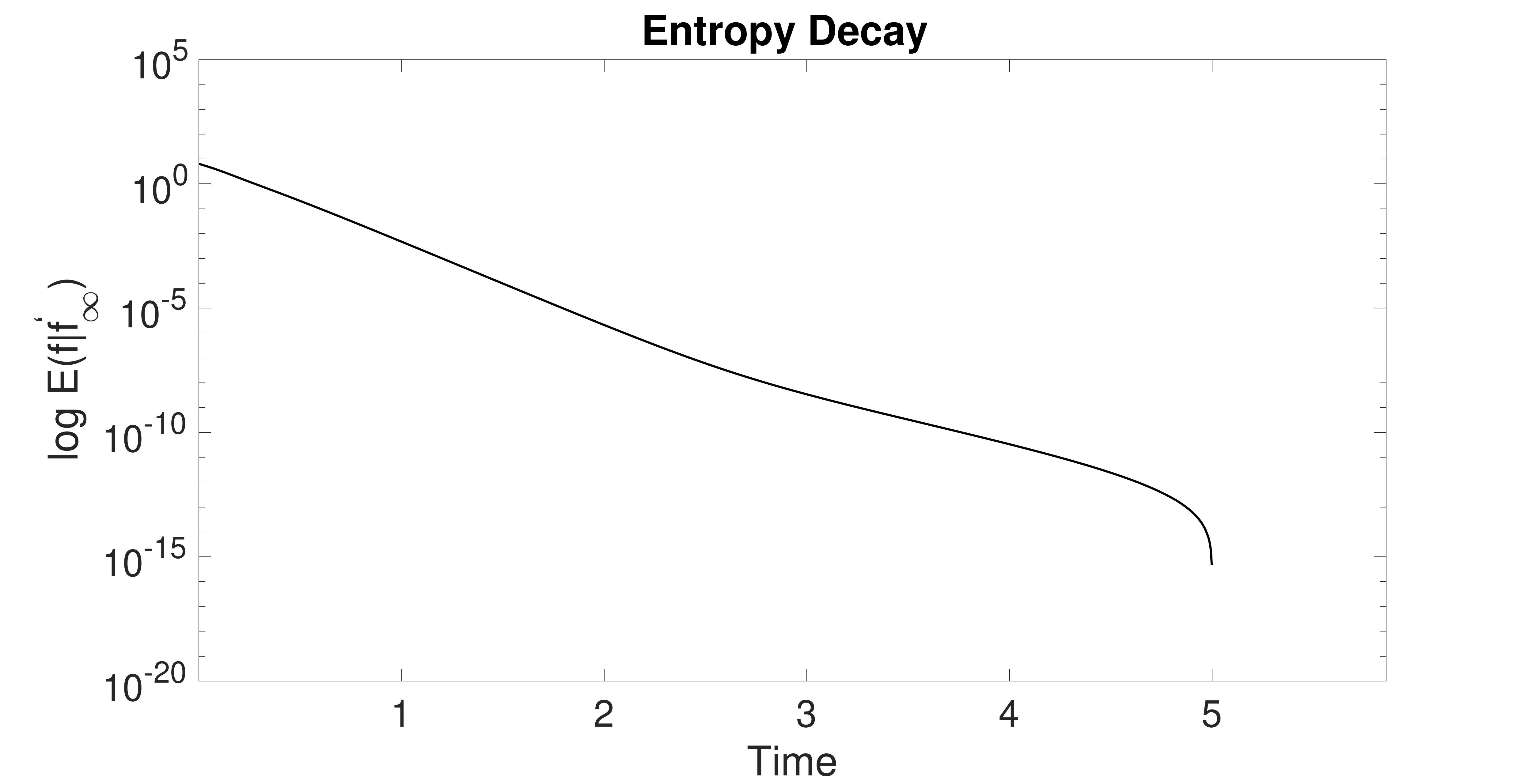}
  \centering
  \caption{Exponential decay of the relative entropy $E(f|f_\infty)$ with $N_{max} =1000$, $\Delta x =1/32$, $\Delta  v =  12/64$ and $\tau = \Delta x/16$. $f_\infty$ is computed at $t=5$.}
  \label{fig:entropy_vpfp}
\end{figure}

\subsubsection{Mixing regime}
In this section, we test the performance of our scheme when $\eps$ has a mixing magnitude:
\[
\epsilon(x)=\left\{\begin{array}{ll}{\epsilon_{0}+\frac{1}{2}(\tanh (5-10 x)+\tanh (5+10 x))} & {x \leq 0.3} ,\\ {\epsilon_{0}} & {x>0.3},\end{array}\right.
\]
with $\eps_0=10^{-3}$. The initial condition is chosen as:
\[
\rho^{0}(x)=\frac{\sqrt{2 \pi}}{6}(2+\sin ( \pi x)), \quad f^{0}(x, v)=\frac{\rho(x)}{\sqrt{2 \pi}} e^{-\frac{\left|v+\phi_{x}^{0}\right|^{2}}{2}},
\quad h(x)=\frac{1.6711}{2.5321} e^{\cos (\pi x)}.
\]
In Fig.~\ref{fig:mix}, we plot the shape of the solution at two different times $t=0.2$ and $t=0.3$, and compare our solution with the reference solution obtained by explicit solver, which uses the second order Runge-Kutta discretization in time and MUSCL scheme for space discretization. Here a good agreement between two solutions is observed, which confirms the efficiency of our method.   
\begin{figure}[!ht]
  \centering
  \includegraphics[width=0.48\textwidth]{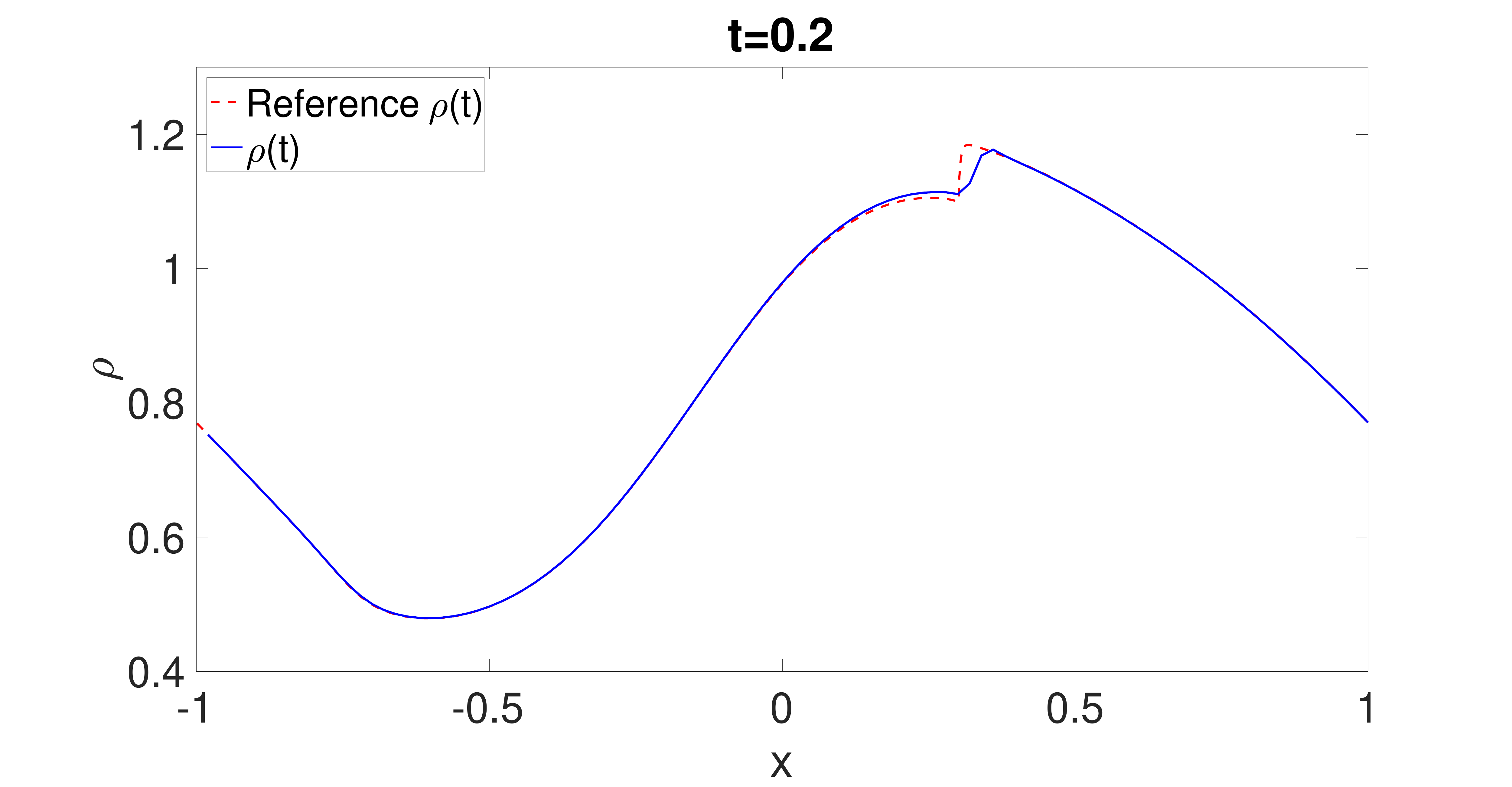}
  \includegraphics[width=0.48\textwidth]{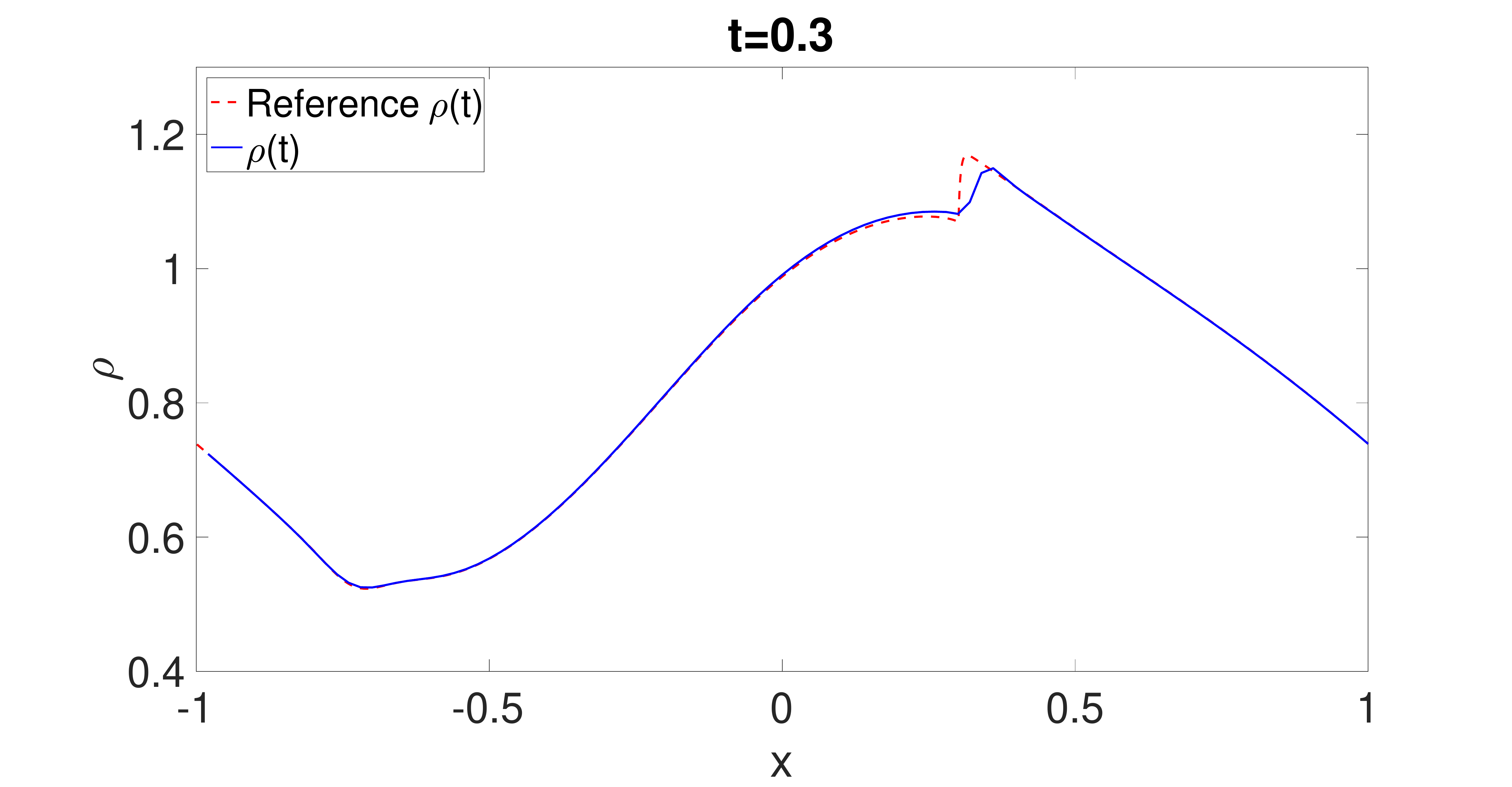}
  \caption{Comparison of our solution with the reference solution obtained by explicit solver. Here we use $x \in [0,1]$, $v \in [-6,6]$, $N_x = 100$, $\Delta x =2/N_x $, $\Delta v=12/64$, and $\tau = \Delta x/15 $ for our method. Use $N_x = 2000$, $\Delta x =2/N_x$, $\Delta v=12/64$, $\tau = \min \{\frac{\Delta x}{\max |v|}, \eps_0 \Delta x, \eps_0 \Delta v^2\}/5 = 7.0313e-6$ for the explicit reference solver. }
  \label{fig:mix}
\end{figure}

\subsection{2D in velocity}
\subsubsection{Convergence rate}
For the two dimensional case, we start again by checking the convergence of our proximal quasi-Newton method to the spatially homogeneous case with varying $\eps$. Here we consider the initial condition with four bumps: 
\[
f_0(v) = e^{(v_2-1)^2-(v_1-1)^2}+\frac{1}{\pi}e^{(v_2+1)^2-(v_1+1)^2}+\frac{2}{\pi}e^{(v_2-1)^2-(v_1+1)^2}+\frac{4}{\pi}e^{(v_2+1)^2-(v_1-1)^2},
\]
where $v \in [-5,5] \times [-5,5]$. In Fig.~\ref{fig:PN_2D_conver_search}, we compute the relative error \eqref{errork} with respect to $k$, where $u^*$ is obtained by running the same algorithm with 110 iterations. 
\begin{figure}[!ht]
    \centering
    \includegraphics[width=0.55\textwidth]{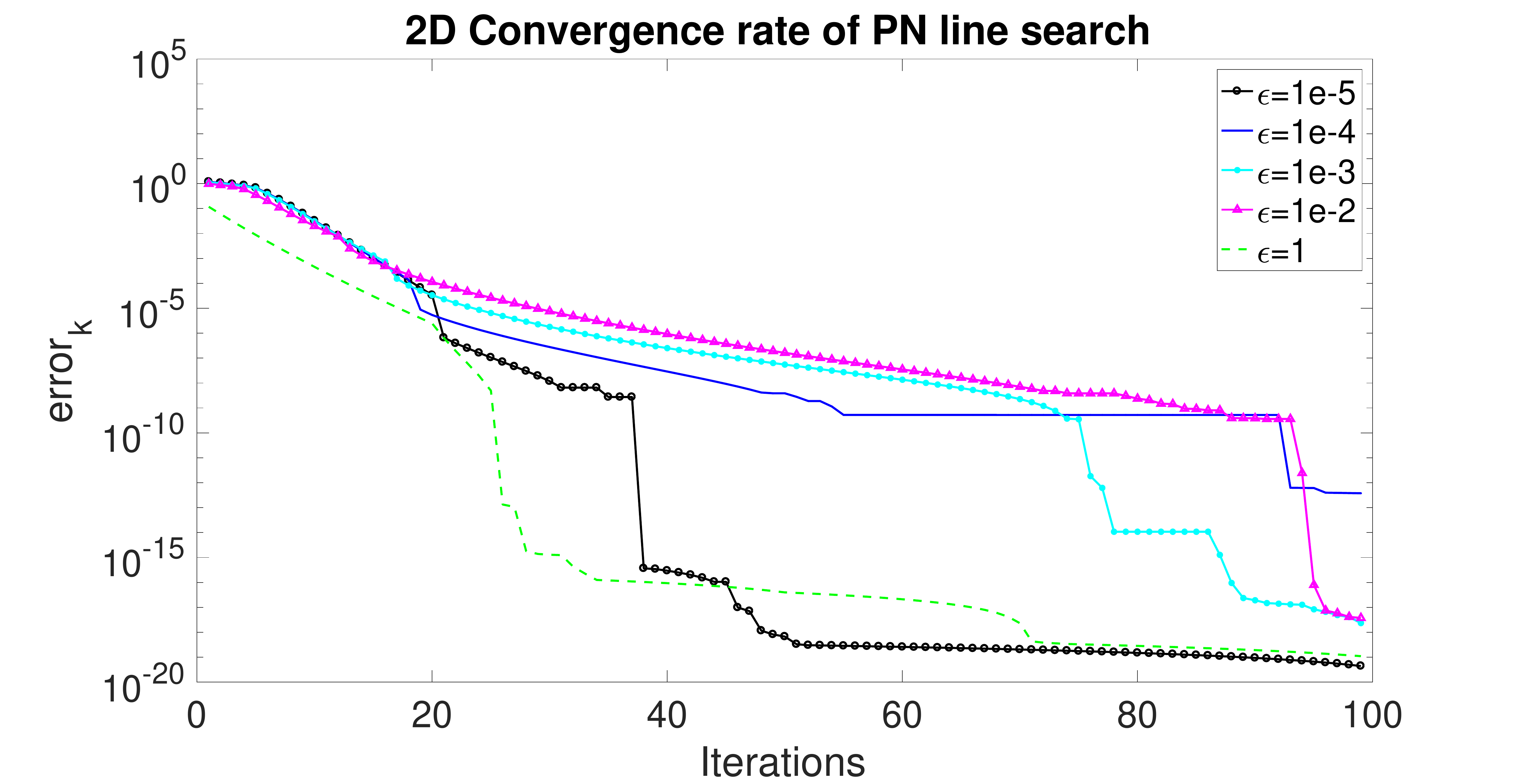}
    \caption{Convergence of Algorithm 2 with varying $\eps$. Here $N_v=40$, $\tau =0.05$.}
    \label{fig:PN_2D_conver_search}
\end{figure}

\subsubsection{Evolution of two semi-torus like initial condition}
In this section, we plot the evolution of VPFP system using Algorithm 2 with the following initial condition: 
\begin{align*}
f_0(v) =~& 1.5\left(1+\left( \sqrt{(v_1-2)^2+(v_2-2)^2}-2\right)^2\right)^{-10} \\
&+2\left(1+\left( \sqrt{(v_1+2)^2+(v_2+2)^2}-2\right)^2\right)^{-10}
\end{align*}
in Fig.~\ref{fig:2D00}. An evolving to the equilibrium and exponential convergence in entropy is observed. 

\begin{figure}[!ht]
{\includegraphics[width=0.45\textwidth]{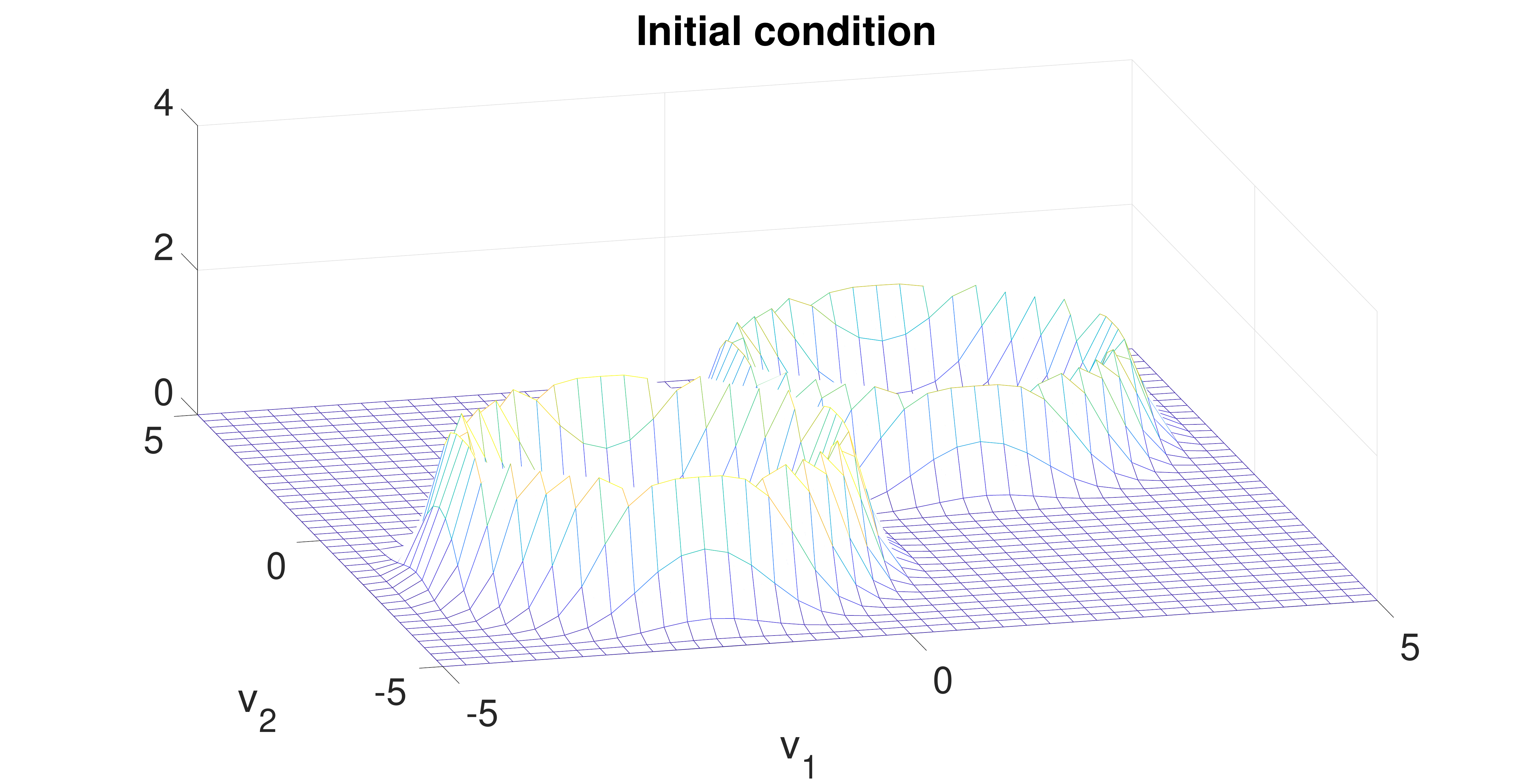}}
{\includegraphics[width=0.45\textwidth]{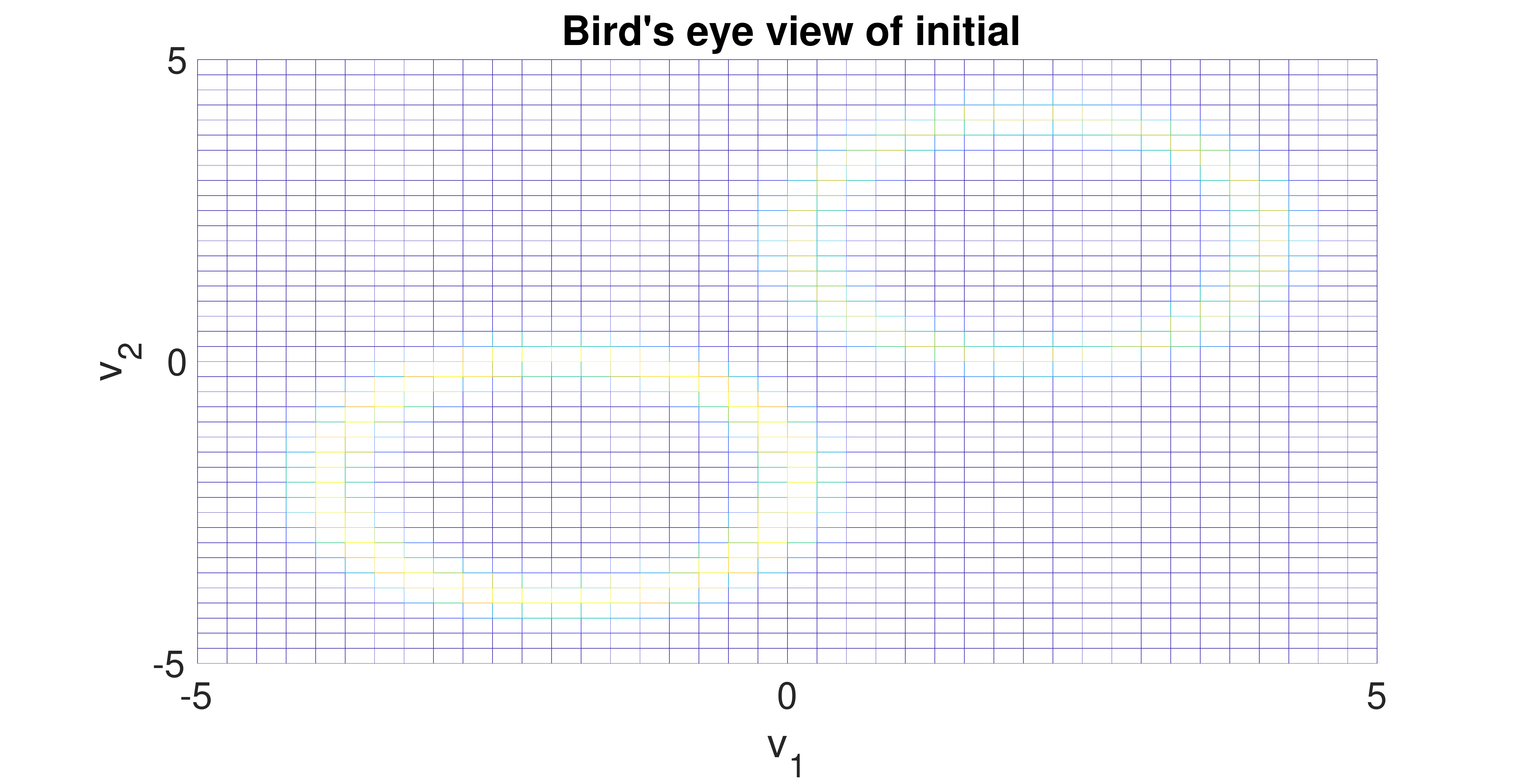}}
  \vfill
{\includegraphics[width=0.45\textwidth]{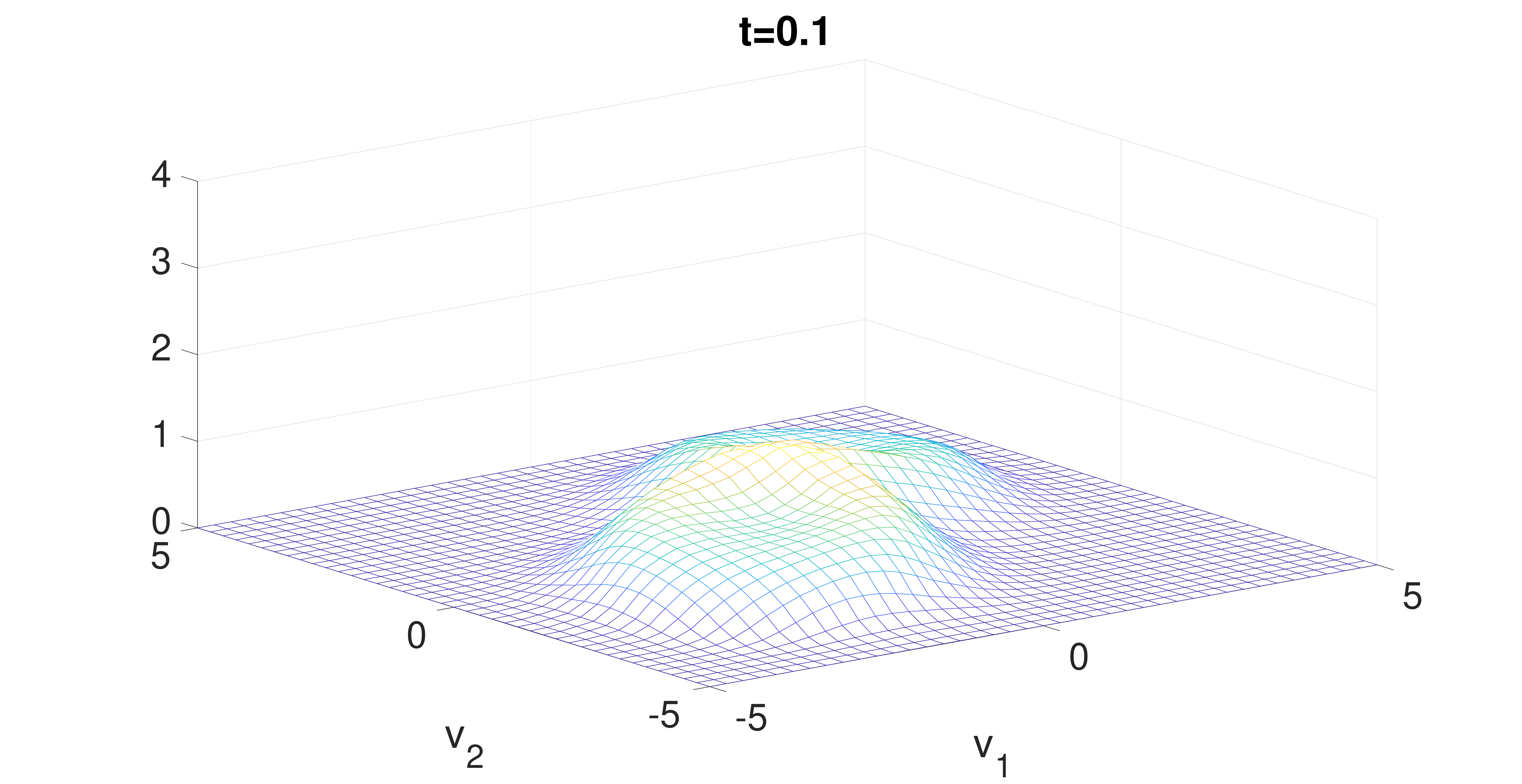}}
{\includegraphics[width=0.45\textwidth]{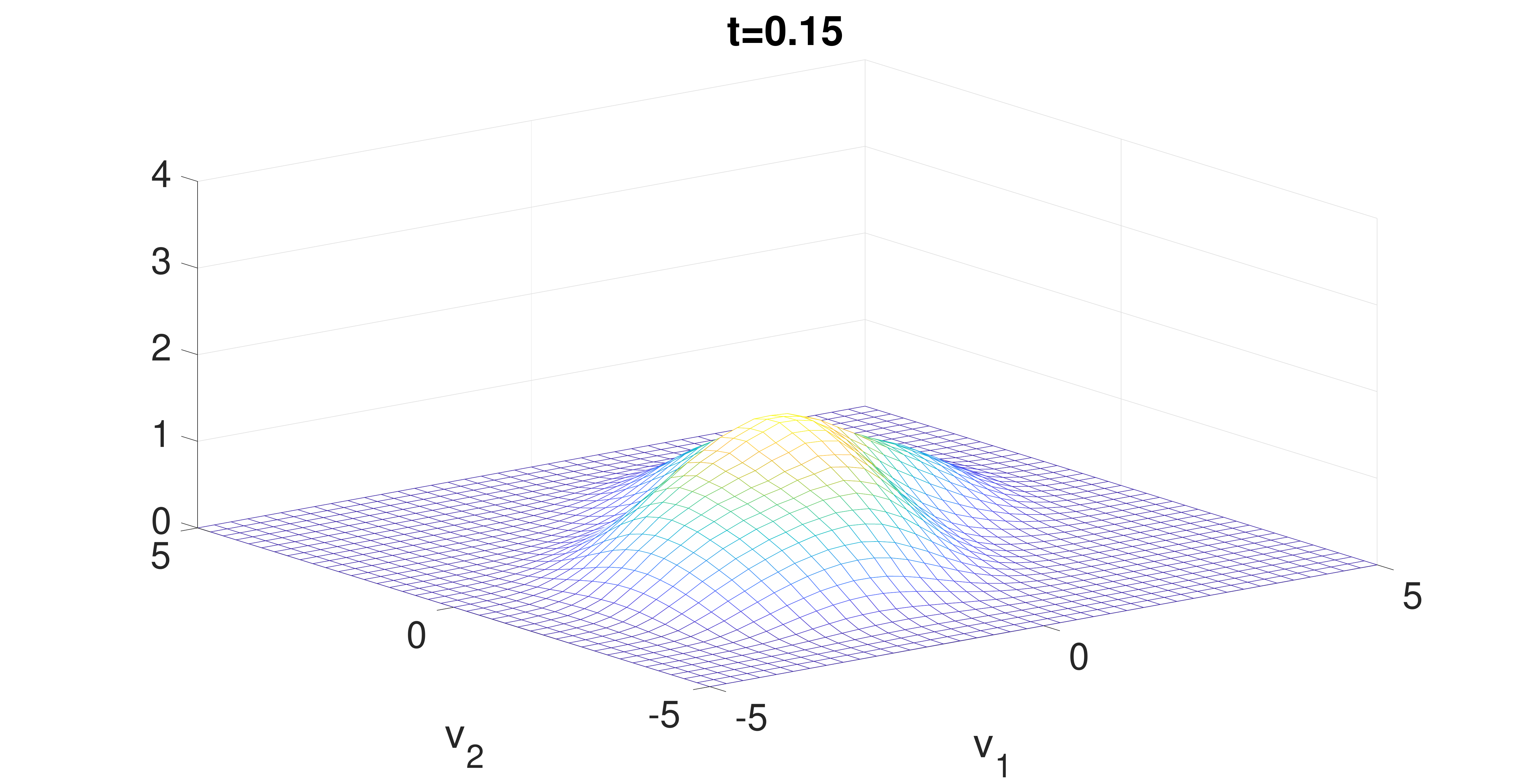}}
  \vfill
{\includegraphics[width=0.45\textwidth]{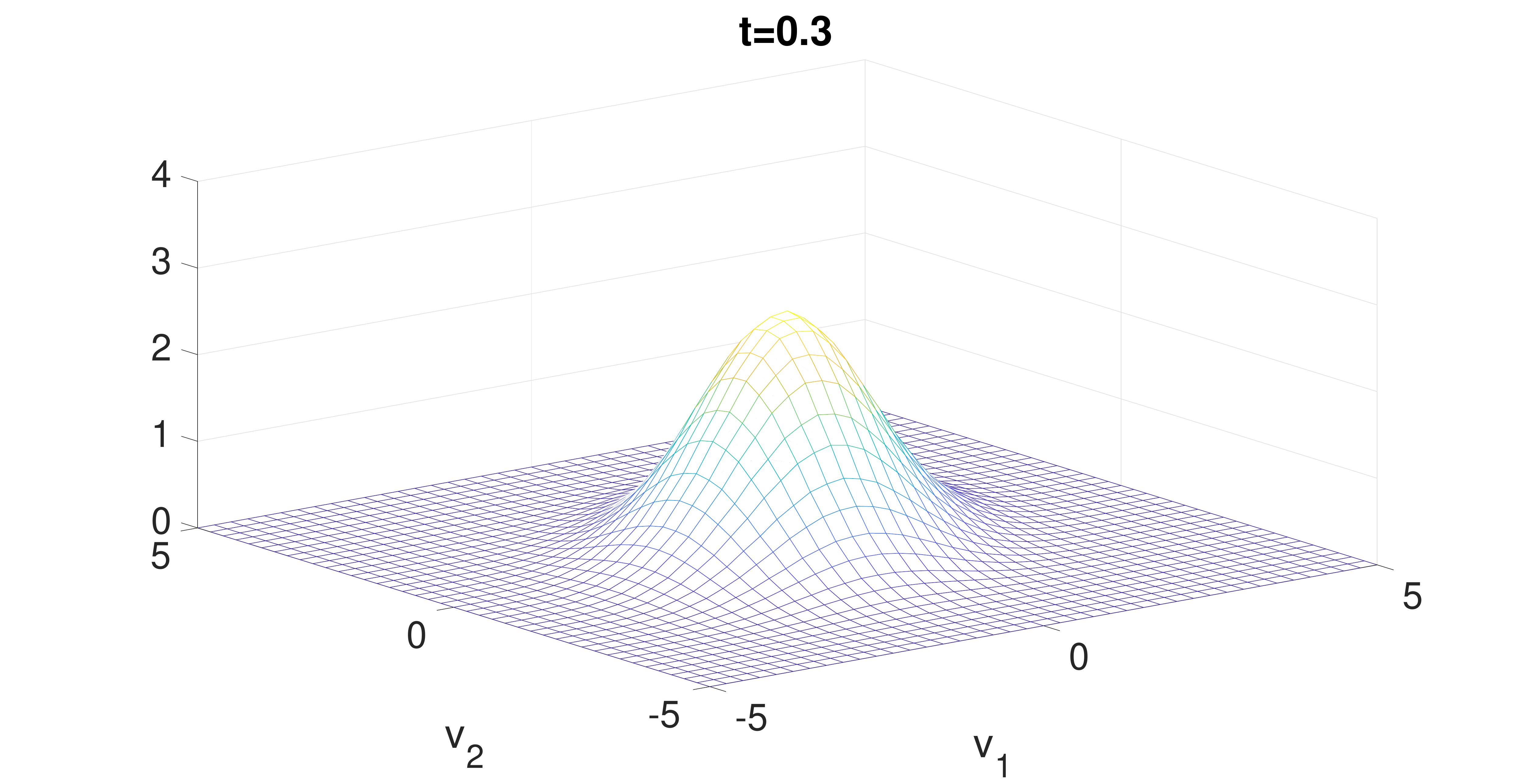}}
{\includegraphics[width=0.45\textwidth]{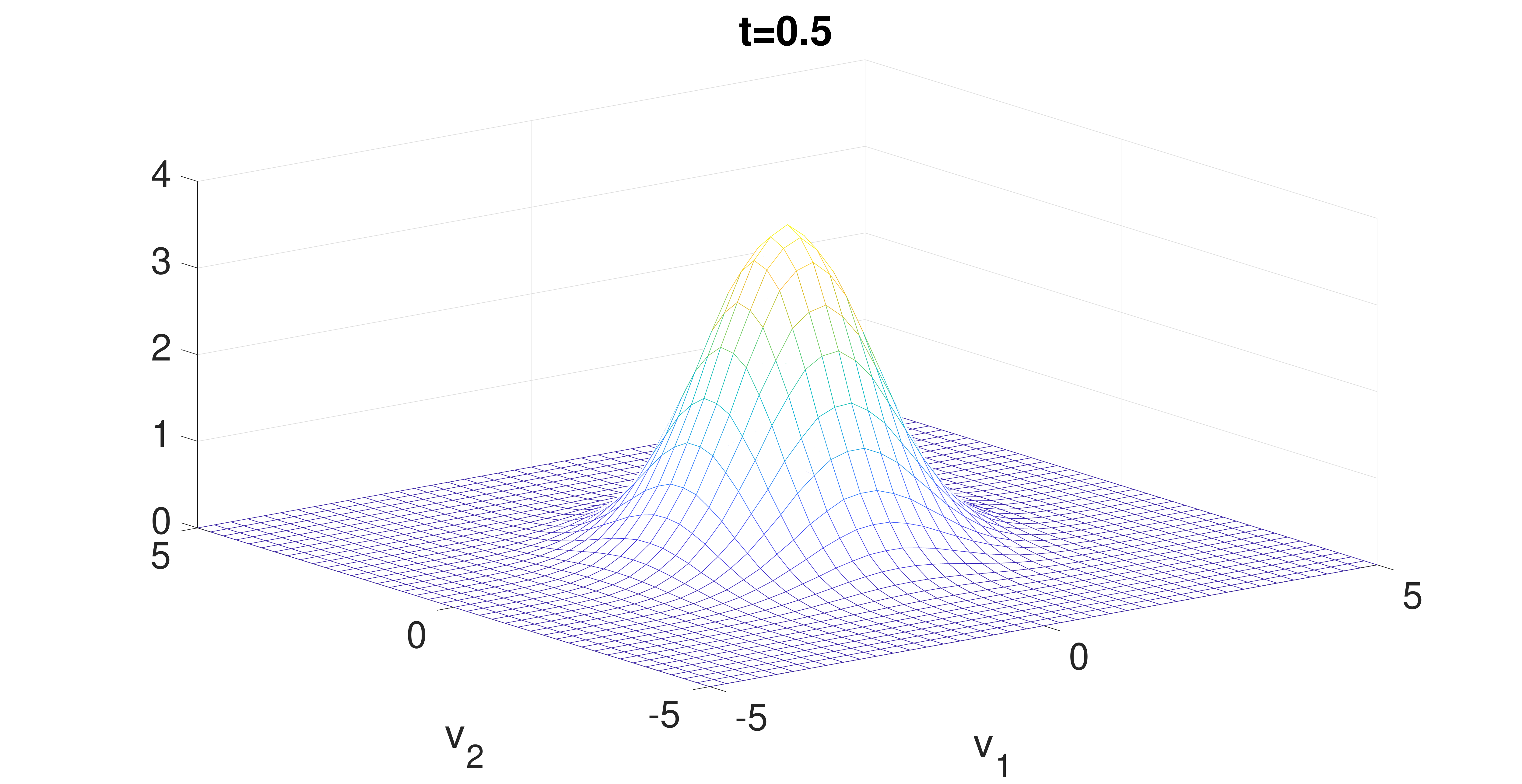}}
  \vfill
{\includegraphics[width=0.45\textwidth]{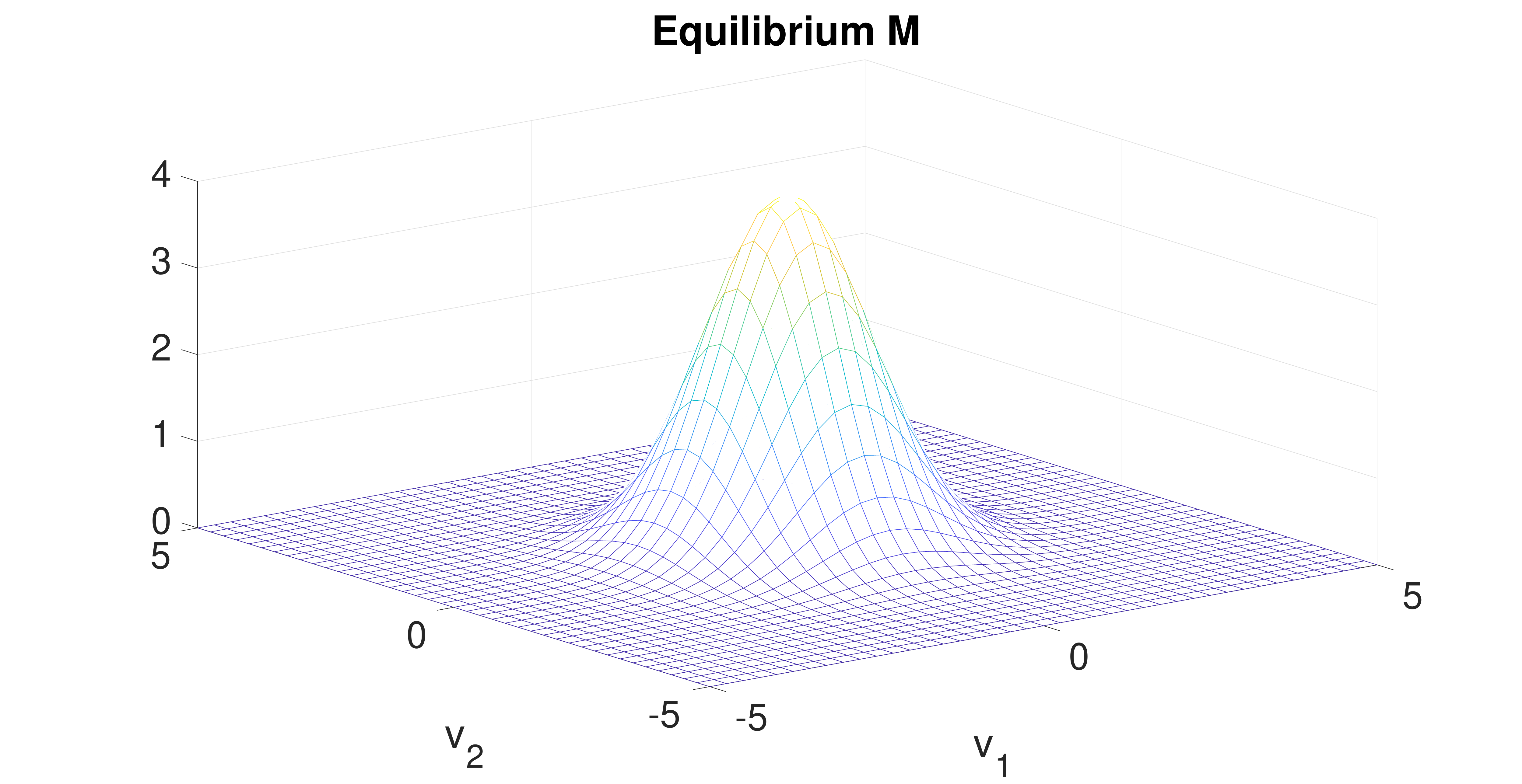}}
{\includegraphics[width=0.45\textwidth]{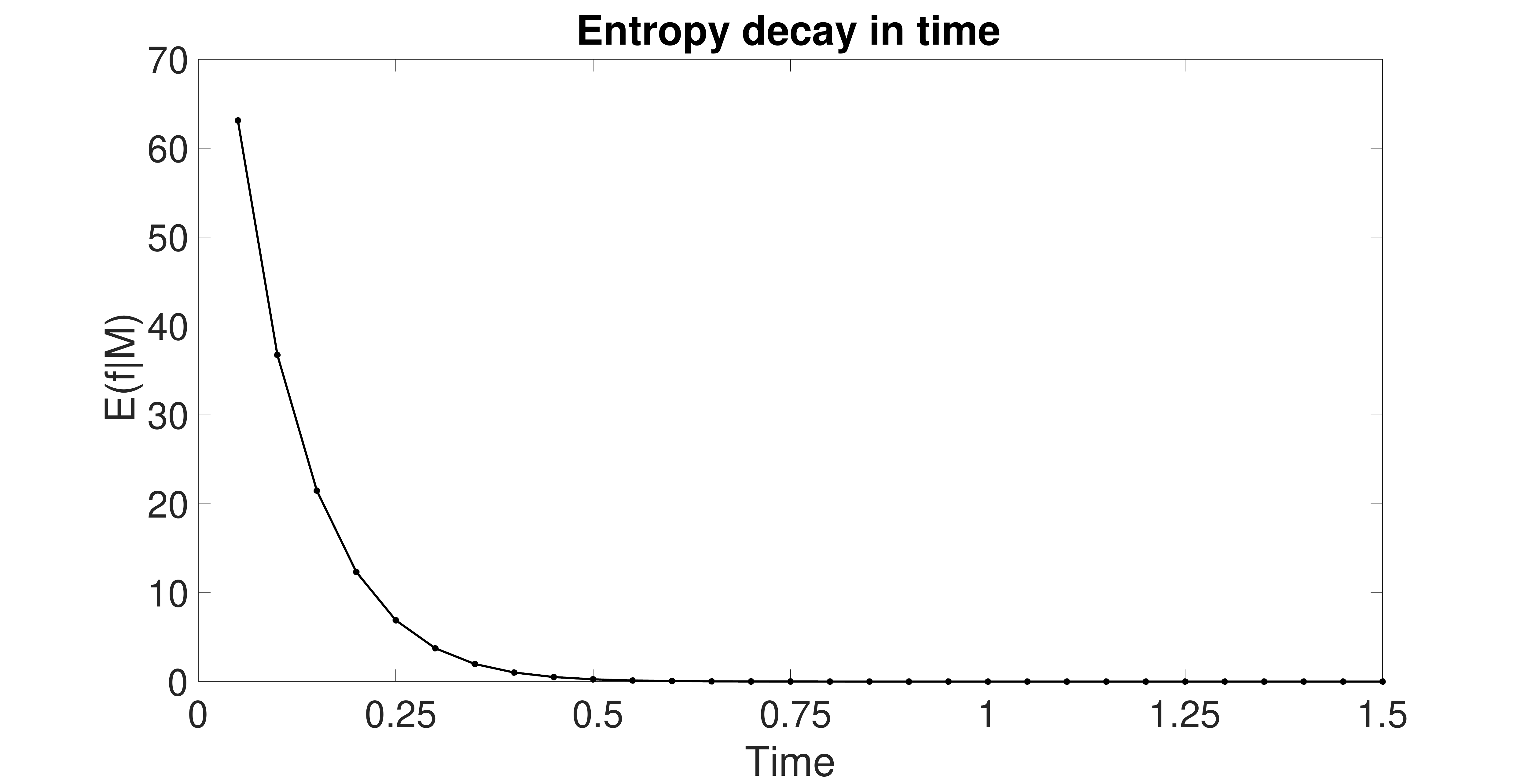}}
 \caption{Evolution of $f(t,v)$. Top two are initial states in different viewpoints; second and third rows are the evolution along time; bottom are equilibrium and evolution of entropy in time. Here we use $v \in [-5,5] \times [-5,5]$, $\tau = 0.05$, $N_{max} = 1000$, $\Delta v = 0.25$, $\delta = 10^{-7}$ and $\eps = 0.2$.}
 \label{fig:2D00}
\end{figure}

\subsubsection{The asymptotic preserving property}
Consider $1d_x \times 2d_v$ VPFP system with initial condition 
\[
f_0(v) = \frac{\rho^0(x)}{4\pi}[ e^{{(v_2-2)^2-(v_1-2)^2}}+e^{(v_2+2)^2-(v_1+2)^2}+e^{(v_2-2)^2-(v_1+2)^2}+e^{(v_2+2)^2-(v_1-2)^2}]\,,
\]
where $v \in [-5,5] \times [-5,5]$ and 
\[
\rho^{0}(x)=\frac{\sqrt{2 \pi}}{2}(2+\cos (2 \pi x)), \quad h(x)=\frac{5.0132}{1.2661} e^{\cos (2 \pi x)}, \quad x \in (0,1)\,.
\]
As in the one dimensional case, we compute the $l_1$ distance between our solution and the local equilibrium at each time $t_n$ as
\[
\|f^n-M^n\|_1 = \sum_{i,j,k} |f^n(x_i,v_j,v_k)-M_i^n(v_j,v_k)|  \Delta x \Delta v^2 .
\]
In Fig.~\ref{fig:AP2D}, we again observe an order $\mathcal O(\eps)$ distance with decreasing $\eps$, which indicate the asymptotic preserving property of our scheme. 
\begin{figure}[!ht]
    \centering
    \includegraphics[width=0.55\textwidth]{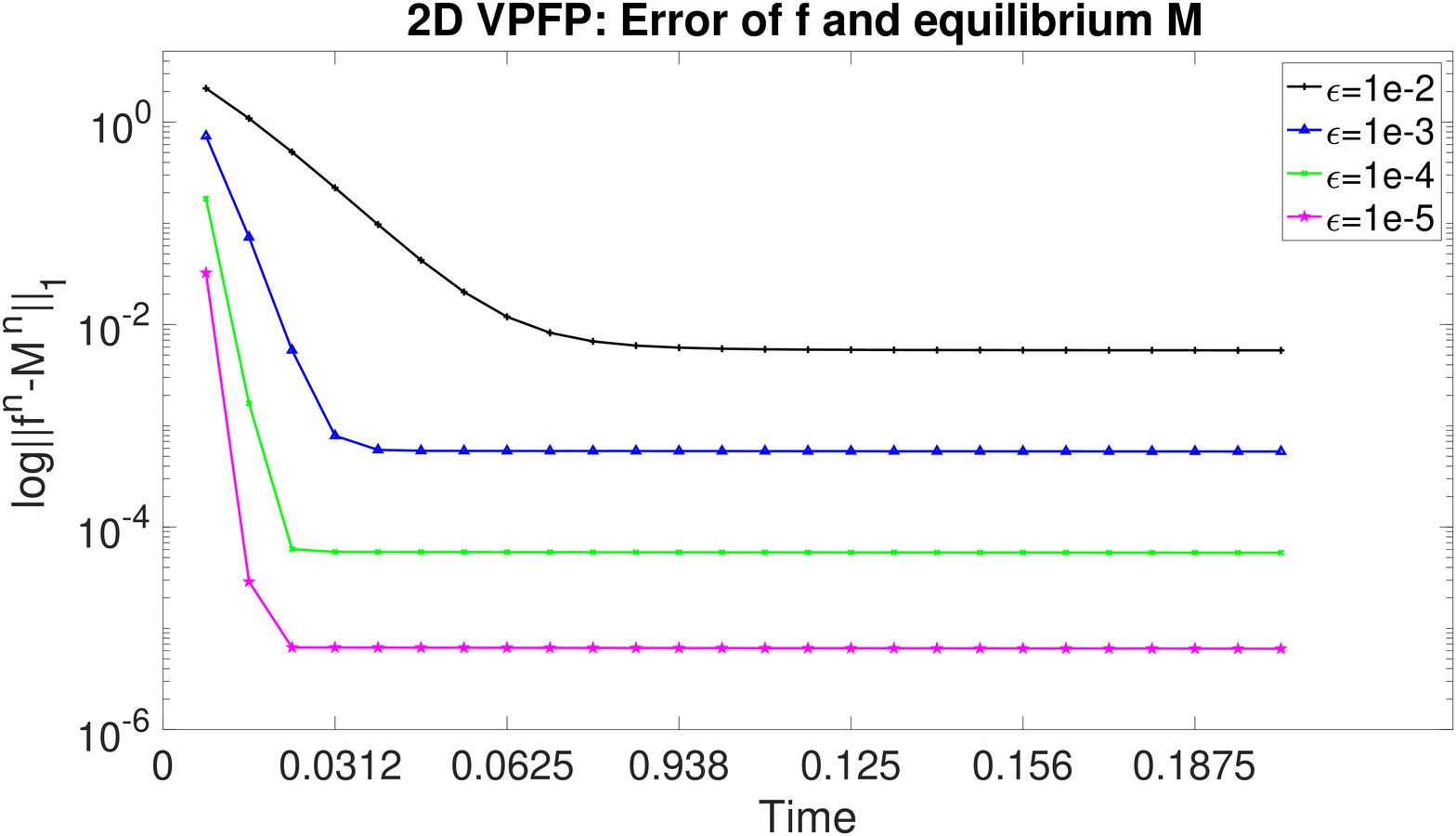}
    \caption{Evolution of the distance between our solution $f$ and local equilibrium $M$ with decreasing $\eps$. Here $\Delta x =1/16 $, $\Delta v=10/40$, $N_{max} = 1000$, and $\tau = 0.0078 $. }
    \label{fig:AP2D}
\end{figure}

\subsection{3D in velocity}
At last, we'd like to emphasize that our scheme can be easily extended to higher dimensions due to its passive parallelizability. To this end, we consider one example in three dimensions. The initial data is taken to be 
\[
f_0(v_1,v_2,v_3)=(2\pi)^{-3/2}(e^{-(v_1-1)^2-(v_2+1)^2-v_3^2/2}+e^{-(v_1+1)^2-(v_2-1)^2-v_3^2/2})\,,
\]
as displayed in Fig.~\ref{fig:3D1}. The computational domain is $v \in [-L,L] \times [-L,L] \times [-L,L] $ with $L=4$ and it is partitioned into 16 cells in each direction, i.e., $\Delta v = 0.5$. We take $\eps = 0.2$, $\Delta t = 0.05$, and $N_{max} = 1000$. Fig.~\ref{fig:3D2} gives the evolution of the initial profile towards the equilibrium and also the decay of entropy. 

\begin{figure}[!ht]
  \includegraphics[width=0.45\textwidth]{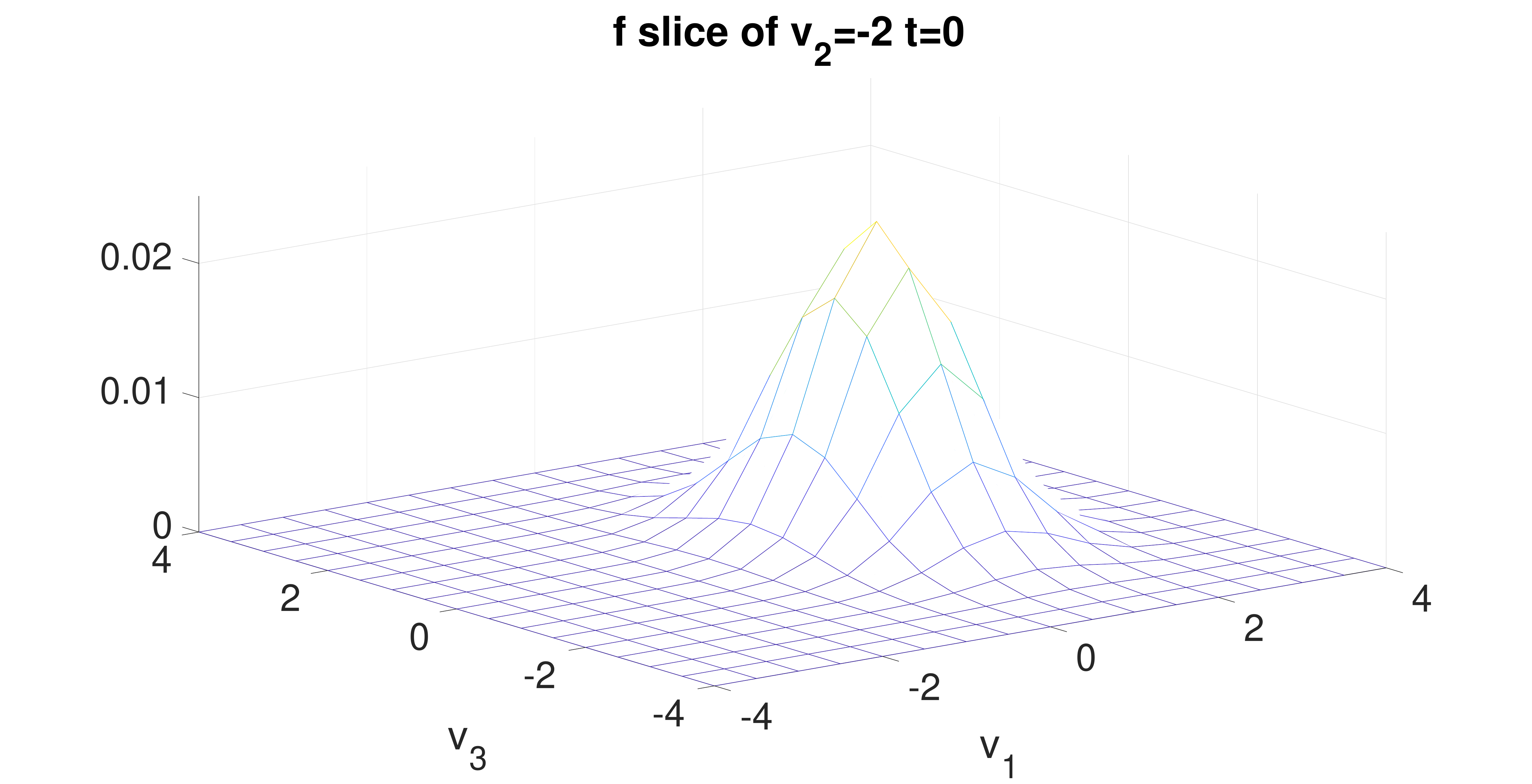}
  \includegraphics[width=0.45\textwidth]{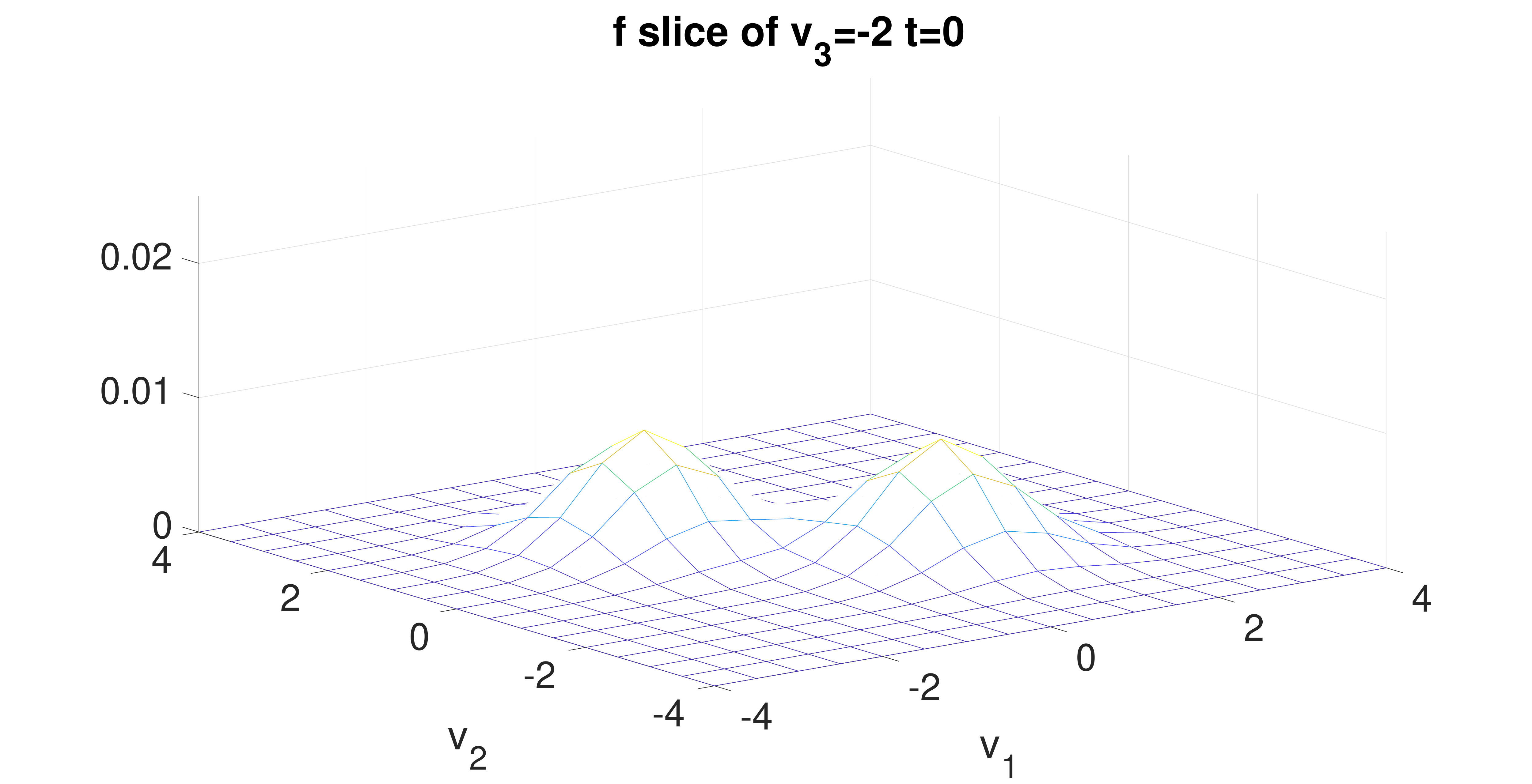}
  \caption{Initial  state of $f(0,v_1,-2,v_3)$ and $f(0,v_1,v_2,-2)$.}
  \label{fig:3D1}
\end{figure}

\begin{figure}[!ht]
{\includegraphics[width=0.45\textwidth]{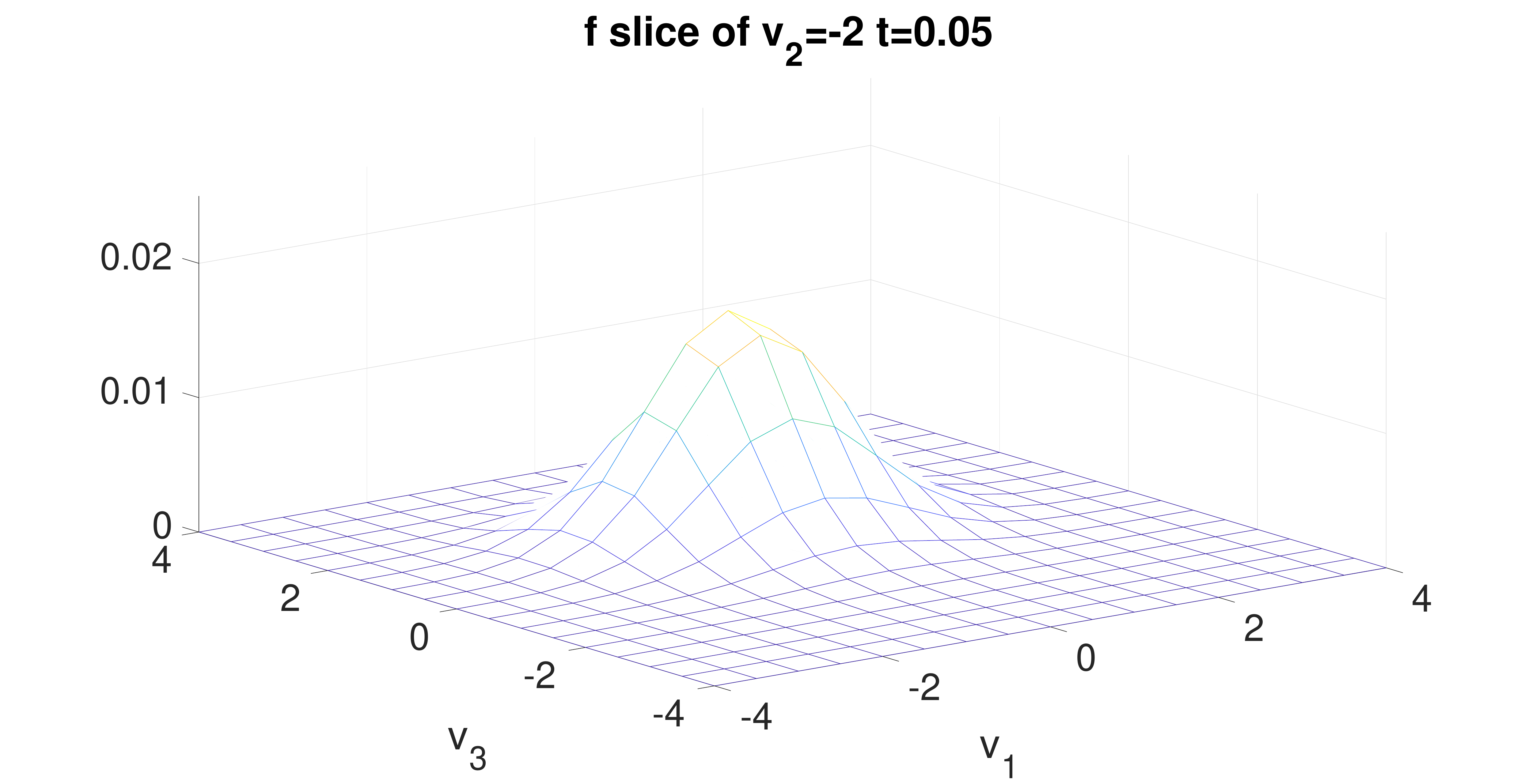}}
{\includegraphics[width=0.45\textwidth]{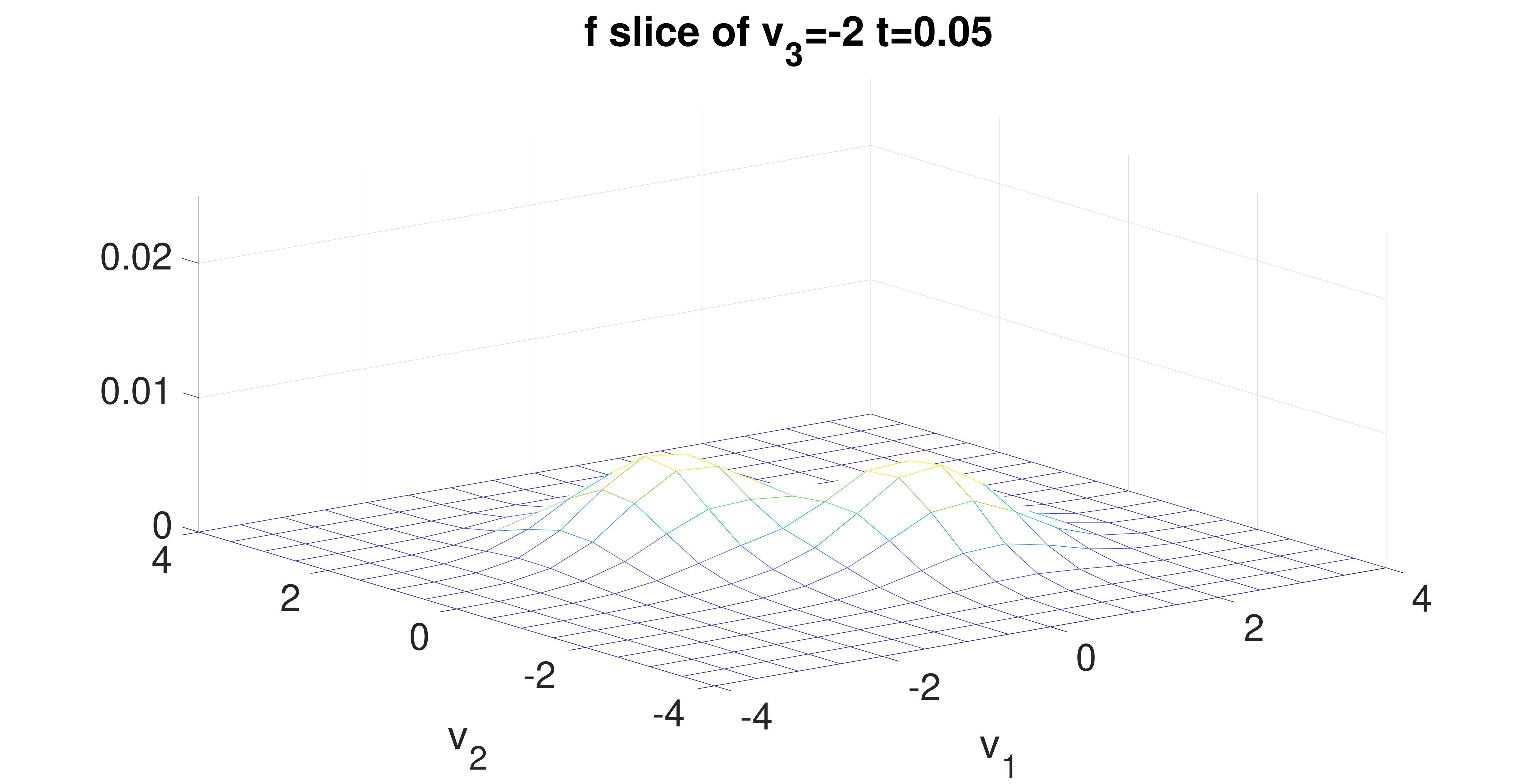}}
  \vfill
{\includegraphics[width=0.45\textwidth]{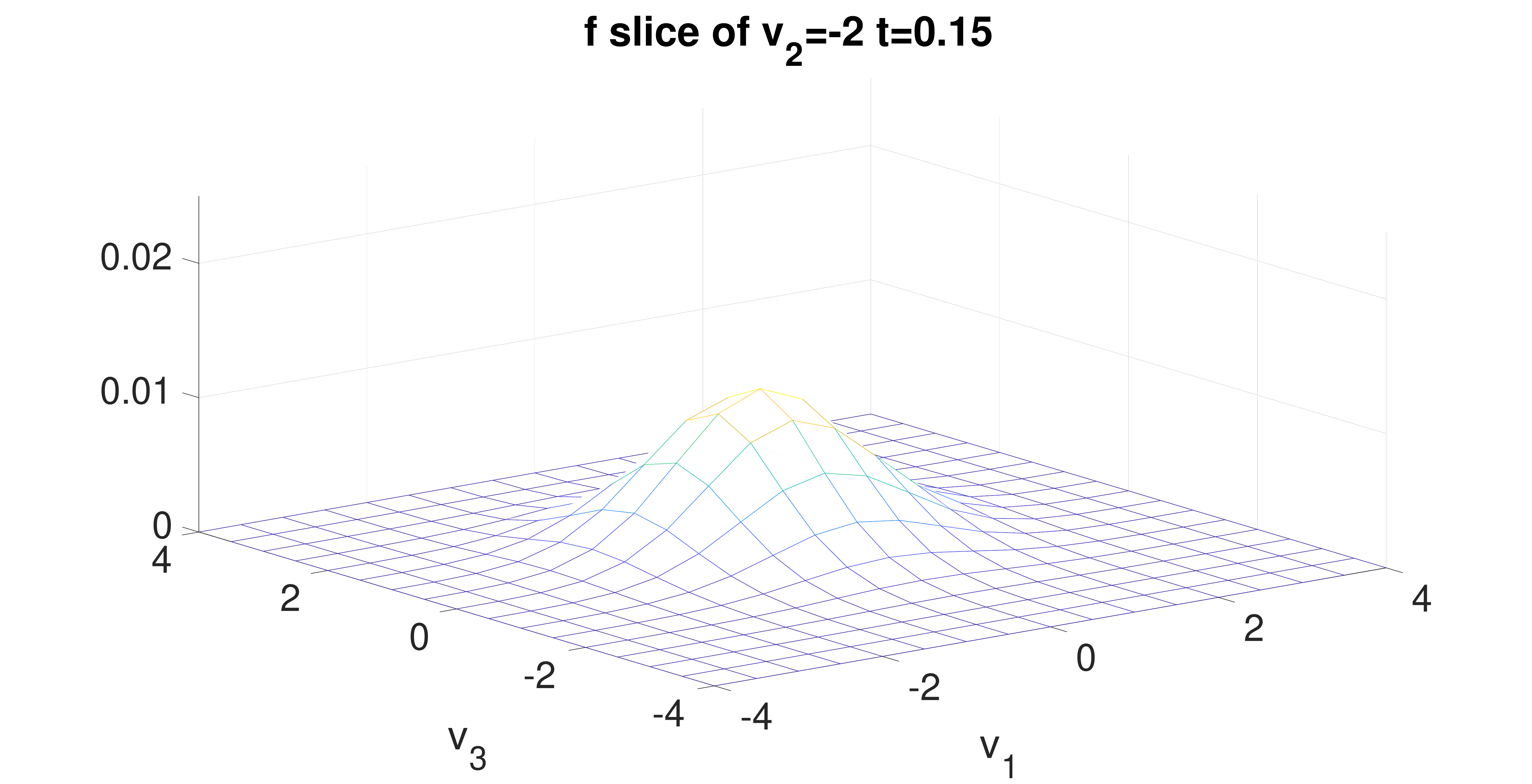}}
{\includegraphics[width=0.45\textwidth]{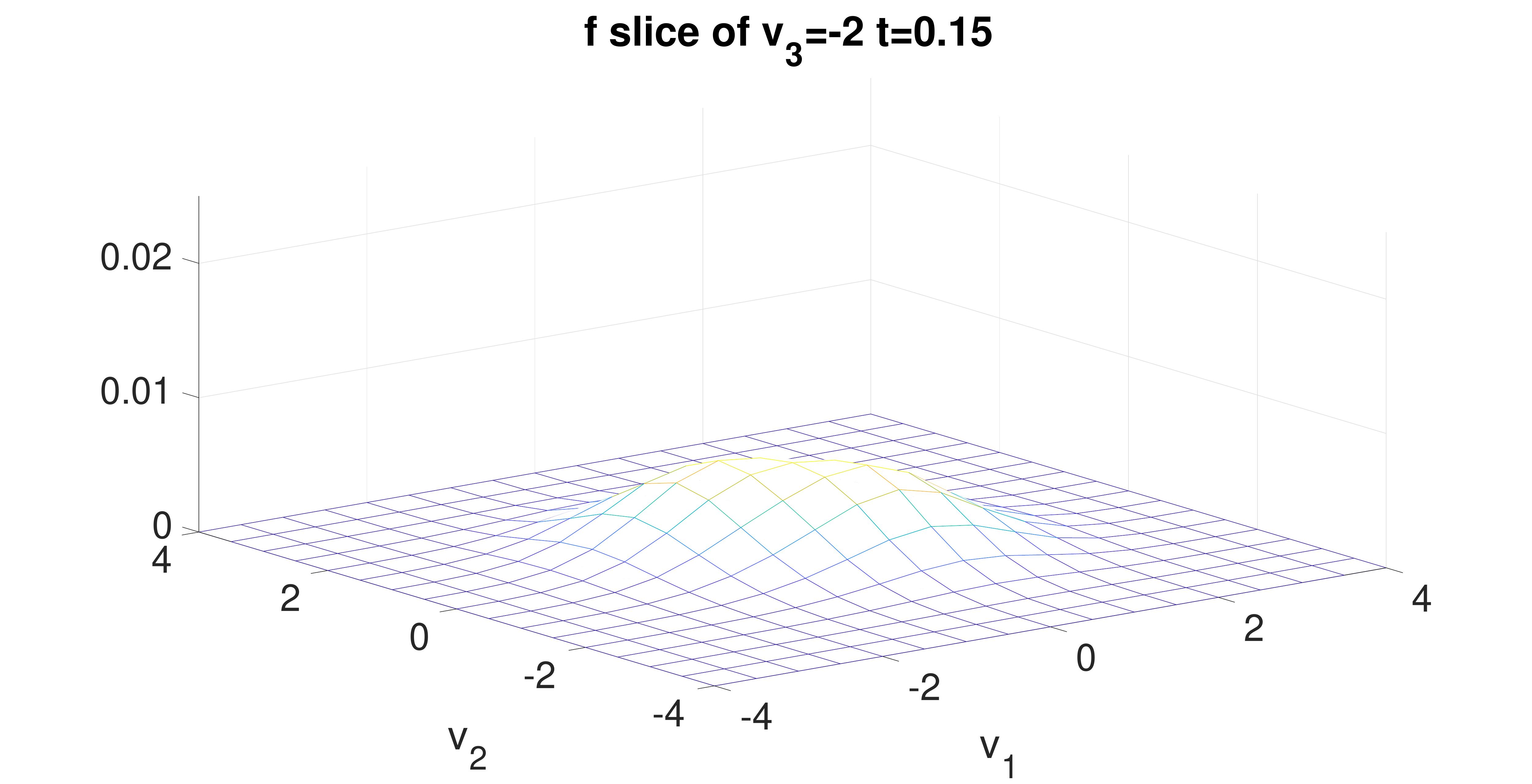}}
  \vfill
{\includegraphics[width=0.45\textwidth]{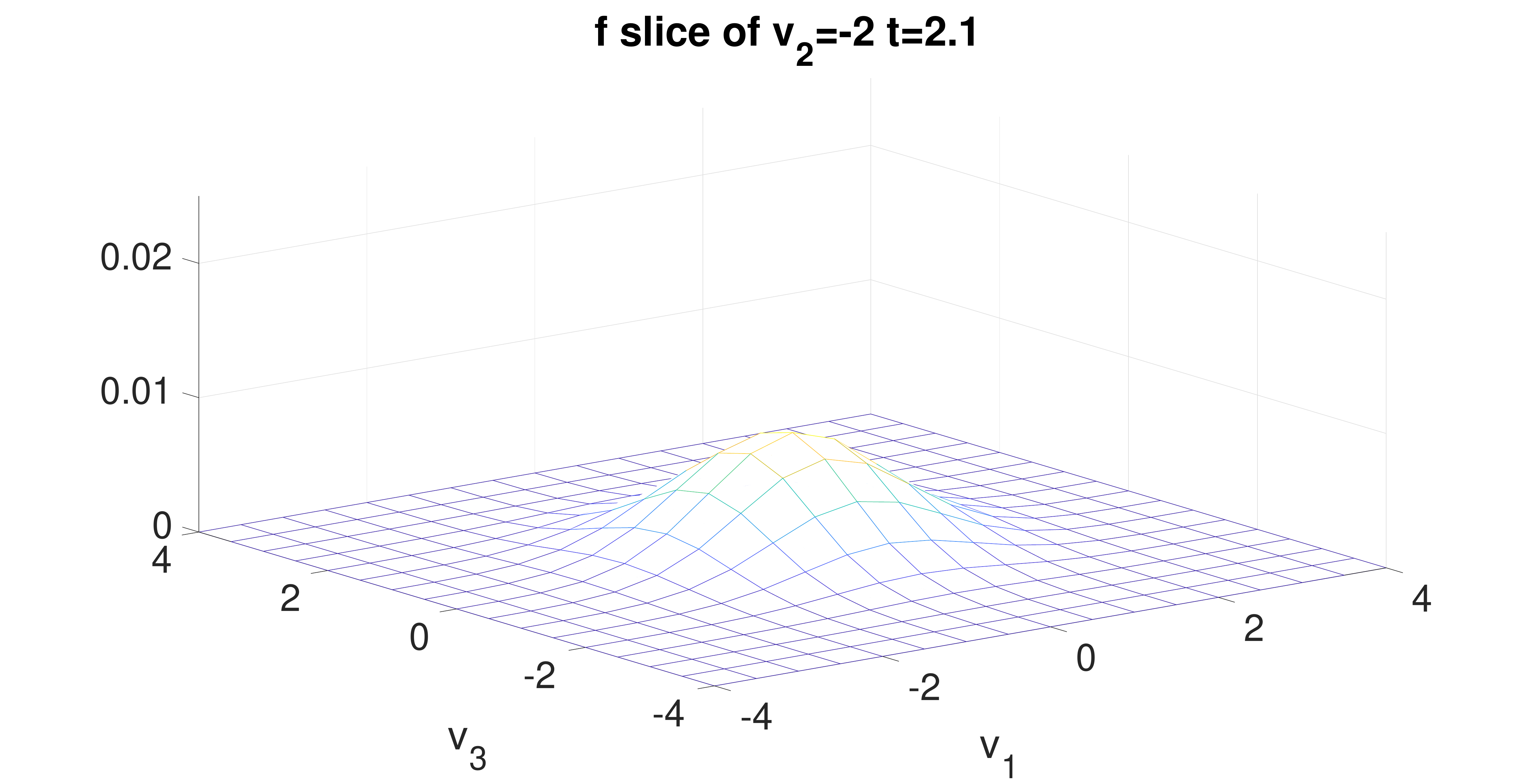}}
{\includegraphics[width=0.45\textwidth]{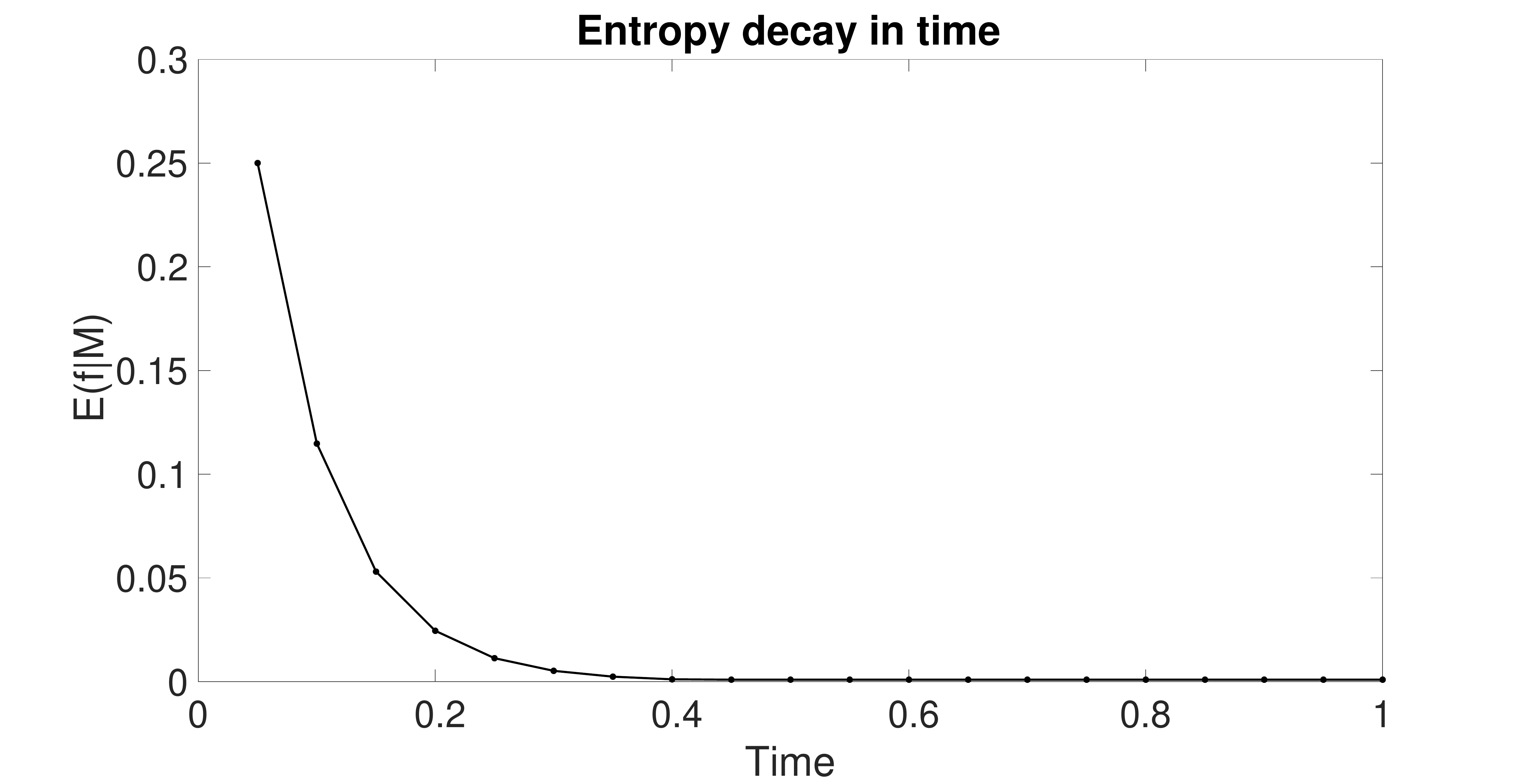}}
 \caption{Evolution of $f(t,v_1,-2,v_3)$ and $f(t,v_1,v_2,-2)$. Top two are $t=0.05$, middle two are $t=0.15$, bottom left is $f(2.1,v_1,-2,v_3)$ and bottom right is entropy decay in time.}
 \label{fig:3D2}
\end{figure}

\section{Conclusion and discussion}
In this paper, we propose an asymptotic preserving (AP) scheme for the VPFP system with high field scaling. The scheme falls into the category of implicit-explicit methods, which is often adopted in designing AP schemes. The major contribution, however, is the treatment of the implicit part, for which we use a variational formulation. Therefore, instead of directly inverting the implicit system, we solve an minimization problem. The minimizer then automatically conserves mass and preserves positivity, both of which are desirable features of numerical schemes. More importantly, the implicit system is stiff and often suffers from ill-conditioning, a problem that has been overlooked in the literature. The optimization algorithm we developed, on the contrary, includes a pre-conditioner that comes from the Hessian of the objective function, and therefore enjoys uniform convergence across different scales. Numerical examples also show that this convergence is insensitive to the dimension of the problem, an important property that is desired for high dimensional problems. Furthermore, the massive parallelizability of our scheme also makes it amenable in high dimensions. Although the implicit part of the VPFP system may be solved efficiently with more sophisticated designed algorithm such as multigrid method, our method has a much better generalizability. In fact, the variational formulation offers a natural implicit treatment that also mimics the real physical process (i.e., entropy decrease) for the collision term in many kinetic equations. And we hope that, the variational framework we put forward in this paper, together with the advanced optimization solver, can provide a new class of AP schemes for kinetic equations applicable to high dimensions efficiently.

\medskip 
\noindent{\bf Acknowledgement:} The authors would like to thank the anonymous reviewers for their comments and suggestions. JAC was supported by EPSRC grant number EP/P031587/1 and the Advanced Grant Nonlocal-CPD (Nonlocal PDEs for Complex Particle Dynamics: Phase Transitions, Patterns and Synchronization) of the European Research Council Executive Agency (ERC) under the European Union's Horizon 2020 research and innovation programme (grant agreement No. 883363). LW and WX was partially supported by NSF grant DMS-1903420 and DMS-1846854. MY was partially supported by NSF grant DMS-2012439.
\appendix
\section{Proof of Theorem~\ref{Thm:PNls2}}\label{apdx:1}

\begin{lemma}\label{apdx_lemma1}
Suppose $\HH_{1}$ and $\HH_{2}$ are positive definite matrices with bounded eigenvalues: $m_1 I \preceq \HH_{1} \preceq M_1 I$ and $m_{2} I \preceq \HH_{2} \preceq M_{2} I$. Let $\Delta u_{1}$ and $\Delta u_{2}$ be the search directions generated using $\HH_{1}$ and $\HH_{2}$ respectively:
\[
\begin{array}{l}
\Delta u_{1}=\operatorname{prox}_{\chi}^{\HH_{1}}\left(u-\HH_{1}^{-1} \nabla F(u)\right)-u\,, \\
\Delta u_{2}=\operatorname{prox}_{\chi}^{\HH_{2}}\left(u-\HH_{2}^{-1} \nabla F(u)\right)-u\,.
\end{array}
\]
Then these two search directions satisfy
\[
\| \Delta u_{1} - \Delta u_{2} \|_{\HH_1} \leq 
\| I-\HH_1^{-1/2} \HH_2 \HH_1^{-1/2} \| \|\Delta u_{2}\|_{\HH_1}\,.
\]
\end{lemma}
\begin{proof}
The main part of the proof is similar to that in \cite[proof of Proposition 3.6]{lee2014proximal} with a little alteration except the very last estimate. But we still include the details for completeness. 
    By the definition of search direction
    \[
    \Delta u=\operatorname{prox}_{\chi}^{\HH}\left(u-\HH^{-1} \nabla F(u)\right)-u .
    \]
    we have
    \[
    \HH(\HH^{-1}\nabla F(u) - \Delta u) \in \partial \chi (u + \Delta u)\,,
    \]
    thus
    \[
    \HH \Delta u  \in - \nabla F(u) - \partial \chi (u + \Delta u).
    \]
    Then
    \[
    \begin{aligned}
\Delta u_{1} &=\argmin _{d} (\nabla F(u))^{\intercal} d+(1/2) d^{\intercal} \HH_{1} d+\chi(u+d), \\
\Delta u_{2} &=\argmin _{d} (\nabla F(u))^{\intercal} d+(1/2) d^{\intercal} \HH_{2} d+\chi(u+d).
\end{aligned}
    \]
    which leads to
    \begin{align*}
& (\nabla F(u))^{\intercal} \Delta u_{1}+(1/2)\Delta u_{1}^{\intercal} \HH_{1} \Delta u_{1}+\chi  \left(u+\Delta u_{1}\right) \\ 
\leq & (\nabla F(u))^{\intercal} \Delta u_{2}+(1/2)\Delta u_{2}^{\intercal} \HH_{1} \Delta u_{2}+\chi \left(u+\Delta u_{2}\right)\\
& -(1/2)(\Delta u_1-\Delta u_{2})^{\intercal} \HH_{1} (\Delta u_1-\Delta u_{2}).
\end{align*}
which is equivalent to 
\begin{align*}
& (\nabla F(u))^{\intercal} \Delta u_{1}+\Delta u_{1}^{\intercal} \HH_{1} \Delta u_{1}+\chi  \left(u+\Delta u_{1}\right) \\ 
\leq & (\nabla F(u))^{\intercal} \Delta u_{2}+\Delta u_{1}^{\intercal} \HH_{1} \Delta u_{2}+\chi \left(u+\Delta u_{2}\right).
\end{align*}
Similarly, we have 
     \begin{align*}
(\nabla F(u))^{\intercal} \Delta u_{2}+\Delta u_{2}^{\intercal} \HH_{2} \Delta u_{2}+\chi  \left(u+\Delta u_{2}\right) \\
\quad \leq (\nabla F(u))^{T} \Delta u_{1}+\Delta u_{1}^{\intercal} \HH_{2} \Delta u_{2}+\chi \left(u+\Delta u_{1}\right).
\end{align*}
    Summing up these two inequalities to get
    \[
    (\Delta u_{1})^{\intercal} \HH_{1} \Delta u_{1}-(\Delta u_{1})^{\intercal}\left(\HH_{1}+\HH_{2}\right) \Delta u_{2}+(\Delta u_{2})^{\intercal} \HH_{2} \Delta u_{2} \leq 0.
    \]
    By completing square we have
    \begin{align*}
&(\Delta u_{1})^{\intercal} \HH_{1} \Delta u_{1}-2 (\Delta u_{1})^{\intercal} \HH_{1} \Delta u_{2}+(\Delta u_{2})^{\intercal} \HH_{1} \Delta u_{2} \\
\leq & (\Delta u_{1})^{\intercal}\left(\HH_{2}-\HH_{1}\right) \Delta u_{2}+(\Delta u_{2})^{\intercal}\left(\HH_{1}-\HH_{2}\right) \Delta u_{2}.
\end{align*}
    Consequently, 
    \[
    \begin{aligned}
    \| \Delta u_{1} - \Delta u_{2} \|^2_{\HH_1}
    & \leq (\Delta u_{1}-\Delta u_{2})^\intercal (\HH_2-\HH_1)\Delta u_{2} \\
    & = (\HH_1^{1/2} (\Delta u_{1}-\Delta u_{2}))^\intercal (\HH_1^{-1/2} \HH_2 \HH_1^{-1/2} - I) \HH_1^{1/2} \Delta u_{2} \\
    & \leq  \| \Delta u_{1} - \Delta u_{2} \|_{\HH_1} \| \HH_1^{-1/2} \HH_2 \HH_1^{-1/2} - I \| \|\Delta u_{2}\|_{\HH_1}.
    \end{aligned}
    \]
    which leads to the result.
\end{proof}
\begin{lemma}\label{apdx_lemma2}
Suppose that there exist constants $r,~r',~R',~L_2 >0$ such that $r' I \prec \HH^k \prec R' I$, $r I \prec \nabla^2F(u^{(k)}) $, and $\nabla^2F$ is Lipschitz continuous with constant $L_2$. Let $u_{nt}^{(k+1)} : = u^{(k)} + \Delta u_{nt}^{(k)}$, where $\Delta u_{nt}^{(k)}$ is the search direction by the proximal Newton method. 
Then
\[
\| u_{nt}^{(k+1)} - u^* \|_{\HH^k} \leq \frac{R \sqrt{R'}}{2r\sqrt{r'}} \| u^{(k)} - u^* \|^2_{\HH^k}.
\]
\end{lemma}
\begin{proof}
    \begin{align*}
        \| u_{nt}^{(k+1)} - u^* \|_{\HH^k} & \leq \sqrt{R'} \| u_{nt}^{(k+1)} - u^* \| \\
        & \leq \frac{\sqrt{R'}}{\sqrt{r}} \| u_{nt}^{(k+1)} - u^* \|_{\nabla^2 F(u^{(k)})} \\
        & \leq \frac{\sqrt{R'}}{\sqrt{r}} \frac{L_2}{2\sqrt{r}} \| u^{(k)} - u^* \|^2 \\
        & \leq \frac{L_2 \sqrt{R'}}{2r\sqrt{r'}} \| u^{(k)} - u^* \|^2_{\HH^k}.
    \end{align*}
    The third inequality comes from the  quadratic convergence of proximal Newton method (Theorem 3.4 in \cite{lee2014proximal}).
\end{proof}

{\bf Proof of Theorem~\ref{Thm:PNls2}:}
\begin{proof} 
Let $\Delta u^{(k)}_{nt}$ be the search direction generated by the proximal newton method and $\Delta u^{(k)}$ generated by Algorithm~\ref{alg:PNS}. Then we have
    \begin{align*}
        \| u^{(k+1)} - u^* \|_{\HH^k} & = \| u^{(k)} + t^l \Delta u^{(k)} - u^* \|_{\HH^k} \\
        & = \| u^{(k)} +  \Delta u^{(k)}_{nt} - u^* -\Delta u^{(k)}_{nt} + t^l \Delta u^{(k)} \|_{\HH^k} \\
        & = \| u^{(k)} +  \Delta u^{(k)}_{nt} - u^* -t^l \Delta u^{(k)}_{nt} + t^l \Delta u^{(k)} -(1-t^l) \Delta u^{(k)}_{nt}\|_{\HH^k} \\
        & \leq \| u^{(k)} +  \Delta u^{(k)}_{nt} - u^* \|_{\HH^k} + t^l \| \Delta u^{(k)} - \Delta u^{(k)}_{nt} \|_{\HH^k}+ (1-t^l) \|\Delta u^{(k)}_{nt} \|_{\HH^k} \\
        & \leq C \| u^{(k)} - u^* \|^2_{\HH^k} + t^l q \|\Delta u^{(k)}_{nt} \|_{\HH^k}+ (1-t^l) \|\Delta u^{(k)}_{nt} \|_{\HH^k} \\
        & = C \| u^{(k)} - u^* \|^2_{\HH^k} + (1-t^l+qt^l)\|\Delta u^{(k)}_{nt} \|_{\HH^k} \\
        & = C \| u^{(k)} - u^* \|^2_{\HH^k} + (1-t^l+qt^l)\|u^{(k)}+\Delta u^{(k)}_{nt} -u^{(k)} + u^*-u^* \|_{\HH^k} \\
        & \leq C \| u^{(k)} - u^* \|^2_{\HH^k} + (1-t^l+qt^l)(\| u^{(k+1)}_{nt} - u^* \|_{\HH^k} + \| u^{(k)} - u^* \|_{\HH^k}) \\
        & \leq C' \| u^{(k)} - u^* \|^2_{\HH^k} + (1-t^l+qt^l)\| u^{(k)} - u^* \|_{\HH^k}.
    \end{align*}
    Here $C = \frac{L_2 \sqrt{R'}}{2r \sqrt{r'}}$, $C' = \frac{L_2\sqrt{R'}}{2r \sqrt{r'}}(2-t^l+qt^l)$,
    and the second and third inequalities use Lemmas~\ref{apdx_lemma1} and~\ref{apdx_lemma2}.
\end{proof}

\bibliographystyle{abbrv}
\bibliography{VPFP.bib}

\end{document}